\documentclass[10pt,a4paper,oneside,reqno]{amsart}

\usepackage{setspace}
\usepackage[totalwidth=14.1cm,totalheight=22cm]{geometry}
\usepackage[T1]{fontenc}
\usepackage[dvips]{graphicx}
\usepackage{epsfig}
\usepackage{amsmath,amssymb,amscd,amsthm}
\usepackage{subfigure}
\usepackage[latin1]{inputenc}
\usepackage{latexsym}
\usepackage{graphics}
\usepackage{ae}
\usepackage{enumitem}
\usepackage[all]{xy}
\usepackage{scalerel}
\usepackage{cancel}
\usepackage{mathrsfs}
\usepackage[dvipsnames]{xcolor}
\usepackage[bookmarks=true, colorlinks=true, urlcolor=Blue,
linkcolor=BrickRed, citecolor=MidnightBlue, pagebackref=true]{hyperref}
\usepackage[normalem]{ulem}
\usepackage{nicefrac}
\usepackage{xfrac}

\makeatletter
\newtheorem*{rep@theorem}{\rep@title}
\newcommand{\newreptheorem}[2]{%
\newenvironment{rep#1}[1]{%
 \def\rep@title{#2 \ref{##1}}%
 \begin{rep@theorem}}%
 {\end{rep@theorem}}}
\makeatother

\makeatletter
\newtheorem*{rep@cor}{\rep@title}
\newcommand{\newrepcor}[2]{%
\newenvironment{rep#1}[1]{%
 \def\rep@title{#2 \ref{##1}}%
 \begin{rep@cor}}%
 {\end{rep@cor}}}
\makeatother

\makeatletter
\newtheorem*{rep@prop}{\rep@title}
\newcommand{\newrepprop}[2]{%
\newenvironment{rep#1}[1]{%
 \def\rep@title{#2 \ref{##1}}%
 \begin{rep@prop}}%
 {\end{rep@prop}}}
\makeatother

\newtheorem{cor}{Corollary}[section]

\newtheorem{theorem}[cor]{Theorem}

\newtheorem{prop}[cor]{Proposition}

\newrepcor{cor}{Corollary}
\newreptheorem{theorem}{Theorem}
\newrepprop{prop}{Proposition}

\newtheorem{lemma}[cor]{Lemma}

\theoremstyle{definition}
\newtheorem{defi}[cor]{Definition}
\theoremstyle{remark}
\newtheorem{remark}[cor]{Remark}
\newtheorem*{remark*}{Remark}
\newtheorem{example}[cor]{Example}

\newtheorem*{notation*}{Notation}

\newlist{steps}{enumerate}{1}
\setlist[steps, 1]{itemsep=8pt,leftmargin=0cm,itemindent=.5cm,labelwidth=\itemindent,labelsep=0cm,align=left,label = \textbf{\emph{Step \arabic*}:\,}}

\newcommand{\C}{{\mathbb C}}
\newcommand{\R}{{\mathbb R}}
\newcommand{\Z}{{\mathbb Z}}

\newcommand{\Hyp}{\mathbb{H}}

\newcommand{\AdS}{\mathbb{A}\mathrm{d}\mathbb{S}}
\newcommand{\HP}{\mathsf{HP}}

\newcommand{\RP}{\mathbb R\mathrm{P}}

\newcommand{\PSL}{\mathrm{PSL}}
\newcommand{\PGL}{\mathrm{PGL}}
\newcommand{\GL}{\mathrm{GL}}
\newcommand{\SO}{\mathrm{SO}}

\newcommand{\so}{\mathfrak{so}}

\newcommand{\ddt}{\left.\frac{d}{dt}\right|_{t=0}}
\newcommand{\Isom}{\mathrm{Isom}}

\renewcommand{\O}{\mathrm{O}}

\newcommand{\perpp}{{\perp_{1}}}
\newcommand{\perpm}{{\perp_{-1}}}

\newcommand{\p}[1]{\ensuremath{\boldsymbol{#1^+} }}
\newcommand{\m}[1]{\ensuremath{\boldsymbol{#1^-} } }
\renewcommand{\l}[1]{\ensuremath{\boldsymbol{#1}} }

\newcommand{\pd}[1]{\ensuremath{\boldsymbol{\dot{#1}^+} } }
\newcommand{\md}[1]{\ensuremath{\boldsymbol{\dot{#1}^-} } }
\newcommand{\ld}[1]{\ensuremath{\boldsymbol{\dot{#1} } } }

\begin{document}

\setcounter{secnumdepth}{2}
\setcounter{tocdepth}{2}

\title[Character varieties of a transitioning Coxeter 4-orbifold]{Character varieties of a\\ transitioning Coxeter 4-orbifold}

\author[Stefano Riolo]{Stefano Riolo}
\address{Stefano Riolo: Dipartimento di Matematica, Universit\`a di Pisa \newline Largo Bruno Pontecorvo 5\\ 56127 Pisa\\ Italy}
\email{stefano.riolo@dm.unipi.it}

\author[Andrea Seppi]{Andrea Seppi}
\address{Andrea Seppi: Institut Fourier, UMR 5582, Laboratoire de Math\'ematiques,
Universit\'e Grenoble Alpes, CS 40700, 38058 Grenoble cedex 9, France.} \email{andrea.seppi@univ-grenoble-alpes.fr}

\thanks{The authors were partially supported by FIRB 2010 project ``Low dimensional geometry and topology'' (RBFR10GHHH003), by the 2020 Germaine de Sta\"el project ``Deformations of geometric structures on higher-dimensional manifolds'', and are members of the national research group GNSAGA.
The first author was supported by the Mathematics Department of the University of Pisa (research fellowship ``Deformazioni di strutture iperboliche in dimensione quattro''), and by the Swiss National Science Foundation (project no.~PP00P2-170560).}

\begin{abstract}
In 2010, Kerckhoff and Storm discovered a path of hyperbolic 4-polytopes eventually collapsing to an ideal right-angled cuboctahedron. This is expressed by a deformation of the inclusion of a discrete reflection group (a right-angled Coxeter group) in the isometry group of hyperbolic 4-space. More recently, we have shown that the path of polytopes can be extended to Anti-de Sitter geometry so as to have geometric transition on a naturally associated 4-orbifold, via a transitional half-pipe structure.

In this paper, we study the hyperbolic, Anti-de Sitter, and half-pipe character varieties of Kerckhoff and Storm's right-angled Coxeter group near each of the found holonomy representations, including a description of the singularity that appears at the collapse. An essential tool is the study of some rigidity properties of right-angled cusp groups in dimension four.
\end{abstract}

\maketitle

%

\section{Introduction}

In the Seventies, Thurston \cite{thurstonnotes} introduced the notion \emph{degeneration} of $(G,X)$-struc\-tures, later widely studied and used in dimension three \cite{Hthesis,P98,CHK,P01,P02,BLP,MR2140265,P07,P13,Kozai_thesis,LM2,LM1,Kozai}. Typical instances of this phenomenon are paths of hyperbolic cone structures on a 3-manifold eventually collapsing to some lower-dimensional orbifold, whose geometric structure is said to \emph{regenerate} to 3-dimensional hyperbolic structures.

In his thesis \cite{danciger}, Danciger showed that when the limit is 2-dimensional and hyperbolic, it often regenerates to Anti-de Sitter (AdS) structures as well, so as to have \emph{geometric transition} from hyperbolic to AdS structures (see also \cite{dancigertransition,dancigerideal,AP,surveyseppifillastre,trettel_thesis}). To that purpose, he introduced half-pipe (HP) geometry, which is a \emph{limit geometry} \cite{CDW} of both hyperbolic and AdS geometries inside projective geometry, and encodes the behaviour of such a collapse ``at the first order''. One can indeed suitably ``rescale'' the structures inside the ``ambient'' projective geometry along the direction of collapse, so as to get at the limit a 3-dimensional ``transitional'' HP structure. 

Concerning dimension four, Kerckhoff and Storm \cite{KS} described a path $t \mapsto \mathcal{P}_t$, $t \in (0, 1]$, of hyperbolic 4-polytopes which collapse as $t \to 0$ to a 3-dimensional ideal right-angled cuboctahedron. This induces a path of incomplete hyperbolic structures on a naturally associated 4-orbifold $\mathcal{O}$. The orbifold fundamental group of $\mathcal{O}$ is a rank-22 right-angled Coxeter group $\Gamma_{22}$, which embeds in $\Isom(\Hyp^4)$ as a discrete reflection group when $t=1$.
In \cite{transition_4-manifold} (see also \cite{TSG}), we found a similar path of AdS 4-polytopes such that the two paths, suitably rescaled, can be joined so as to give geometric transition on the orbifold $\mathcal{O}$. In particular, there is a transitional HP orbifold structure on $\mathcal{O}$ joining the two paths. 

Keckhoff and Storm's deformation has been studied and used in \cite{MR} to show, among other things, the first examples of collapse of 4-dimensional hyperbolic cone structures to 3-dimensional ones. Similarly, thanks to the found AdS deformation and HP transitional structure, in \cite{transition_4-manifold} the authors provided the first examples of geometric transition from hyperbolic to AdS cone structures in dimension four.

The goal of this paper is to describe the hyperbolic, AdS, and HP character varieties of the right-angled Coxeter group $\Gamma_{22}$, including a study of the behaviour at the collapse. The results are summarised in Theorem \ref{teo:main} below.

\subsection{The three character varieties of $\Gamma_{22}$}

Let $G$ be $\Isom(\Hyp^4)$, $\Isom(\AdS^4)$, or the group $G_{\HP^4}$ of transformations of half-pipe geometry, and let $G^+<G$ be the subgroup of orientation-preserving transformations. 

Recall that $\mathrm{Hom}(\Gamma_{22},G)$ 
is naturally a real algebraic affine set \cite{Weil}.
We call \emph{character variety} of $\Gamma_{22}$ the {(topological) quotient
$$X(\Gamma_{22},G)=\mathrm{Hom}(\Gamma_{22},G)/G^+$$
by the action of $G^+$ by conjugation. When $G$ is reductive, that is in the hyperbolic and AdS settings, it is also possible to define the GIT quotient, which has a structure of real semialgebraic set by general results \cite{GIT}. We will come back to this point of view at the end of the subsection.}

The holonomy representations of the geometric structures on the orbifold $\mathcal{O}$ constructed in \cite{KS,transition_4-manifold} provide a smooth path $t\mapsto{[\rho^G_t]}$ in $X(\Gamma_{22},G)$. 
This path was originally defined in \cite{KS} when $G=\Isom(\Hyp^4)$ only for $t\in(0,1]$, and is easily continued analytically also for non-positive times. The Anti-de Sitter path, introduced in \cite{transition_4-manifold}, is only defined for $t\in(-1,1)$ and diverges as $|t|\to 1^-$, while for $G=G_{\HP^4}$ there is a ``trivial'' path of non-equivalent HP representations (defined for $t\in\R$, and diverging as $|t|\to+\infty$) differing from one another by ``stretching'' in the ambient real projective space (see below {in Section \ref{sec:intro2}}).

The representations obtained at $t=0$ correspond geometrically to a ``collapse'' and play a special role in two ways. First, they correspond to a ``symmetry'' in the character varieties, since the representations $\rho_t^G$ and $\rho_{-t}^G$ 
are conjugated in $G$ but not in $G^+$. Second, interpreting $\Isom(\Hyp^4)$, $\Isom(\AdS^4)$ and $G_{\HP^4}$ as subgroups of $\PGL(5,\R)$, the three representations $\rho_0^G$ coincide. They correspond to a representation (we omit the superscript $G$ here)
$$\rho_0\colon\Gamma_{22}\to\mathrm{Stab}(\Hyp^3)<G~,$$
for a fixed copy of $\Hyp^3$ in $\Hyp^4$, $\AdS^4$, or $\HP^4$, respectively. Projecting the image of $\rho_0$ in $\mathrm{Stab}(\Hyp^3)\cong\Isom(\Hyp^3)\times\Z/2\Z$ to $\Isom(\Hyp^3)$ gives the reflection group of an ideal right-angled cuboctahedron (see Figure \ref{fig:cuboct}).




Our main result is resumed as follows. (A precise explaination of the terminology is given below the statement; see Figure \ref{fig:rep_var} for a schematic picture.)
%

\begin{theorem} \label{teo:main}
Let $G$ be $\Isom(\Hyp^4)$, $\Isom(\AdS^4)$, or $G_{\HP^4}$. A neighbourhood $\mathcal U$ of $[\rho_0]$ in $X(\Gamma_{22},{G})$ consists of two smooth, transverse, components $\mathcal V$ and $\mathcal H$ satisfying $\mathcal{V} \cap \mathcal{H} = \{ [\rho_0] \}$: 
\begin{itemize}
\item the curve $\mathcal V$ of the conjugacy classes of all the holonomy representations $\rho_t^G$;
\item a $12$-dimensional ball $\mathcal H$, identified to a neighbourhood of the complete hyperbolic orbifold structure of the ideal right-angled cuboctahedron in its deformation space. 
\end{itemize}
The group $G/G^+\cong\Z/2\Z$ acts on $\mathcal U$ fixing $\mathcal{H}$ point-wise and sending $[\rho_t^G]$ to $[\rho_{-t}^G]$.
\end{theorem}

\begin{figure}
\includegraphics[scale=.33]{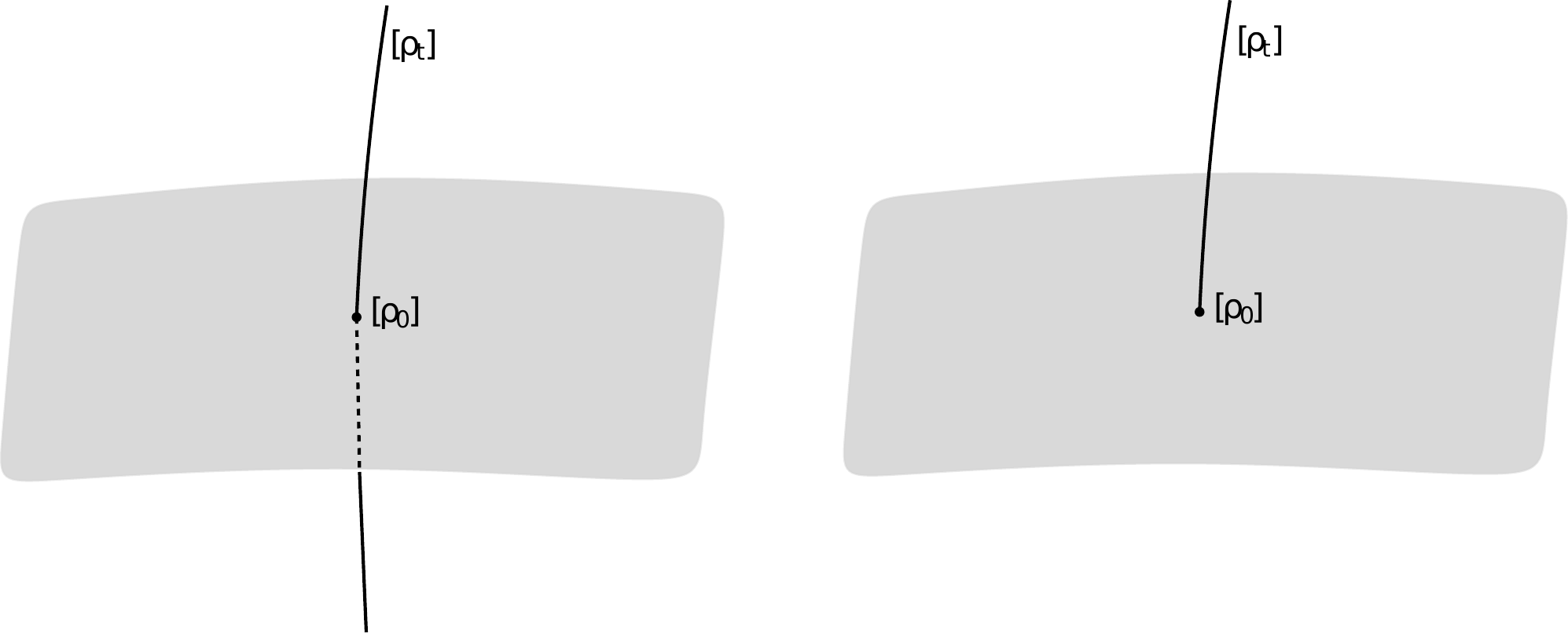}
\caption[The character variety near the collapse.]{\footnotesize On the left, a topological picture of $X(\Gamma_{22},G)$ near the collapse, which corresponds to the point $[\rho_0]$. The vertical component $\mathcal V$ is 
{$\{[\rho^G_t]\}_t$}. The horizontal component $\mathcal H$ is 12-dimensional and corresponds to the deformations of the complete hyperbolic structure of the ideal right-angled cuboctahedron. On the right, the corresponding neighbourhood in 
$\mathrm{Hom}(\Gamma_{22},G)/G$, i.e. in the further quotient of $X(\Gamma_{22},G)$ by $G/G^+\cong\Z/2\Z$.} \label{fig:rep_var}
\end{figure} 


Let us include some comments to elucidate the content of Theorem \ref{teo:main}. First, our proofs actually show that the representation $\rho_0$ has a neighbourhood in $\mathrm{Hom}(\Gamma_{22},G)$ that is homeomorphic to $(\mathcal H\cup\mathcal V)\times G^+$, in such a way that the action of $G^+$ corresponds to obvious left multiplication by $G^+$ on the second factor {(see Remark \ref{rmk:local product})}.

Let $\widetilde{\mathcal{U}}$, $\widetilde{\mathcal{V}}$ and $\widetilde{\mathcal{H}}$ be the preimages in $\mathrm{Hom}(\Gamma_{22},G)$ of $\mathcal{U}$, $\mathcal{V}$ and $\mathcal{H}$, respectively. By ``smoothness'' of the ``components'' $\mathcal{V}$ and $\mathcal{H}$ of $\mathcal{U}$ we actually refer to $\widetilde{\mathcal{V}}$, $\widetilde{\mathcal{H}}$ and $\widetilde{\mathcal{U}}$, 
respectively. In particular, $\widetilde{\mathcal{V}}$ and $\widetilde{\mathcal{H}}$ are smooth 
manifolds (of dimension 11 and 22, respectively). {The smoothness of $\widetilde{\mathcal{V}}$ and $\widetilde{\mathcal{H}}$ together with the local product structure in a neighbourhood of $\rho_0$ induce a smooth structure on the components $\mathcal{H}$ and $\mathcal{V}$ in the quotient.}

The ``transversality'' of 
$\mathcal V$ and $\mathcal H$ is defined 
as follows: $\widetilde{\mathcal{V}} \cap \widetilde{\mathcal{H}}$ is the $G$-orbit of $\rho_0$, and 
the Zariski tangent spaces of 
$\widetilde{\mathcal{V}}$ and $\widetilde{\mathcal{H}}$ intersect transversely in the Zariski tangent space of $\mathrm{Hom}(\Gamma_{22},G)$ at $\rho_0$ (and hence at any other point of its orbit). {In particular, every infinitesimal deformation tangent to both $\widetilde{\mathcal{V}}$ and $\widetilde{\mathcal{H}}$ is tangent to the $G^+$-orbit of $\rho_0$.} (See Section \ref{sec:intro_alg} below and Remark \ref{rem:transverse} for more details.) 

{Our analysis will also show that, when $G$ is $\Isom(\Hyp^4)$ or $\Isom(\AdS^4)$, the character variety $X(\Gamma_{22},G)$ is homeomorphic to the GIT quotient $\mathrm{Hom}(\Gamma_{22},G)/\!\!/G^+$ near each $[\rho_t]$ (see Remark \ref{remark quotient hausdorff}). In other words, $X(\Gamma_{22},G)$ is Hausdorff near $[\rho_t]$. Moreover, the natural smooth structure of each component is coherent with the real semialgebraic structure of the GIT quotient (see also Remark \ref{rem:transverse}).}

In order to further discuss our results, we first need to describe the three paths of geometric representations of Theorem \ref{teo:main}.

\subsection{The three deformations} \label{sec:intro2}

For $t=1$, the hyperbolic polytope ${\mathcal P}_1 \subset \Hyp^4$ is obtained in \cite{KS} from the ideal right-angled 24-cell by removing two opposite bounding hyperplanes. So $\mathcal{P}_1$ has two ``Fuchsian ends'', and in particular its volume is infinite. The reflection group
$$\Gamma_{22}<\mathrm{Isom}(\Hyp^4)$$
associated to $\mathcal P_1$ is thus a right-angled Coxeter group obtained by removing from the reflection group $\Gamma_{24}$ of the ideal right-angled 24-cell two generators (reflections at two opposite facets).

As a sort of ``reflective hyperbolic Dehn filling'', Kerckhoff and Storm show that the inclusion $\Gamma_{22} \to \mathrm{Isom}(\Hyp^4)$ is not locally rigid. This is done by moving the bounding hyperplanes of $\mathcal P_1$ in such a way that the orthogonality conditions given by the relations of $\Gamma_{22}$ are maintained, and thus obtaining a path $\rho^{\Hyp^4}_t$ of geometric representations of $\Gamma_{22}$.

As $t$ decreases from 1, the combinatorics of $\mathcal{P}_t$ changes a few times, until the volume of $\mathcal{P}_t$ becomes finite. Most of the dihedral angles of $\mathcal{P}_t$ are constantly right, while the varying ones are all equal and tend to $\pi$ as $t \to 0$, when $\mathcal{P}_t$ collapses to the cuboctahedron.
As an abstract group, $\Gamma_{22}$ can be identified to the orbifold fundamental group of an orbifold $\mathcal{O}$ supported on the complement in $\mathcal P_t$ of the ridges with non-constant dihedral angle.

Kerckhoff and Storm show moreover that the space of conjugacy classes of representations $\Gamma_{22}\to\mathrm{Isom}(\Hyp^4)$ deforming the inclusion is a smooth curve outside of the collapse. In other words, for $t \neq 0$ the only non-trivial deformation (up to conjugacy) is given by the found holonomies $\rho^{\Hyp^4}_t$.

In \cite{transition_4-manifold}, we produced a path of AdS 4-polytopes with the same combinatorics of the hyperbolic polytope $\mathcal{P}_t$ for $t \in (0,\varepsilon)$, such that the same orthogonality conditions between the bounding hyperplanes are satisfied, and again collapsing to an ideal right-angled cuboctahedron in a spacelike hyperplane $\Hyp^3$ of $\AdS^4$. Some bounding hyperplanes are spacelike, and some others are timelike. We have in particular a path of AdS orbifold structures on $\mathcal{O}$, with holonomy representation $\rho^{\AdS^4}_t \colon \Gamma_{22} \to \Isom(\AdS^4)$ given by sending each generator to the corresponding AdS reflection.

We moreover find in \cite{transition_4-manifold} a one-parameter family of transitional HP structures on $\mathcal{O}$, with holonomy $\rho_t^{\HP^4}$. To interpret distinct elements in this family, recall that in half-pipe space there is a preferred direction under which the HP metric is degenerate. An HP structure is never rigid, because one can always conjugate with a transformation which ``stretches'' the degenerate direction, and obtain a new structure equivalent to the initial one \emph{as a real projective structure}, but inequivalent \emph{as a half-pipe structure}. We discover here (this is part of the content of Theorem \ref{teo:main}) that such stretchings are the only possible deformations, so that the found HP structures are essentially unique. 

Finally, we remark that the geometric transition described in \cite{transition_4-manifold} induces a continuous deformation connecting in the $\PGL(5,\R)$-character variety ``half'' of the path in $\mathcal V\subset X(\Gamma_{22},\Isom(\Hyp^4))$ (which is exactly the path of hyperbolic representations exhibited by Kerckhoff and Storm, for $t\in(0,1]$) and ``half'' of the analogous path in $X(\Gamma_{22},\Isom(\AdS^4))$, going through a single half-pipe representations $\rho_{t_0}^{\HP^4}$ with $t_0\neq 0$ (this value of $t_0$ can be chosen arbitrarily, up to reparameterising the entire deformation).

\subsection{About the result}

Theorem \ref{teo:main} contains several novelties. First, while the smoothness of the $\Isom(\Hyp^4)$-character variety for $t>0$ was proved in \cite{KS}, the smoothness on the AdS and HP sides is completely new. Second, the study of the character variety at the ``collapsed'' point $[\rho_0]$ is a new result in all three settings. Some motivations follow.

First of all, we found worthwhile analysing the behaviour of the deformation space of the AdS orbifold $\mathcal O$ --- equivalently, the $\Isom(\AdS^4)$-character variety of $\Gamma_{22}$ near $[\rho_t]$ --- and compare it with the hyperbolic counterpart. In fact, the literature seems to miss a study of deformations of AdS polytopes in this spirit. With respect to hyperbolic geometry, one may expect more flexibility in AdS geometry, but we find that the behaviour on the AdS side is the same as the hyperbolic counterpart. (In regard, see however \cite[Question 9.3]{questionsads} and the related discussion.)

Regarding the collapse, in \cite[Section 14]{KS} Kerckhoff and Storm mention that the family of hyperbolic polytopes $\mathcal P_t\subset\Hyp^4$, $t>0$, is expected to have interesting geometric limits by rescaling in the direction transverse to the collapse. On the other hand, they assert that ``providing the details of this geometric construction would require more space than perhaps is merited here''.

The work \cite{transition_4-manifold} provides a complete description of such a geometric limit in half-pipe geometry, which, after the work of Danciger, seems the best suited in order to analyse this kind of collapse. One can in fact consider the limiting HP structure as an object which encodes the collapse \emph{at the first order}, essentially keeping track of the derivatives of all the associated geometric quantities. Thus, if \cite{transition_4-manifold} describes the collapse \emph{at level of geometric structures}, Theorem \ref{teo:main} (in the hyperbolic and AdS setting) gives a precise description of the collapse \emph{at level of the character variety}.

Finally, our study of the HP character variety shows that the only deformations of the found half-pipe orbifold structure are obtained by ``stretching'' in the degenerate direction. The presence of many commutation relations forces the rigidity of the HP structures. 

Together with the hyperbolic and AdS picture, this shows that ``nearby''  there is no collapsing path of hyperbolic or AdS orbifold structures  other than the ones we found (up to reparameterisation). This should be compared with some 3-dimensional examples found by Danciger \cite[Section 6]{dancigertransition}, where the transitional HP structure deforms non-trivially to nearby HP structures that regenerate to non-equivalent AdS structures, despite not regenerating to hyperbolic structures. 

All in all, Theorems \ref{teo:main} exhibits a strong lack of flexibility around this example. Its proof, explained in the next section, suggests that this could be more generally due to dimension issues, confirming the usual feeling that ``the rigidity increases with the dimension''.

An overview of the ideas behind the proof of Theorem \ref{teo:main} follows. 
We start in Section \ref{sec:intro_geo} with 
the geometric tools, which lead to the topological description of 
the neighbourhood {$\mathcal{U}$} of $[\rho_0]$
. In 
Section \ref{sec:intro_alg}, we then outline the algebraic aspects, which lead to the description of the Zariski tangent space and the transversality statement. 

\subsection{Cusp rigidity in dimension four} \label{sec:intro_geo}

The holonomy representations $\rho_t$ have the property that each generator in the standard presentation of $\Gamma_{22}$ is sent by $\rho_t$ to a (hyperbolic or AdS) reflection, and this property is preserved by small deformations. As in \cite{KS}, we thus reduce to studying the configurations of hyperplanes of reflections satisfying certain orthogonality conditions. Once this set-up is established, there are two main facts to prove: the smoothness of the character variety outside the collapse, and the description of the collapse itself.

For the first fact, the proof on the AdS side follows the general lines of the proof given in \cite{KS} for the hyperbolic case. However, different arguments are required for one point of fundamental importance concerning a property of rigidity of cusp representations in dimension four. 

In fact, in \cite{KS} a preliminary lemma is proved, which can be summarised by saying that in dimension 4 ``cusp groups stay cusp groups''. More precisely, if we consider the orbifold fundamental group of a Euclidean cube $\Gamma_{\mathrm{cube}}$, this property states that any  representation of $\Gamma_{\mathrm{cube}}$ into $\Isom(\Hyp^4)$ sending the six standard generators to reflections in six distinct hyperplanes 
{with the property of being asymptotic to a common point at infinity} (a ``cusp group'') can only be deformed by preserving 
{this property}. Note that the analogue fact is false in dimension three, where the situation is more flexible.

We do prove the analogous property for Anti-de Sitter geometry in dimension 4 (Proposition \ref{prop cube group ads}), where a cusp group is defined analogously. There are however remarkable differences due to the different nature of hyperbolic and AdS geometries, for instance a cusp group in $\AdS^4$ will be generated by 4 reflections in timelike hyperplanes and 2 reflections in spacelike hyperplanes. The proof of this rigidity property in AdS uses therefore ad hoc arguments and is somehow more surprising than its hyperbolic counterpart, as in general a little more flexibility might be expected for AdS geometry. 

Once this fundamental property is established, the proof of the smoothness of the curve is based on a careful analysis of the structure of the group $\Gamma_{22}$ and the possible deformations of the polytope $\mathcal P_t$, relying on the application of the above rigidity property to each peripheral subgroup (there is a cusp group $\Gamma_{\mathrm{cube}}<G$ associated to each ideal vertex of $\mathcal P_t$). The methods are rather elementary, although some intricate computation is necessary, and the general strategy is similar to that of the hyperbolic analogue provided in \cite{KS}.

Let us now explain our arguments to analyse the collapse in both the $\Hyp^4$ and $\AdS^4$ character variety. The proof is essentially the same for both cases, so let us focus on the hyperbolic case (that is, $G=\Isom(\Hyp^4)$) in this introduction for definiteness.

It is not difficult to describe the two components $\mathcal V$ and $\mathcal H$ of the neighbourhood $\mathcal U$ from a geometric point of view: the ``vertical'' curve $\mathcal V$ consists of the conjugacy classes of the holonomy representations $\rho^G_t$ of $\mathcal O$, $t>0$, plus the natural extension of the path for $t<0$ given by $r\circ \rho^{G}_{-t}\circ r$. Here $r$ is the reflection in the totally geodesic copy of $\Hyp^3$ to which the polytope collapses as $t=0$. On the other hand, the ``horizontal'' 12-dimensional component $\mathcal H$ consists of representations which fix setwise this copy of $\Hyp^3$, and deform the reflection group of the ideal right-angled cuboctahedron in $\Hyp^3$.

One then has to show that there exists a neighborhood of $[\rho_0]$ such that every point in this neighborhood belongs to one of these two components --- namely, there are no other conjugacy classes of representations nearby $[\rho_0]$. To prove this, we refine the study of the rigidity properties of the cusps. We introduce a notion of \emph{collapsed cusp group}: a representation of $\Gamma_{\mathrm{cube}}$ defined similarly to cusp groups, but allowing that two generators are sent to reflections in the same hyperplane. The restriction of $\rho_0$ to each peripheral subgroup is in fact a collapsed cusp group. Then we prove a more general version of the aforementioned rigidity property ``cusp groups stay cusp groups'', by showing that ``collapsed cusp groups either stay collapsed, or deform to cusp groups''. More precisely, representations nearby a collapsed cusp group either keep the property that two opposite generators are sent to the same reflection, or they become cusp groups in the usual sense.

By an analysis of the character variety in the spirit of Kerckhoff and Storm, we show that the ``vertical'' curve $\mathcal V$ is smooth also at $t=0$ if we impose that the 
{asymptoticity conditions} are preserved. Applying the  more general property of rigidity which includes the ``collapsed'' case is then the fundamental step to conclude the proof.

Concerning the proof for the half-pipe case, it follows a similar line, but many steps are dramatically simpler. The key point is again a rigidity property for four-dimensional (collapsed) cusp groups, which is showed rather easily by using the isomorphism between $G_{\HP^4}$ and the group of isometries of Minkowski space $\R^{1,3}$, which is a semidirect product $\O(1,3)\ltimes \R^{1,3}$. The proof then parallels the steps for the hyperbolic and AdS case, except that the smoothness of the vertical component $\mathcal V$ is granted by the fact that --- thanks to this semidirect product structure of $G_{\HP^4}$ --- $\mathcal V$ identifies with the first cohomology vector space $H^1_{\rho_0}(\Gamma_{22},\R^{1,3})$. The proof that this vector space is 1-dimensional (see \eqref{eq: cohomology Gamma22} below) requires a certain amount of technicality and relies on a precise study of the group-theoretical structure of $\Gamma_{22}$.

\subsection{The Zariski tangent space and the  first cohomology group} \label{sec:intro_alg}

Let $\mathfrak g$ be the Lie algebra of $G$, and $\mathrm{Ad}\colon G\to\mathrm{Aut}(\mathfrak g)$ be the adjoint representation. 
The Zariski tangent space to $\mathrm{Hom}(\Gamma_{22},G)$ at $\rho$ is naturally identified to the space $Z^1_{\mathrm{Ad}\,\rho}(\Gamma_{22},\mathfrak{g})$ of cocycles with coefficients twisted by $\mathrm{Ad} \circ \rho$. Roughly speaking, the cohomology group $H^1_{\mathrm{Ad}\,\rho}(\Gamma_{22},\mathfrak{g})$  plays the same role for the 
quotient $X(\Gamma,G)$ at $[\rho]$. Indeed the coboundaries $B^1_{\mathrm{Ad}\,\rho}(\Gamma_{22},\mathfrak{g})$ are precisely the infinitesimal deformations tangent to the orbit. See \cite{Weil,JM} for more details.

Since $\rho_0$ preserves a totally geodesic copy of $\Hyp^3$, we have a natural decomposition:
\begin{equation} \label{eq:decomp}
H^1_{\mathrm{Ad}\,\rho_0}(\Gamma_{22},\mathfrak g)\cong H^1_{\mathrm{Ad}\,\rho_0}(\Gamma_{22},\mathfrak o(1,3))\oplus H^1_{\rho_0}(\Gamma_{22},\R^{1,3})~.
\end{equation}

%
We show that the vector space $H^1_{\mathrm{Ad}\,\rho_0}(\Gamma_{22},\mathfrak{g})$ is 13-dimensional. In the decomposition \eqref{eq:decomp}, the first factor is 12-dimensional and ``tangent'' to $\mathcal{H}$, while the second factor is 1-dimensional and ``tangent'' to $\mathcal{V}$; moreover integrable vectors are precisely those lying in one of these two factors. This statement is made more precise by looking at the representation variety $\mathrm{Hom}(\Gamma_{22},G)$, where 
the tangent spaces of the smooth varieties $\widetilde{\mathcal{V}}$ and $\widetilde{\mathcal{H}}$ are generated by the preimages in $Z^1_{\mathrm{Ad}\,\rho_0}(\Gamma_{22},\mathfrak{g})$ of the two factors in the decomposition \eqref{eq:decomp}. Hence the 
intersection $\widetilde{\mathcal{V}} \cap \widetilde{\mathcal{H}}$ is transverse and consists precisely of the orbit of $\rho_0$.

As mentioned above,  we prove (in Proposition \ref{prop:first cohomology group}) that the second factor in the decomposition \eqref{eq:decomp} has dimension 1, namely 
\begin{equation} \label{eq: cohomology Gamma22}
H^1_{\rho_0}(\Gamma_{22},\R^{1,3})\cong \R~.
\end{equation}
The non-trivial elements in this vector space are obtained geometrically from the half-pipe holonomy representations  that we constructed in \cite{transition_4-manifold}, and are easily shown to be ``first-order deformations'' of paths in $\mathcal V$. On the other hand, the first factor in the decomposition \eqref{eq:decomp} is 12-dimensional by general reasons, namely by an orbifold version of hyperbolic Dehn filling {(note that the cuboctahedron has exactly 12 vertices)}. Its elements are again integrable and tangent to deformations in $\mathcal H$. This cohomological computation is the main algebraic step in the proof of Theorem  \ref{teo:main}. 

As another noteworthy comment on the consequences of \eqref{eq: cohomology Gamma22}, recall Danciger's result \cite[Theorem 1.2]{dancigertransition}: the existence of geometric transition on a compact half-pipe 3-manifold $\mathcal X$, with singular locus a knot $\Sigma$, is proved under the sole condition 
\begin{equation} \label{eq:1-dimensionality-n=3}
H^1_{\mathrm{Ad}\,\rho_0}(\pi_1(\mathcal X\smallsetminus\Sigma), \mathfrak{so}(1,2))\cong\R~.
\end{equation}

Now, for any representation $\rho\colon\Gamma\to\O(1,2)$ with $\Gamma$ finitely generated there is a natural identification
$$H^1_{\mathrm{Ad}\,\rho}(\Gamma,\mathfrak{so}(1,2))\cong H^1_{\rho}(\Gamma,\R^{1,2})~.$$
In presence of a collapse of hyperbolic or AdS structures of dimension $n$, the holonomy representations of a rescaled HP structure naturally provide elements of the cohomology group $H^1_{\rho_0}(\pi_1(\mathcal X\smallsetminus\Sigma),\R^{1,n-1})$. In particular, the correct generalisation of Danciger's condition \eqref{eq:1-dimensionality-n=3} to any dimension $n$ would be:
\begin{equation} \label{eq:1-dimensionality}
H^1_{\rho_0}(\pi_1(\mathcal X\smallsetminus\Sigma),\R^{1,n-1})\cong \R~,
\end{equation}
in agreement with what we found for $\Gamma_{22}$ (the orbifold fundamental group of $\mathcal O$) --- compare with \eqref{eq: cohomology Gamma22}. 
In conclusion, this suggests that a higher-dimensional regeneration result in the spirit of Danciger, although far from reach at the present time, might be reasonable. 

\subsection*{Organisation of the paper}

In Section \ref{sec:prelim}, we establish the set-up for the proof of Theorem \ref{teo:main} in the hyperbolic and AdS cases. In Section \ref{sec:cusp_rigidity}, which may be of independent interest, we study deformations of some right-angled Coxeter groups represented as ``cusp groups'' in $\Isom(\Hyp^n)$ {and} $\Isom(\AdS^n)$ 
for $n=3,4$. In Section \ref{sec:char_var}, we study the $\Isom(\Hyp^n)$ and $\Isom(\AdS^n)$ character varieties of $\Gamma_{22}$, concluding the first part of the proof of Theorem \ref{teo:main} for the hyperbolic and AdS case. Section \ref{sec:cusp_groups_HP} is the analogue of Sections \ref{sec:prelim} and \ref{sec:cusp_rigidity} for half-pipe geometry. In Section \ref{sec:gp cohomology} we provide the explicit computation of \eqref{eq: cohomology Gamma22}, and use it for the first part of the proof of Theorem \ref{teo:main} in the HP case. In Section \ref{sec:extra} we conclude the proof of Theorem \ref{teo:main} {in all the three cases}, relying on algebraic computations of the first cohomology group.

\subsection*{Acknowledgments}

We are grateful to Francesco Bonsante and Bruno Martelli for interesting discussions, useful advices, and encouragement. We also thank Jeffrey Danciger and Gye-Seon Lee for interest in this work and related discussions
, and Cl\'ement Gu\'erin and Andrea Maffei for ad-hoc explanations about character varieties. We are grateful to an anonymous referee whose suggestions highly improved the exposition and results of this paper.

We thank the mathematics departments of Pavia, Luxembourg and Neuch\^atel, for the warm hospitality during the respective visits while part of this work was done. The stage of this collaboration was set during the workshop ``Moduli spaces'', held in Ventotene in September 2017: we are grateful to the organisers for this opportunity.

\section{Reflections in hyperbolic and AdS geometry} \label{sec:prelim}

In this section we establish the set-up for the study of hyperbolic and AdS character varieties of right-angled Coxeter groups. 
 
\subsection{Hyperbolic and AdS geometry}\label{subsec:defi hyp ads}

We begin with the necessary definitions and notation. Let $q_{\pm1}$ be the quadratic form on $\R^{n+1}$ of signature $(-, +, \ldots, +, \pm)$ defined by:
$$q_{\pm1}(x) = -x_0^2 + x_1^2 + \ldots+x_{n-1}^2 \pm x_n^2~,$$
and let $b_{\pm1}$ be the associated bilinear form. 

The $n$-dimensional \emph{hyperbolic space} $\Hyp^n$ is defined {via the ``hyperboloid model'' as:
$$\Hyp^n=\{x\in\R^{n+1}\,|\,q_1(x) 
= -1\, , \,x_0>0\}~.$$
The restriction of $q_1$ endows $\Hyp^n$} with a complete Riemannian metric of constant negative sectional curvature, whose isometry group $\Isom(\Hyp^n)$ is identified to {an index-two subgroup of $\O(1,n)$, namely the subgroup of those linear isometries of $q_1$ which preserve $\Hyp^n$. Despite this subgroup is defined by an inequality, if $n$ is even $\Isom(\Hyp^n)$ is naturally isomorphic to the algebraic Lie group $\SO(1,n)$ (via $A \mapsto A$ if $A$ preserves $\Hyp^n \subset \R^{n+1}$, and $A \mapsto -A$ otherwise).} The \emph{boundary at infinity} of $\Hyp^n$
is the projectivisation of the cone of null directions  for the quadratic form $q_1$:
$$\partial\Hyp^n=\{x\in\R^{n+1}\,|\,q_1(x)=0\}/\R^*~.$$
{Hence both $\Hyp^n$ and $\partial\Hyp^n$ can be seen as subsets of $\RP^n$. The topology of $\RP^n$ induces a natural topology on $\Hyp^n \cup \partial \Hyp^n$, which makes it homeomorphic to the closed $n$-ball $D^n$. Finally, given a subset $A \subset \Hyp^n$ that is closed as a subspace of $\Hyp^n$, its \emph{ideal closure} $\overline{A}$ is the closure of $A$ in 
$\RP^n$. In particular, we have $\overline{\Hyp}{}^n = \Hyp^n \cup \partial \Hyp^n$.}

The $n$-dimensional \emph{Anti-de Sitter space} is defined as:
{$$\AdS^n=\{x\in\R^{n+1}\,|\,q_{-1}(x)
= -1 \}~,$$
and the restriction of $q_{-1}$ endows $\AdS^n$} with a Lorentzian metric of constant negative sectional curvature. {Observe that $\AdS^n$ is 
homeomorphic to $S^1 \times \R^{n-1}$}. 
{Its isometry group $\Isom(\AdS^n)$ is identified to the algebraic Lie groups $\O(q_{-1})\cong\O(2,n-1)$}. The \emph{boundary at infinity} of $\AdS^n$ is defined {as the image of the null directions for the quadratic form $q_{-1}$, now in the projective sphere $\widetilde{\RP}{}^n:=(\R^{n+1}\smallsetminus\{0\})/\R_{>0}$:
$$\partial\AdS^n=\{x\in\R^{n+1}\,|\,q_{-1}(x)=0\}/\R_{>0}~.$$
Interpreting $\AdS^n$ and $\partial\AdS^n$ as subsets of $\widetilde{\RP}{}^n$ gives $\AdS^n\cup\partial\AdS^n$ a natural topology, which makes it 
homeomorphic to $S^1 \times D^{n-1}$. Again, the \emph{ideal closure} of a 
subset $A \subset \AdS^n$ {that is closed in $\AdS^n$} is its closure $\overline{A}$ in 
$\widetilde{\RP}{}^n$, and we have $\overline{\AdS}{}^n = \AdS^n \cup \partial \AdS^n$.}

\subsection{Hyperplanes and reflections} \label{sec:hyperplanes}

{A \emph{hyperplane} of $\Hyp^n$ (resp. $\AdS^n$) is
the intersection, when non-empty, of a linear hyperplane of $\R^{n+1}$ with $\Hyp^n$ (resp. $\AdS^n$). Notice that, unlike the hyperbolic case, the intersection of a linear hyperplane with $\AdS^n$ is always non-empty, and sometimes disconnected (see the discussion preceding Lemma \ref{lemma hyperplane ads}). In both cases hyperplanes are totally geodesic.}

A \emph{reflection} in hyperbolic, resp. Anti-de Sitter, geometry is a non-trivial involution $r \in \Isom(\Hyp^n)$, resp. $\Isom(\AdS^n)$, that fixes point-wise a  hyperplane.

Let us denote by $\perpp$ the orthogonal complement with respect to the bilinear form $b_1$, and let $X \in \R^{n+1}$ be a vector. {The linear hyperplane} $X^{\perpp}$ of $\R^{n+1}$ intersects $\Hyp^n$ if and only if $q_1(X)>0$, i.e. if $X$ is \emph{spacelike} for $q_1$. Hence to every $q_1$-spacelike vector $X$ is associated a hyperplane 
$$H_{X}={X^{\perpp}}\cap\Hyp^n~.$$
It is clearly harmless to assume that $q_1(X)=1$, so that the vector $X$ is uniquely determined up to changing the sign.

Given a hyperplane $H_{X}$ in $\Hyp^n$, there is a unique reflection $r_{X}$ fixing the given hyperplane. Indeed, the reflection $r_{X}$ is  the linear transformation in $\O(q_1)$ which fixes $ X^\perpp$ and acts on the subspace generated by $X$ as minus the identity. Two spacelike unit vectors $X$ and $Y$ give the same reflection if and only if $X = \pm Y$. Finally, it is a simple exercise to show that  two reflections $r_{X}$ and $r_{Y}$ commute if and only if either $X=\pm Y$ or $X$ and $Y$ are orthogonal for the bilinear form $b_1$.

We summarise the above considerations in the following statement:

\begin{lemma} \label{lemma:commuting hyp}
There is a two-sheeted covering map
$$\{X\in\R^{n+1}\,:\,q_1(X)=+1\}\to\{r\in\Isom(\Hyp^n):\,r\mathrm{\ is\ a\ reflection}\}~,$$
which maps a spacelike unit vector $X$ to the unique reflection $r_{X}$ fixing  $H_{X}$ point-wise. Moroever, two distinct reflections $r_{X}$ and $r_{Y}$ commute if and only if $b_1(X,Y)=0$. 
\end{lemma} 
  
The subset of vectors in $\R^{n+1}$ such that $q_1(X)=+1$ is usually called \emph{de Sitter space}. 

Let us now move to Anti-de Sitter geometry. {A hyperplane}  is called \emph{spacelike}, \emph{timelike} or \emph{lightlike} if the induced bilinear form, obtained as the restriction of the Lorentzian metric of $\AdS^n$, is positive definite, indefinite or degenerate, respectively. {Spacelike hyperplanes are disconnected, and each of the two connected components is an isometrically embedded copy of $\Hyp^{n-1}$.} Timelike hyperplanes are isometrically embedded copies of $\AdS^{n-1}$.

Let us denote by $\perpm$ the orthogonality relation with respect to $b_{-1}$. We have:

\begin{lemma} \label{lemma hyperplane ads}
Given a vector $X \in \R^{n+1}$, the intersection $X^\perpm\cap\AdS^n$ is non-empty, and is:
\begin{itemize}
\item a spacelike hyperplane if $q_{-1}(X)<0$, 
\item a timelike hyperplane if $q_{-1}(X)>0$,
\item a lightlike hyperplane if $q_{-1}(X)=0$.
\end{itemize}
\end{lemma}

The hyperplane of fixed points of an AdS reflection is either spacelike or timelike. Similarly to the hyperbolic case, given a vector $X$ such that $q_{-1}(X)=\pm 1$, the unique reflection fixing 
$$H_{X}= {X^{\perpm}}\cap\AdS^n$$ is induced by the linear transformation in $\O(q_{-1})$ acting on $X^\perpm$ as the identity and on $\mathrm{Span}(X)$ (which is  in direct sum with  $X^\perpm$ since $q_{-1}(X)\neq 0$) as minus the identity.

In conclusion, we have another summarising statement:

\begin{lemma} \label{lemma:map refl ads}
There is a two-sheeted covering map
$$\{X\in\R^{n+1}\,:\,q_{-1}(X)=\pm 1\}\to\{r\in\Isom(\AdS^n):\,r\mathrm{\ is\ a\ reflection}\}~,$$
which maps a spacelike or timelike unit vector $X$ to the unique reflection $r_{X}$ fixing  $H_{X}$ point-wise. Moroever, two distinct reflections $r_{ X}$ and $r_{ Y}$ commute if and only if $b_{-1}(X,Y)=0$. 
\end{lemma} 

The space $\{X\in\R^{n+1}\,:\,q_{-1}(X)=\pm 1\}$ has two connected components, as well as the space of reflections. The component defined by $q_{-1}(X)=-1$  is a {copy of} Anti-de Sitter space itself. 

\subsection{Relative position of hyperplanes} \label{sec:relative}

It will be useful to discuss the relative position of hyperplanes
. For hyperbolic geometry, this is easily summarised:

\begin{lemma} \label{lemma angle hyp}
Let $H_{X}$ and $H_{Y}$ be two distinct hyperplanes in $\Hyp^n$, for $q_1(X)=q_1(Y)=1$. Then the following hold:
\begin{itemize}
\item $H_{X}$ and $H_{Y}$ intersect in $\Hyp^n$ if and only if $|b_1(X,Y)|<1$;
\item {$\overline{H}_{X}$ and $\overline{H}_{Y}$ intersect in $\partial\Hyp^n$} if and only if $|b_1(X,Y)|=1$;
\item {$\overline{H}_{X}$ and $\overline{H}_{Y}$ are disjoint in $\overline{\Hyp}{}^n$}  if and only if $|b_1(X,Y)|>1$.
\end{itemize}
\end{lemma}

%

{In the second item of Lemma \ref{lemma angle hyp}, $\overline{H}_X$ and $\overline{H}_Y$ intersect 
in exactly one point at infinity $p \in \partial \Hyp^n$. In this case, we say that $H_X$ and $H_Y$ are \emph{asymptotic (at $p$).} 
By little abuse, sometimes we also say that a hyperplane $H$ is \emph{asymptotic} to a point at infinity $p$ if $p \in \overline{H}$. Note the difference between the two notions: if the first item holds, then $H_X$ and $H_Y$ are not asymptotic, despite being asymptotic to $p$ for any point at infinity $p \in \overline{H}_X \cap \overline{H}_Y \neq \emptyset$.}

\begin{figure}
\includegraphics[scale=0.33]{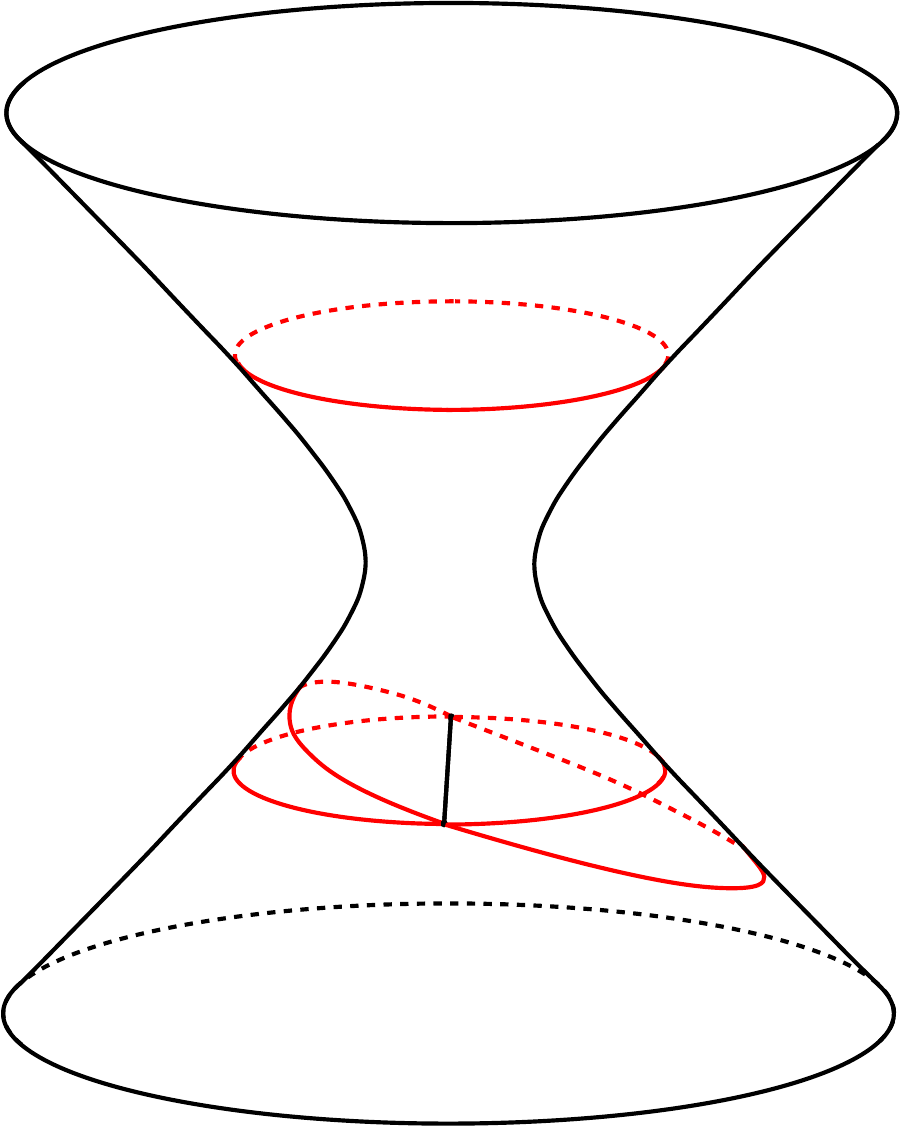}
\caption[Intersecting spacelike hyperplanes in $\AdS^n$.]{\footnotesize Two spacelike hyperplanes in $\AdS^n$ can intersect in a totally geodesic spacelike hyperplane, at a point at infinity, or be disjoint. The picture ($n=3$) is in an affine chart for {the real projective sphere}, where Anti-de Sitter space is the interior of a one-sheeted hyperboloid. {Each disk drawn in the picture represents a connected component of a spacelike hyperplane, and has an isometric copy on the ``opposite'' affine chart of $\widetilde\RP{}^n$,} {obtained by applying minus the identity to $\AdS^n$.}}
\label{fig:AdSplanes}
\end{figure}

\begin{figure}
\begin{minipage}[c]{.4\textwidth}
\centering
\includegraphics[scale=.33]{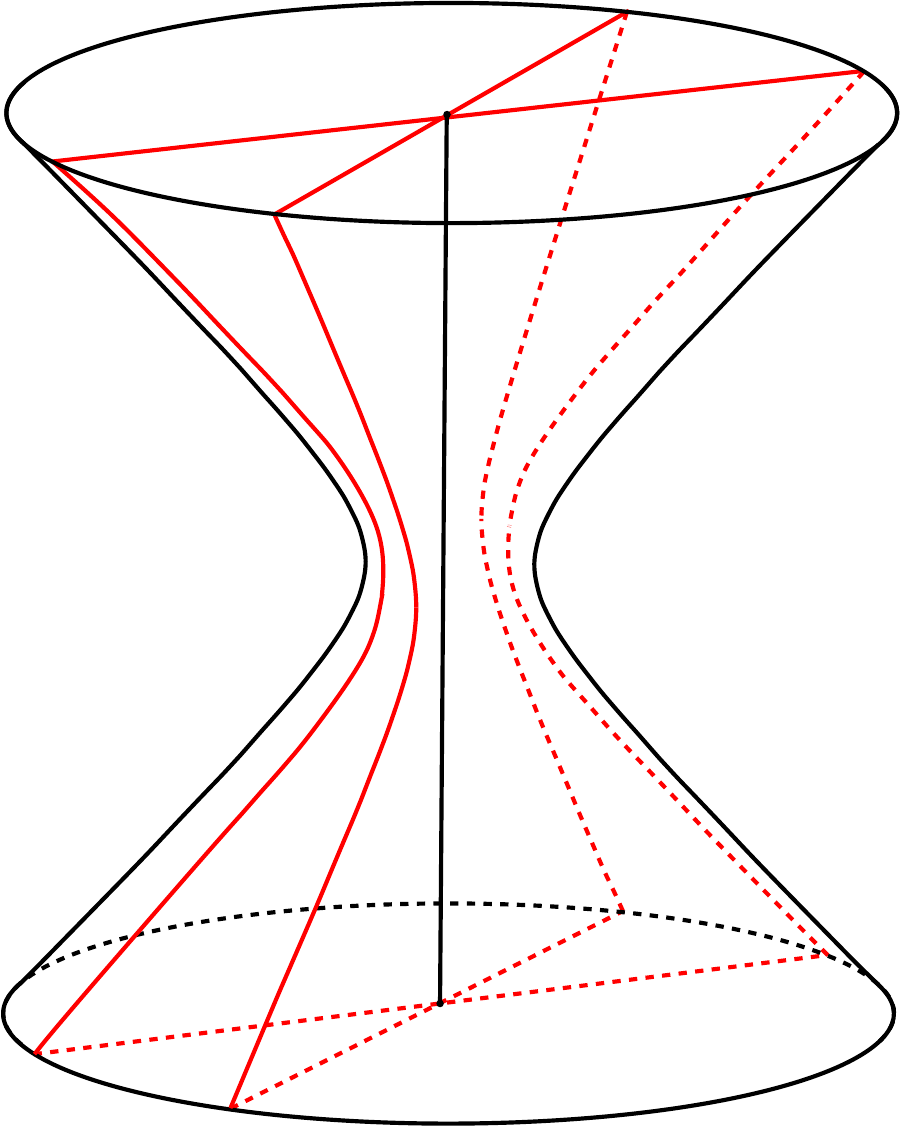}
\end{minipage}
\hspace{1cm}
\begin{minipage}[c]{.4\textwidth}
\centering
\includegraphics[scale=.33]{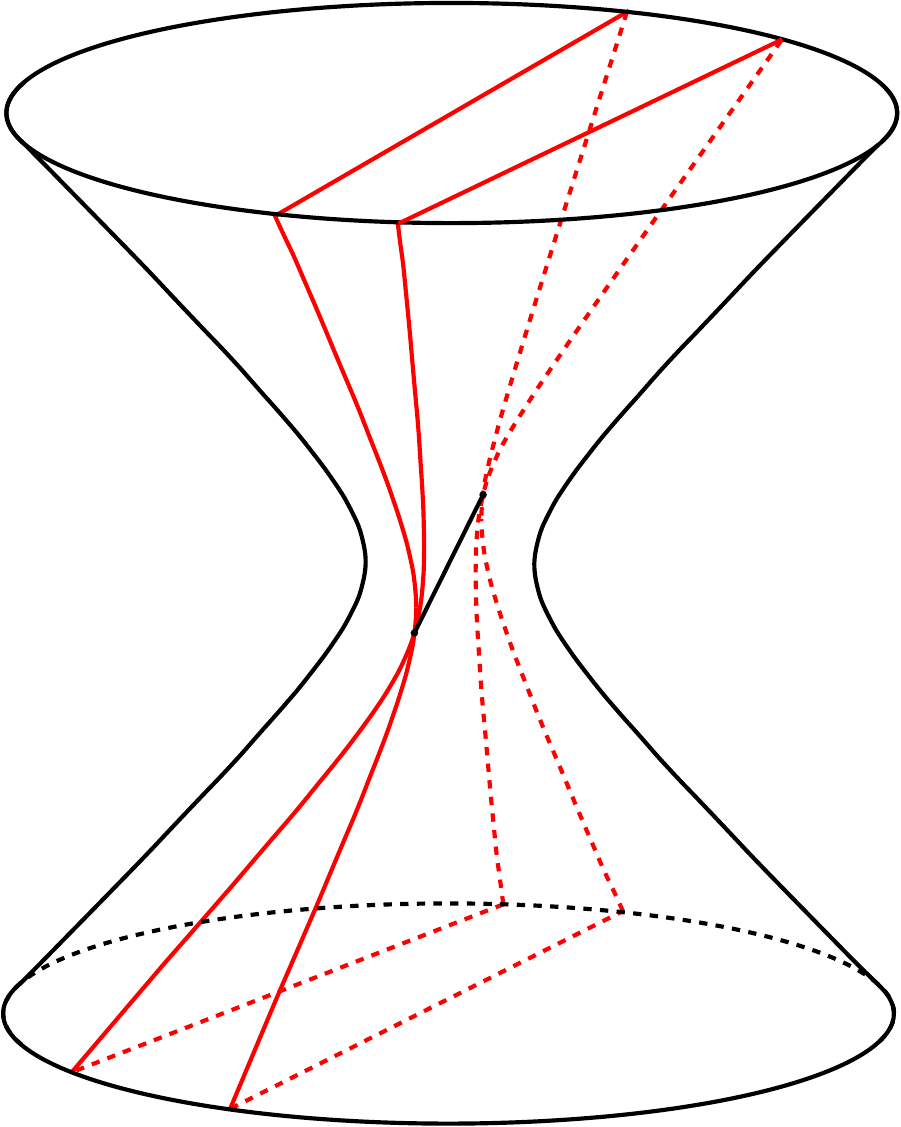}
\end{minipage}
\caption[Intersecting timelike hyperplanes in $\AdS^n$.]{\footnotesize Two timelike planes in $\AdS^3$ intersecting in a timelike (left) or spacelike (right) line. {One should keep in mind that the hyperplanes enter into the ``opposite''} {affine chart, where one has a completely analogous picture.}}
\label{fig:AdSplanes_timelike}
\end{figure}

For Anti-de Sitter hyperplanes, it is necessary to distinguish several cases. Here we will only consider the cases of interest for the proofs of our main results.
For spacelike hyperplanes, we have (see Figure \ref{fig:AdSplanes}):

\begin{lemma}\label{lem: AdS angle}
Let $H_{X}$ and $H_{Y}$ be two distinct spacelike hyperplanes in $\AdS^n$, for $q_{-1}(X)=q_{-1}(Y)=-1$. Then,
\begin{itemize}
\item $H_{X}$ and $H_{Y}$ intersect in $\AdS^n$ if and only if $|b_{-1}(X,Y)|>1$;
\item {$\overline{H}_{X}$ and $\overline{H}_{Y}$ intersect in $\partial\AdS^n$} if and only if $|b_{-1}(X,Y)|=1$;
\item {$\overline{H}_{X}$ and $\overline{H}_{Y}$ are disjoint in $\overline{\AdS}{}^n$}  if and only if $|b_{-1}(X,Y)|<1$.
\end{itemize}
\end{lemma}

{Similar terminology for asymptoticity is adopted when the second item of Lemma \ref{lem: AdS angle} occurs,} {the only difference with the hyperbolic case being that now $\overline{H}_X$ and $\overline{H}_Y$ intersect in exactly two antipodal points at infinity $\pm p \in \partial \AdS^n$.}

For two AdS timelike hyperplanes the situation is different, as explained in the following lemma. See also Figure \ref{fig:AdSplanes_timelike}.

\begin{lemma} \label{lem: AdS intersection timelike}
Let $H_{X}$ and $H_{Y}$ be two distinct timelike hyperplanes in $\AdS^n$, for $q_{-1}(X)=q_{-1}(Y)=1$. Then $H_{X}$ and $H_{Y}$ intersect in $\AdS^n$ and the intersection $H_{X}\cap H_{Y}$ is:
\begin{itemize}
\item spacelike  if and only if ${|b_{-1}(X,Y)|}>1$;
\item lightlike if and only if ${|b_{-1}(X,Y)|}=1$;
\item timelike if and only if ${|b_{-1}(X,Y)|}<1$.
\end{itemize} 
\end{lemma}

{In the first case of Lemma \ref{lem: AdS angle} and Lemma \ref{lem: AdS intersection timelike}, the intersection $H_{X}\cap H_{Y}$ consists of two totally geodesic copies} of $\Hyp^{n-2}$; in the third case of Lemma \ref{lem: AdS intersection timelike} it is a totally geodesic copy of $\AdS^{n-2}$.

\section{Right-angled cusp groups in hyperbolic and AdS geometry} \label{sec:cusp_rigidity}

In this section, which may be of independent interest, we study deformations of some right-angled Coxeter groups represented as ``cusp groups'' in $\Isom(\Hyp^n)$ and $\Isom(\AdS^n)$, for $n=3,4$. {Let us begin in the next subsection with a general set up.} 

\subsection{Coxeter groups and representation varieties} \label{sec: deform reflections}

Given a finitely presented group $\Gamma$ and an algebraic Lie group $G$, 
we denote by $\mathrm{Hom}(\Gamma,G)$ the space of representations $\rho\colon\Gamma\to G$.  Since $\mathrm{Hom}(\Gamma,G)$ is naturally an affine algebraic set (see also Section \ref{sec:zariski}), it is called \emph{representation variety}.

In the remainder of the paper, we will restrict the attention to the case where $\Gamma$ is a right-angled Coxeter group, which we now define.

\begin{defi}[RACG] \label{defi RACG}
Given a finite set $S$ and a subset $R$ of (unordered) pairs of distinct elements of $S$, the associated \emph{right-angled Coxeter group} has presentation:
$$\langle \, S \ \big| \ {\l s}^2 = 1 \ \forall \, {\l s} \in S,\ {\l s}_1 {\l s}_2 = {\l s}_1 {\l s}_2 \ \forall \, ({\l s}_1, {\l s}_2) \in R \, \rangle~.$$
\end{defi}

For instance, the group generated by the reflections in the sides of a right-angled Euclidean or hyperbolic polytope is a right-angled Coxeter group. 

We will only be interested in representations of a right-angled Coxeter group $\Gamma$ which send every generator 
to a reflection. Let us introduce more formally this space.

\begin{defi}[The set $\mathrm{Hom}_{\mathrm{refl}}$] \label{defi hom refl}
Let $G$ be $\Isom(\Hyp^n)$ or $\Isom(\AdS^n)$ and $\Gamma$ a right-angled Coxeter group as above. We define $\mathrm{Hom}_{\mathrm{refl}}(\Gamma,G)$ as the subset of $\mathrm{Hom}(\Gamma,G)$ of representations $\rho$ such that:
\begin{itemize}
\item for every ${\l s} \in S$, the isometry $\rho({\l s})$ is a reflection,  and
\item for every $({\l s}_1,{\l s}_2) \in R$, the reflections $\rho({\l s}_1)$ and $\rho({\l s}_2)$ are distinct.
\end{itemize}
\end{defi}

Reflections constitute a connected component in the space of order-two isometries in $\Isom(\Hyp^n)$, while in $\Isom(\AdS^n)$ they constitute two connected components, given by reflections in spacelike and timelike hyperplanes. Moreover, by Lemmas \ref{lemma:commuting hyp} and \ref{lemma:map refl ads}, two distinct reflections commute if and only if their fixed hyperplanes are orthogonal. Hence we immediately get:

\begin{lemma} \label{lemma clopen}
The subset $\mathrm{Hom}_{\mathrm{refl}}(\Gamma,G)$ is clopen in $\mathrm{Hom}(\Gamma,G)$.
\end{lemma}

To simplify the computations, we will follow \cite{KS} and adopt a local model for the representation variety which is well adapted to our setting. Roughly speaking, we only consider the deformations of the hyperplanes fixed by the reflection associated to each generator. This will reduce significantly the complexity of the problem, since (in dimension $n$) for each generator we have a vector of $n+1$ entries (giving the hyperplane of reflection) in place of an $(n+1)\times (n+1)$ matrix (giving the reflection itself).

More precisely, the following lemma gives a local parametrisation of the set $\mathrm{Hom}_{\mathrm{refl}}(\Gamma,G)$:

\begin{lemma} \label{lemma homeo model general} \label{lemma homeo model rep var}
The set $\mathrm{Hom}_{\mathrm{refl}}(\Gamma,G)$ is finitely covered by a disjoint union of subsets of $\R^{(n+1)|S|}$ defined by the vanishing of $|S|+|R|$ quadratic conditions.
\end{lemma}

\begin{proof}
Let us first give the proof for $G=\Isom(\Hyp^n)$. Let us identify $\R^{(n+1)|S|}$ to the vector space of functions $f \colon S \to \R^{n+1}$. For every representation $\rho \colon \Gamma \to G$ in $\mathrm{Hom}_{\mathrm{refl}}(\Gamma,G)$, we can choose a function $f$ such that $q_1(f({\l s}))=1$ and $\rho({\l s})=r_{f({\l s})}$ for every generator ${\l s}$. In fact there are $2^{|S|}$ possible choices of such an $f$, differing by changing sign to the image of each generator, and they all satisfy the following conditions: 
\begin{enumerate}
\item The vector $f({\l s})$ is unitary, meaning that $q_{1}(f({\l s}))=1$, hence giving $|S|$ quadratic conditions.
\item For each of the commutation relations ${\l s}_i{\l s}_j={\l s}_j{\l s}_i$ in $\Gamma$, by Lemma \ref{lemma:commuting hyp}  the corresponding vectors $f({\l s}_i)$ and $f({\l s}_j)$ are orthogonal with respect to $b_{1}$.
\end{enumerate}

Conversely, every $f$ satisfying these conditions induces the representation $\rho$ in $\mathrm{Hom}_{\mathrm{refl}}(\Gamma,G)$ defined by $\rho({\l s})=r_{f({\l s})}$. Define a function
$$g \colon \R^{(n+1)|S|}\to \R^{|S|+|R|}$$
by
$$g(f)=((q_1(f({\l s}))-1)_{{\l s}\in S},(b_1(f({\l s}_i),f({\l s}_j)))_{({\l s}_i,{\l s}_j)\in R})~.$$
We have shown that $g^{-1}(0)$ is a $2^{|S|}$-sheeted covering of $\mathrm{Hom}_{\mathrm{refl}}(\Gamma,\Isom(\Hyp^n))$, with deck transformations given by the group $(\Z/2\Z)^{|S|}$. 

The proof for the Anti-de Sitter case is analogous, except that we have to distinguish several cases, depending on whether $\rho({\l s})$ is a reflection in a spacelike or timelike hyperplane. In the former case, we must impose $q_{-1}(f({\l s}))=-1$, and in the latter  $q_{-1}(f({\l s}))=1$ (see Lemma \ref{lemma hyperplane ads}). The orthogonality conditions are the same, but now using the bilinear form $b_{-1}$, by Lemma \ref{lemma:map refl ads}.  In conclusion we have that $\mathrm{Hom}_{\mathrm{refl}}(\Gamma,\Isom(\AdS^n))$ is finitely covered by a disjoint union of $|S|$ subsets each defined by the vanishing of a quadratic function $g\colon\R^{(n+1)|S|}\to \R^{|S|+|R|}$.
\end{proof}

\begin{remark} \label{rmk:action of isometry group}
The covering map of Lemma \ref{lemma homeo model rep var} is equivariant with respect to the following two actions of $\O(q_{\pm 1})$:
\begin{itemize} 
\item The action on the space of functions $f\colon S\to\R^{n+1}$ by post-composition $(A\cdot f)({\l s})=A(f({\l s}))$. This action preserves the zero locus of the defining functions $g\colon\R^{(n+1)|S|}\to \R^{|S|+|R|}$ introduced in the proof of Lemma \ref{lemma homeo model rep var};
\item The action on $\mathrm{Hom}_{\mathrm{refl}}(\Gamma,G)$ given by the $G$-action by conjugation. This action preserves the 
clopen subset $\mathrm{Hom}_{\mathrm{refl}}(\Gamma,G)$ of $\mathrm{Hom}(\Gamma,G)$.
 \end{itemize}
\end{remark}

\subsection{Flexibility in dimension three} \label{sec:cusp_flex}

Let $\Gamma_{\mathrm{rect}}$ denote the right-angled Coxeter group generated by the reflections along the sides of a Euclidean rectangle. The standard presentation of $\Gamma_{\mathrm{rect}}$ has 4 generators (one for each side of the rectangle), and  relations such that each generator has order two and reflections in adjacent sides commute.

\begin{defi}[Cusp group in dimension 3]
The image of a representation of $\Gamma_{\mathrm{rect}}$ into $\Isom(\Hyp^3)$ or $\Isom(\AdS^3)$ is called a \emph{cusp group} if the four generators are sent to reflections in four distinct planes {asymptotic to a common point at infinity}. 
\end{defi}

We will also consider other similar representations of  $\Gamma_{\mathrm{rect}}$, which  occur in correspondence to a collapse, when two non-commuting generators are sent to the same reflection. Let us begin with the hyperbolic case:

\begin{defi}[Collapsed cusp group for $\Hyp^3$]
The image of a representation of $\Gamma_{\mathrm{rect}}$ into $\Isom(\Hyp^3)$  is called a \emph{collapsed cusp group} if the four generators are sent to reflections along three distinct planes {asymptotic to a common point at infinity}.
\end{defi}

Let $\rho'$ be a representation near a given $\rho \in \mathrm{Hom}_{\mathrm{refl}}(\Gamma_{\mathrm{rect}},\Isom(\Hyp^3))$, and ${\l s}$ be a generator of $\Gamma_{\mathrm{rect}}$. In virtue of Lemma \ref{lemma clopen} and the discussion below, we refer to the fixed-point set of $\rho'({\l s})$ as a \emph{plane} of $\rho'$.

In \cite[Lemma 5.1]{KS}, the following property of cusp groups is proved:

\begin{prop} \label{prop rect group hyp}
Let $\rho\colon \Gamma_{\mathrm{rect}}\to\Isom(\Hyp^3)$ be a representation whose image is a cusp group. For all nearby representations whose image is not a cusp group, a pair of opposite planes intersect in $\Hyp^3$, while the other pair of opposite planes have disjoint {ideal} closures in $\overline{\Hyp}^3$. 
\end{prop}

In fact, a simple adaptation of the proof shows:

\begin{prop} \label{prop rect group hyp collapsed}
Let $\rho\colon \Gamma_{\mathrm{rect}}\to\Isom(\Hyp^3)$ be a representation whose image is a cusp group or a collapsed cusp group. For all nearby representations $\rho'$, exactly one of the following possibilities holds:
\begin{enumerate}
\item If ${\l s}_1$ and ${\l s}_2$ are generators such that $\rho({\l s}_1)=\rho({\l s}_2)$, then $\rho'({\l s}_1)=\rho'({\l s}_2)$. 
\item The image of $\rho'$ is a cusp group. 
\item A pair of opposite planes intersect in $\Hyp^3$, while the other pair of opposite planes have disjoint {ideal} closures in $\overline{\Hyp}^3$. 
\end{enumerate}
\end{prop}

The first case may hold only if $\rho$ is a collapsed cusp group. Under this hypothesis, Proposition \ref{prop rect group hyp collapsed} can be rephrased by saying that a deformation of a collapsed cusp group either preserves the property that two planes corresponding to non-adjacent sides of the rectangle coincide (which is the case when the collapsed cusp group remains a collapsed cusp group, for instance), or it falls in the class of representations described in Proposition \ref{prop rect group hyp}, namely, deformations of non-collapsed cusp groups. If $\rho$ is a cusp group, then the content of Proposition \ref{prop rect group hyp collapsed} is the same as Proposition \ref{prop rect group hyp}.

We now move to the AdS version of Propositions \ref{prop rect group hyp} and \ref{prop rect group hyp collapsed}, for which we will give a complete proof. Proofs for the hyperbolic case are easier and can be repeated by mimicking the AdS case. 

Note that for an AdS cusp group the four planes necessarily satisfy the orthogonality conditions as in a rectangle, and therefore two of them are spacelike and two timelike. We will show the following proposition, which is the AdS version of Proposition \ref{prop rect group hyp}.

\begin{prop} \label{prop rect group ads}
Let $\rho\colon \Gamma_{\mathrm{rect}}\to\Isom(\AdS^3)$ be a representation whose image is a cusp group. For all nearby representations whose image is not a cusp group, exactly one of the following possibilities holds:
\begin{enumerate}
\item The {ideal closures of the} two spacelike planes are disjoint in $\overline{\AdS}^3$, whereas the two timelike planes intersect in a timelike geodesic of $\AdS^3$;
\item The two spacelike planes intersect in $\AdS^3$, whereas the two timelike planes intersect in {two spacelike geodesics} of $\AdS^3$.
\end{enumerate}
\end{prop}

Proposition \ref{prop rect group ads} follows from the more general Proposition \ref{prop rect group ads collapsed} below, which also includes the collapsed case. We will consider only the degeneration of cusp groups to a collapsed cusp group when the two planes which coincide are spacelike, as in the following definition:

\begin{defi}[Collapsed cusp group for $\AdS^3$]
The image of a representation of $\Gamma_{\mathrm{rect}}$ into $\Isom(\AdS^3)$ is called a \emph{collapsed cusp group} if the four generators are sent to reflections along three distinct planes, two timelike and one spacelike, {asymptotic to a common} point at infinity.
\end{defi}

\begin{prop} \label{prop rect group ads collapsed}
Let $\rho\colon\Gamma_{\mathrm{rect}}\to\Isom(\AdS^3)$ be a representation whose image is a cusp group or a collapsed cusp group. For all nearby representations $\rho'$, exactly one of the following possibilities holds:
\begin{enumerate}
\item If ${\l s}_1$ and ${\l s}_2$ are  generators such that $\rho({\l s}_1)=\rho({\l s}_2)$ is a reflection in a spacelike plane, then $\rho'({\l s}_1)=\rho'({\l s}_2)$. 
\item The image of $\rho'$ is a cusp group. 
\item The {ideal closures of the} two spacelike planes are disjoint in $\overline{\AdS}^3$, whereas the two timelike planes intersect in a timelike geodesic of $\AdS^3$.
\item The two spacelike planes intersect in $\AdS^3$, whereas the two timelike planes intersect in {two spacelike geodesics} of $\AdS^3$.
\end{enumerate}
\end{prop}

\begin{proof}
By Lemmas \ref{lemma clopen} and \ref{lemma homeo model general}, it is sufficient to analyse a neighborhood of a lift $f\colon S\to\R^4$ of $\rho$, where $S$ is the standard generating set of $\Gamma_{\mathrm{rect}}$. Let us denote ${\l s}_1,{\l s}_2$ the generators which are sent by $\rho$ to a reflection in a spacelike plane, and  ${\l t}_1,{\l t}_2$ those sent to a reflection in a timelike plane. The same will occur for representations nearby $\rho$.

Let us fix a nearby representation $\rho'$ and a lift $f'\colon S\to\R^4$. Let us denote 
$X_i=f'({\l s}_i)$ and $Y_i=f'({\l t}_i)$. Recall that $X_i$ is orthogonal to $Y_j$ for $i,j=1,2$. 

Suppose that $X_1\neq \pm X_2$, for otherwise we are in the case of item (1). Up to the action of $\O(q_{-1})$ (see Remark \ref{rmk:action of isometry group}) and up to changing signs, we can assume once and forever that
$$X_1=(1,0,0,0)\quad\mbox{and}\quad Y_1=(0,1,0,0)~.$$

{Suppose first that the hyperplanes 
 $H_{Y_2}$, $H_{X_1}$ and $H_{Y_1}$ are asymptotic to the same point at infinity $p$}. {(As a side remark, observe that in this case they are also asymptotic to the antipodal point $-p$.)} We can assume ${p=(0,0,1,1)}\in\partial\AdS^3$. Together with the orthogonality of $Y_2$ with $X_1$, this implies (up to changing the sign if necessary)
$$Y_2=(0,1,a,-a)$$
for some parameter $a\neq 0$. Applying the orthogonality of $X_2$ with $Y_1$ and $Y_2$, we now find (always up to a sign)
$$X_2=(1,0,b,-b)$$
for some $b$, which implies that $H_{X_2}$ also {
is asymptotic to} the point {$p=(0,0,1,1)$}. Thus we still have a cusp group and we are in the case of item (2).

Suppose instead that {
$\overline{H}_{Y_2} \cap \overline{H}_{X_1} \cap \overline{H}_{Y_1} = \emptyset$}. {Consider a connected component $H_{X_1}^0$ of $H_{X_1}$, which is a copy of $\Hyp^2$. We have two geodesics in $H_{X_1}^0$: $\ell_1=H_{Y_1}\cap H_{X_1}^0 $ and $\ell_2=H_{Y_2}\cap H_{X_1}^0$}. There are two possibilities: either $\ell_1$ and $\ell_2$ intersect in {$H_{X_1}^0$}, or they are ultraparallel. See Figure \ref{fig:AdSgeodesics} and the related Figure \ref{fig:AdSplanes_timelike}.

\begin{figure}
\begin{minipage}[c]{.4\textwidth}
\centering
\includegraphics[scale=.60]{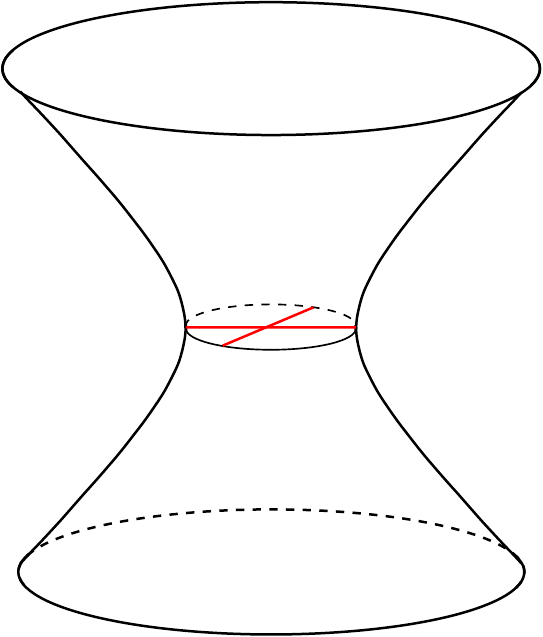}
\end{minipage}
\hspace{1cm}
\begin{minipage}[c]{.4\textwidth}
\centering
\includegraphics[scale=.6]{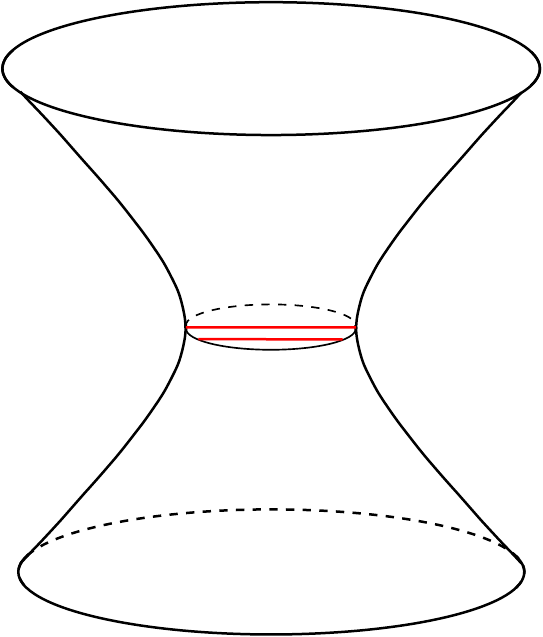}
\end{minipage}
\caption[Two geodesics in $\Hyp^2\subset\AdS^3$.]{\footnotesize The two configurations for the geodesics $\ell_1$ and $\ell_2$ in the ``horizontal'' spacelike plane $H_{X_1}$, as in the proof of Proposition \ref{prop rect group ads collapsed}. The two timelike planes $H_{Y_i}$ and  containing $\ell_i$ and orthogonal to $H_{X_1}$ are as in Figure \ref{fig:AdSplanes_timelike}, left and right figure respectively.}
\label{fig:AdSgeodesics}
\end{figure}

If $\ell_1$ and $\ell_2$ intersect in {$H_{X_1}^0$}, we can assume that $\ell_1\cap\ell_2=\{{(0,0,0,1)}\}$. Equivalently,
$$Y_2=(0,\cos\theta,\sin\theta,0)~,$$
where $\theta$ is the angle between the two geodesics in {$H_{X_1}^0$}. In this case, the two timelike planes $H_{Y_1}$ and $H_{Y_2}$ have timelike intersection by Lemma \ref{lem: AdS intersection timelike} (the intersection is indeed the timelike geodesic ${(\cos(s),0,0,\sin(s))}$). Imposing the orthogonality of $H_{X_2}$ with $H_{Y_1}$ and $H_{Y_2}$, we find (up to a sign)
$$X_2=(\cos\varphi,0,0,\sin\varphi)~,$$
which means that $H_{X_1}$ and $H_{X_2}$ are disjoint in $\overline{\AdS}{}^3$ by Lemma \ref{lem: AdS angle}. (The parameter $\varphi$ is indeed the timelike distance between $H_{X_1}$ and $H_{X_2}$, which is achieved on the timelike geodesic we have just introduced.) So, in this case item (3) of the statement holds.

If $\ell_1$ and $\ell_2$ are ultraparallel, we can assume
$$Y_2=(0,\cosh\theta,0,\sinh\theta)~,$$
where $\theta$ is now the distance between the two aforementioned geodesics in {$H_{X_1}^0$}. In this case, $H_{Y_1}$ and $H_{Y_2}$ have spacelike intersection (which is the geodesic ${(\cosh(s),0,\sinh(s),0)}$, see Lemma \ref{lem: AdS intersection timelike}). Imposing again the orthogonality of $H_{X_2}$ with $H_{Y_1}$ and $H_{Y_2}$, and changing sign if necessary, we find
$$X_2=(\cosh\varphi,0,\sinh\varphi,0)~,$$
namely, $H_{X_1}$ and $H_{X_2}$ intersect in $\AdS^3$ by Lemma  \ref{lem: AdS angle} (the parameter $\varphi$ now being their angle of intersection). Thus, item (4) of the statement holds. This concludes the proof.
\end{proof}

\begin{remark}
In the proof of Proposition \ref{prop rect group ads collapsed} we have used only that $\rho$ can be continuously deformed to $\rho'$. Hence the conclusions of Proposition \ref{prop rect group ads collapsed} and Proposition \ref{prop rect group hyp collapsed} actually hold on the entire connected component of $\mathrm{Hom}_{\mathrm{refl}}(\Gamma_{\mathrm{rect}},G)$ containing $\rho$. 
\end{remark}

\subsection{Rigidity in dimension four} \label{sec:rig dim four}

Let us now move to dimension four.

Let $\Gamma_{\mathrm{cube}}$ be the group generated by the reflections in the faces of a Euclidean cube. The group $\Gamma_{\mathrm{cube}}$ has 6 generators, one for each face, and 12 commutation relations, one for each edge of the cube, involving the two faces adjacent to that edge. Of course, there is also a square-type relation for each generator. There is no relation between the generators corresponding to opposite faces. 

%
\begin{defi}[Cusp group in dimension 4] The image of a representation of $\Gamma_{\mathrm{cube}}$ into $\Isom(\AdS^4)$ or $\Isom(\Hyp^4)$ is called a \emph{cusp group} if the 6 generators are sent to reflections in 6  distinct hyperplanes {asymptotic to a common} point at infinity. 
\end{defi} 

In the AdS case, among these 6 hyperplanes, two opposite hyperplanes are necessarily spacelike, while the other 4 are timelike. 

The following proposition is the fundamental property that can be roughly rephrased as: ``cusp groups stay cusp groups''. Its hyperbolic counterpart is proved in \cite[Lemma 5.3]{KS}.

\begin{prop} \label{prop cube group ads}
Let $\rho\colon \Gamma_{\mathrm{cube}}\to\Isom(\AdS^4)$ be a representation whose image is a cusp group. Then all nearby representations are cusp groups.
\end{prop}

Similarly to dimension three, we will obtain Proposition \ref{prop cube group ads} as a special case of a more general statement including the collapsed case. Let us first give the definition of collapsed cusp group, where two non-commuting generators can be sent to the same reflection (along a spacelike hyperplane in the AdS case):

\begin{defi}[Collapsed cusp group in dimension 4]
The image of a representation of $\Gamma_{\mathrm{cube}}$ into $\Isom(\Hyp^4)$ or $\Isom(\AdS^4)$ is called a \emph{collapsed cusp group} if the 6 generators are sent to reflections along 5 distinct hyperplanes 
{asymptotic to a common} point at infinity. In the AdS case, we require that the unique reflection associated to two generators is along a spacelike hyperplane.
\end{defi}
 
Let us now formulate and prove the more general version of Proposition \ref{prop cube group ads}. 

\begin{prop} \label{prop cube group ads collapsed} 
Let $\rho\colon \Gamma_{\mathrm{cube}}\to\Isom(\AdS^4)$ be a representation whose image is a cusp group or a collapsed cusp group. For all nearby representations $\rho'$, exactly one of the following possibilities holds:
\begin{enumerate}
\item If ${\l s}_1$ and ${\l s}_2$ are  generators such that $\rho({\l s}_1)=\rho({\l s}_2)$ is a reflection in a spacelike hyperplane, then $\rho'({\l s}_1)=\rho'({\l s}_2)$. 
\item The image of $\rho'$ is a cusp group. 
\end{enumerate}
\end{prop}

\begin{proof}
Similarly to the three-dimensional case treated in the previous section, any representation $\rho'$ nearby $\rho$ lies in $\mathrm{Hom}_{\mathrm{refl}}(\Gamma_{\mathrm{cube}},G)$, hence it sends the 6 standard generators of $\Gamma_{\mathrm{cube}}$ to reflections. Moreover, the hyperplanes of $\rho'$ have the same type (spacelike or timelike) as for $\rho$.

Let us pick a lift $f'\colon S\to\R^5$ of $\rho'$, for $S$ the standard generating set of $\Gamma_{\mathrm{cube}}$. Denote by ${\l s}_1,{\l s}_2$ the two generators corresponding to opposite faces of the cube which are sent to reflections in spacelike hyperplanes, and $X_i=f'({\l s}_i)$ (so that $q_{-1}(X_i)=-1$). Similarly we define $Y_i$ and $Z_i$ for $i=1,2$, on which $q_{-1}$ takes value $1$.  Hence each of this 6 vectors is orthogonal to 4 of the others: more precisely, $A_i$ is orthogonal to all the others except $A_j$, for $A\in\{X,Y,Z\}$ and $i,j=1,2$.

Now, let us assume that the hyperplanes $H_{X_1}$ and $H_{X_2}$ do not coincide, that is $X_1\neq \pm X_2$. We shall show that the image of $\rho'$ is still a cusp group.

Let us start by considering the intersection with {a connected component $H_{X_1}^0$ of $H_{X_1}$}, which is a copy of $\Hyp^3$. Here we see the (two-dimensional) planes $H_{Y_1}\cap {H_{X_1}^0}$, $H_{Y_2}\cap {H_{X_1}^0}$, $H_{Z_1}\cap {H_{X_1}^0}$ and $H_{Z_2}\cap {H_{X_1}^0}$, whose associated reflections give a representation $\Gamma_{\mathrm{rect}}\to\Isom(\Hyp^3)$ which is nearby a (rectangular) cusp group. As in the proof of Proposition \ref{prop rect group ads}, it is easy to see that if this representation of  $\Gamma_{\mathrm{rect}}$ is a cusp group in ${H_{X_1}^0}$, then necessarily also $H_{X_2}$ 
{is asymptotic to a common} point at infinity with $H_{Y_1}$, $H_{Y_2}$, $H_{Z_1}$, $H_{Z_2}$. Therefore the image of $\Gamma_{\mathrm{cube}}$ is still a cusp group, since we are assuming that $H_{X_2} \neq H_{X_1}$.

Hence let us assume that the representation of $\Gamma_{\mathrm{rect}}$ is not a cusp group, and we will derive a contradiction. By Proposition \ref{prop rect group hyp} (up to relabelling) we may assume that $H_{Y_1}\cap {H_{X_1}^0}$ and $H_{Y_2}\cap {H_{X_1}^0}$ intersect in ${H_{X_1}^0}$, while $H_{Z_1}\cap {H_{X_1}^0}$ and $H_{Z_2}\cap {H_{X_1}^0}$ are disjoint in ${H_{X_1}^0}$. This implies that $H_{Y_1}\cap H_{Y_2}$ is a timelike plane (i.e. a copy of $\AdS^2$), while $H_{Z_1}\cap H_{Z_2}$ is spacelike (i.e. a copy of $\Hyp^2$). To see this, one can in fact assume that, up to the signs,
$$X_1=(1,0,0,0,0)\qquad Y_1=(0,1,0,0,0)\qquad Y_2=(0,\cos\theta,\sin\theta,0,0)~,$$
and apply Lemma \ref{lem: AdS intersection timelike} --- and similarly for $Z_1$ and $Z_2$.

Now, let us consider the intersection with $H_{Y_1}$, which is a copy of $\AdS^3$. We have thus a representation of $\Gamma_{\mathrm{rect}}$ acting on this copy of $\AdS^3$ as a cusp group or collapsed cusp group, with generators which are reflections in $H_{Z_1}\cap H_{Y_1}$, $H_{Z_2}\cap H_{Y_1}$, $H_{X_1}\cap H_{Y_1}$ and $H_{X_2}\cap H_{Y_1}$. Since $H_{Z_1}\cap H_{Z_2}$ is spacelike, then $H_{Z_1}\cap H_{Z_2}\cap H_{Y_1}$ is also spacelike, and therefore we are in the situation of Proposition \ref{prop rect group ads collapsed} item (4), recalling that $H_{X_1}\neq H_{X_2}$ by our assumption. This implies that $H_{X_1}\cap H_{Y_1}$ and $H_{X_2}\cap H_{Y_1}$ intersect in $H_{Y_1}\subset\AdS^4$.

On the other hand, considering the intersection with $H_{Z_1}$, which is again a copy of $\AdS^3$, since $H_{Y_1}\cap H_{Y_2}$ is timelike, we find that $H_{Y_1}\cap H_{Y_2}\cap H_{Z_1}$ is a timelike geodesic. By Proposition \ref{prop rect group ads collapsed} item (3), $H_{X_1}\cap H_{Z_1}$ and $H_{X_2}\cap H_{Z_1}$ do not intersect in $H_{Z_1}$, which in turn implies (since $H_{X_1}$ and $H_{X_2}$ are both orthogonal to $H_{Z_1}$) that $H_{X_1}$ and $H_{X_2}$ are disjoint in $\AdS^4$. This contradicts the conclusion of the previous paragraph.
\end{proof}

\begin{remark}
In case (1) of Proposition \ref{prop cube group ads collapsed}, i.e. when $\rho({\l s}_1)=\rho({\l s}_2)$, the following possibility is not excluded: for some deformation $\rho'$ of $\rho$, the remaining 4 generators ${\l s}_3,\ldots,{\l s}_6$ (which are sent by $\rho$ to a rectangular cusp group in a copy of $\Hyp^3$) are \emph{not} sent by $\rho'$ to a cusp group. 
\end{remark}

The analogous property for $\Hyp^4$, which is a generalisation of \cite[Lemma 5.3]{KS} can be proved along the same lines. We state it here:

\begin{prop} \label{prop cube group ads collapsed 2}
Let $\rho\colon \Gamma_{\mathrm{cube}}\to\Isom(\Hyp^4)$ be a representation whose image is a cusp group or a collapsed cusp group. For all nearby representations $\rho'$ exactly one of the following possibilities holds:
\begin{enumerate}
\item If ${\l s}_1$ and ${\l s}_2$ are generators such that $\rho({\l s}_1)=\rho({\l s}_2)$, then $\rho'({\l s}_1)=\rho'({\l s}_2)$. 
\item The image of $\rho'$ is a cusp group. 
\end{enumerate}
\end{prop}

\section{The hyperbolic and AdS character variety of $\Gamma_{22}$} \label{sec:char_var}

In this section, we study the $\Isom(\Hyp^4)$ and $\Isom(\AdS^4)$ character varieties of the group $\Gamma_{22}$ near the conjugacy classes of the holonomy representations $\rho_t$ found in \cite{KS,transition_4-manifold}. We prove here the ``topological part'' of Theorem \ref{teo:main} (Theorem \ref{teo:main_weak}) in the hyperbolic and AdS case.

\subsection{The group $\Gamma_{22}$} \label{sec:Gamma22}

As in \cite{KS}, we define
$$\Gamma_{22}<\mathrm{Isom}(\Hyp^4)$$
as the group generated by the hyperbolic reflections  along the hyperplanes determined by the $22$ vectors in Table \ref{table:G22}. These hyperplanes bound a right-angled polytope in $\Hyp^4$ of infinite volume, which is obtained by ``removing two opposite walls'' from the ideal right-angled $24$-cell. 

All the dihedral angles between two intersecting hyperplanes are right. Therefore $\Gamma_{22}$ is a right-angled Coxeter group. We will consider $\Gamma_{22}$ as an abstract group, that is the right-angled Coxeter groups on $22$ generators
$$\p0,\ldots,\p7,\m0,\ldots,\m7,\l A,\ldots,\l F$$
satisfying the following relations:
\begin{itemize}
\item $\l s^2=1$ for each generator $\l s$,
\item $\l s_1\l s_2=\l s_2\l s_1$ for each pair $\l s_1,\l s_2$ of generators such that the corresponding vectors in Table \ref{table:G22} are orthogonal with respect to the bilinear form $b_1$. 
\end{itemize}

The generators are partitioned into 3 \emph{types}: \emph{positive} $\p0,\ldots,\p7$, \emph{negative} $\m0,\ldots,\m7$, and \emph{letters} $\l A,\ldots,\l F$. The type is inherited from the standard 3-colouring of the facets of the 24-cell (see \cite{KS} for more details).

The reader can check from Table \ref{table:G22} that there are no commutation condition between two generators of the same type, that every $\l i^+$ commutes with 4 vectors of type $\l j^-$ (including $\m i$), and every $\l X\in\{\l A,\ldots,\l F\}$ commutes with $\m i$ and $\p i$ for 4 choices of $\l i\in\{\l 0,\ldots,\l 7\}$. Hence there are $8\cdot 4+6\cdot 8=80$ commutation relations. Altogether, there are $102=22+80$ relations. 

We would like to stress once more that throughout the following (with a few exceptions which will be remarked opportunely) we will use the symbols $\p i\in\{\p0,\ldots,\p7\}$, $\m i\in\{\m0,\ldots,\m7\}$ and $\l X\in\{\l A,\ldots,\l F\}$ to denote the 22 abstract generators of $\Gamma_{22}$ (rather than vectors in $\R^5$). 

\begin{table} 
\begin{eqnarray*}
\p{0} = \left( \sqrt{2},+1,+1,+1,+1 \right) , & &
\m{0} = \left( \sqrt{2},+1,+1,+1,-1 \right),\\
\p{1} = \left( \sqrt{2},+1,-1,+1,-1 \right),& &
\m{1} = \left( \sqrt{2},+1,-1,+1,+1 \right),\\
\p{2} = \left( \sqrt{2},+1,-1,-1,+1 \right),& &
\m{2} = \left( \sqrt{2},+1,-1,-1,-1 \right),\\
\p{3} = \left( \sqrt{2},+1,+1,-1,-1 \right),& &
\m{3} = \left( \sqrt{2},+1,+1,-1,+1 \right),\\
\p{4} = \left( \sqrt{2},-1,+1,-1,+1 \right),& &
\m{4} = \left( \sqrt{2},-1,+1,-1,-1 \right),\\
\p{5} = \left( \sqrt{2},-1,+1,+1,-1 \right),& &
\m{5} = \left( \sqrt{2},-1,+1,+1,+1 \right),\\
\p{6} = \left( \sqrt{2},-1,-1,+1,+1 \right),& &
\m{6} = \left( \sqrt{2},-1,-1,+1,-1 \right),\\
\p{7} = \left( \sqrt{2},-1,-1,-1,-1 \right),& &
\m{7} = \left( \sqrt{2},-1,-1,-1,+1 \right),\\
\l{A} = \left( 1,+\sqrt{2},0,0,0 \right),& &
\l{B} = \left( 1,0,+\sqrt{2},0,0 \right),\\
\l{C} = \left( 1,0,0,+\sqrt{2},0 \right),& &
\l{D} = \left( 1,0,0,-\sqrt{2},0 \right),\\
\l{E} = \left( 1,0,-\sqrt{2},0,0 \right),& &
\l{F} = \left( 1,-\sqrt{2},0,0,0 \right).
\end{eqnarray*}
\caption{\footnotesize The $22$ unit vectors defining the bounding hyperplanes of a right-angled polytope in $\Hyp^4$. The reflections in these hyperplanes generate the Coxeter group $\Gamma_{22}$. Adding the vectors $(1,0,0,0,\pm\sqrt 2)$ to this list one obtains the ideal right-angled $24$-cell.}\label{table:G22}
\end{table}

\subsection{A curve of geometric representations}

Let us now introduce the representations  of our interest, which appear in the statement of Theorem \ref{teo:main}. Unlike the introduction, we will omit the superscript $G$ hereafter, and the ambient geometry we consider will be clear from the context.

\begin{defi}[The two paths $\rho_t$] \label{defi rhot}
For $t\in (-1,1)$, we define $\rho_t$ to be the representation of $\Gamma_{22}$ in $\Isom(\Hyp^4)$ (resp. $\Isom(\AdS^4)$) sending each generator $\l s$ of $\Gamma_{22}$ to the hyperbolic (resp. AdS) reflection $r_{f_t(\l s)}$ associated to the corresponding vector $f_t(\l s)$ of Table \ref{table:walls} (resp. Table \ref{table:walls2}).
\end{defi}

\begin{table} 
\small
\begin{eqnarray*}
f_t(\p{0}) =\frac{1}{\sqrt{1+t^2}} \left( \sqrt{2}\ t,+t,+t,+t,+1 \right) , & &
f_t(\m{0}) =\frac{1}{\sqrt{1+t^2}} \left( \sqrt{2},+1,+1,+1,-t \right),\\  
f_t(\p{1}) =\frac{1}{\sqrt{1+t^2}} \left( \sqrt{2}\ t,+t,-t,+t,-1 \right),& &
f_t(\m{1}) =\frac{1}{\sqrt{1+t^2}} \left( \sqrt{2},+1,-1,+1,+t \right),\\
f_t(\p{2}) =\frac{1}{\sqrt{1+t^2}} \left( \sqrt{2}\ t,+t,-t,-t,+1 \right),& &
f_t(\m{2}) =\frac{1}{\sqrt{1+t^2}} \left( \sqrt{2},+1,-1,-1,-t \right),\\
f_t(\p{3}) =\frac{1}{\sqrt{1+t^2}} \left( \sqrt{2}\ t,+t,+t,-t,-1 \right),& &
f_t(\m{3}) =\frac{1}{\sqrt{1+t^2}} \left( \sqrt{2},+1,+1,-1,+t \right),\\
f_t(\p{4}) =\frac{1}{\sqrt{1+t^2}} \left( \sqrt{2}\ t,-t,+t,-t,+1 \right),& &
f_t(\m{4}) =\frac{1}{\sqrt{1+t^2}} \left( \sqrt{2},-1,+1,-1,-t \right),\\
f_t(\p{5}) =\frac{1}{\sqrt{1+t^2}} \left( \sqrt{2}\ t,-t,+t,+t,-1 \right),& &
f_t(\m{5}) =\frac{1}{\sqrt{1+t^2}} \left( \sqrt{2},-1,+1,+1,+t \right),\\
f_t(\p{6}) =\frac{1}{\sqrt{1+t^2}} \left( \sqrt{2}\ t,-t,-t,+t,+1 \right),& &
f_t(\m{6}) =\frac{1}{\sqrt{1+t^2}} \left( \sqrt{2},-1,-1,+1,-t \right),\\
f_t(\p{7}) =\frac{1}{\sqrt{1+t^2}} \left( \sqrt{2}\ t,-t,-t,-t,-1 \right),& &
f_t(\m{7}) =\frac{1}{\sqrt{1+t^2}} \left( \sqrt{2},-1,-1,-1,+t \right),\\
f_t(\l{A}) = \left( 1,+\sqrt{2},0,0,0 \right),& &
f_t(\l{B}) = \left( 1,0,+\sqrt{2},0,0 \right),\\
f_t(\l{C}) = \left( 1,0,0,+\sqrt{2},0 \right),& &
f_t(\l{D}) = \left( 1,0,0,-\sqrt{2},0 \right),\\
f_t(\l{E}) = \left( 1,0,-\sqrt{2},0,0 \right),& &
f_t(\l{F}) = \left( 1,-\sqrt{2},0,0,0 \right).
\end{eqnarray*}
\caption{\footnotesize The list of vectors $X$, satisfying $q_1(X)=1$, in Definition \ref{defi rhot}. The representation $\rho_t$ maps each generator $\l s$ to the hyperbolic reflection in the orthogonal complement of $f_t(\l s)$.}\label{table:walls}
\end{table}

\begin{table} 
\small
\begin{eqnarray*}
f_t(\p{0}) =\frac{1}{\sqrt{1-t^2}} \left( \sqrt{2}\ t,+t,+t,+t,+1 \right) , & &
f_t(\m{0}) =\frac{1}{\sqrt{1-t^2}} \left( \sqrt{2},+1,+1,+1,+t \right),\\
f_t(\p{1}) =\frac{1}{\sqrt{1-t^2}} \left( \sqrt{2}\ t,+t,-t,+t,-1 \right),& &
f_t(\m{1}) =\frac{1}{\sqrt{1-t^2}} \left( \sqrt{2},+1,-1,+1,-t \right),\\
f_t(\p{2}) =\frac{1}{\sqrt{1-t^2}} \left( \sqrt{2}\ t,+t,-t,-t,+1 \right),& &
f_t(\m{2}) =\frac{1}{\sqrt{1-t^2}} \left( \sqrt{2},+1,-1,-1,+t \right),\\
f_t(\p{3}) =\frac{1}{\sqrt{1-t^2}} \left( \sqrt{2}\ t,+t,+t,-t,-1 \right),& &
f_t(\m{3}) =\frac{1}{\sqrt{1-t^2}} \left( \sqrt{2},+1,+1,-1,-t \right),\\
f_t(\p{4}) =\frac{1}{\sqrt{1-t^2}} \left( \sqrt{2}\ t,-t,+t,-t,+1 \right),& &
f_t(\m{4}) =\frac{1}{\sqrt{1-t^2}} \left( \sqrt{2},-1,+1,-1,+t \right),\\
f_t(\p{5}) =\frac{1}{\sqrt{1-t^2}} \left( \sqrt{2}\ t,-t,+t,+t,-1 \right),& &
f_t(\m{5}) =\frac{1}{\sqrt{1-t^2}} \left( \sqrt{2},-1,+1,+1,-t \right),\\
f_t(\p{6}) =\frac{1}{\sqrt{1-t^2}} \left( \sqrt{2}\ t,-t,-t,+t,+1 \right),& &
f_t(\m{6}) =\frac{1}{\sqrt{1-t^2}} \left( \sqrt{2},-1,-1,+1,+t \right),\\
f_t(\p{7}) =\frac{1}{\sqrt{1-t^2}} \left( \sqrt{2}\ t,-t,-t,-t,-1 \right),& &
f_t(\m{7}) =\frac{1}{\sqrt{1-t^2}} \left( \sqrt{2},-1,-1,-1,-t \right),\\
f_t(\l{A}) = \left( 1,+\sqrt{2},0,0,0 \right),& &
f_t(\l{B}) = \left( 1,0,+\sqrt{2},0,0 \right),\\
f_t(\l{C}) = \left( 1,0,0,+\sqrt{2},0 \right),& &
f_t(\l{D}) = \left( 1,0,0,-\sqrt{2},0 \right),\\
f_t(\l{E}) = \left( 1,0,-\sqrt{2},0,0 \right),& &
f_t(\l{F}) = \left( 1,-\sqrt{2},0,0,0 \right).
\end{eqnarray*}
\caption{\footnotesize The list of vectors for the definition of $\rho_t$, in the AdS case. The quadratic form $q_{-1}$ takes value $-1$ on the vectors $f_t(\p i)$, and $+1$ on the vectors $f_t(\m i)$ and $f_t(\l X)$.}\label{table:walls2}
\end{table}

Some comments are necessary to explain Definition \ref{defi rhot} and the tables involved:
 
\begin{enumerate}
\item It can be checked that all the orthogonality relations (with respect to the bilinear form $b_1$) between vectors in Table \ref{table:G22} are maintained for the vectors in Table \ref{table:walls} with respect to $b_1$, and in Table \ref{table:walls2} with respect to $b_{-1}$. This shows that Definition \ref{defi rhot} is well-posed, meaning that $\rho_t$ are representations of $\Gamma_{22}$ by Lemma \ref{lemma:map refl ads}.

\item By construction the representations $\rho_t$ are in $\mathrm{Hom}_{\mathrm{refl}}(\Gamma_{22},G)$, for $G=\Isom(\Hyp^4)$ or $\Isom(\AdS^4)$ (see Definition \ref{defi hom refl}). Tables \ref{table:walls}  and  \ref{table:walls2} exhibit continuous lifts $f_t\colon S\to \R^{110}$ as in Lemma \ref{lemma homeo model general}, taking values in a subset of $\R^{110}$ defined by the vanishing of $102$ quadratic conditions, for $S$ the standard generating set of $\Gamma_{22}$. 

\item The vectors of Table \ref{table:walls} coincide with those of Table \ref{table:G22} for $t=1$. Hence, in the hyperbolic case, the path of representations $\rho_t$ is a deformation of the reflection group of the aforementioned right-angled polytope with 22 facets. For $t\in (0,1)$, this coincides with the path of representations exhibited in \cite{KS}. For $t\in (-1,0)$, the representation $\rho_t$ is obtained by conjugating $\rho_{-t}$ by the reflection $r$ in the ``horizontal'' hyperplane $x_4=0$
\begin{equation} \label{eq:relfection singular}
r\colon(x_0,x_1,x_2,x_3,x_4)\mapsto (x_0,x_1,x_2,x_3,-x_4)~.
\end{equation}
(This is seen immediately using Remark \ref{rmk:action of isometry group}.) 

\item On the Anti-de Sitter side, the path $\rho_t$ 
has been exhibited in \cite{transition_4-manifold} for $t\in (-1,0)$. Again, the path is extended here for positive times by conjugation by $r$.

\item Both these paths occur as the holonomy representations of a deformation of hyperbolic and Anti-de Sitter cone-orbifold structures. The purpose of our previous work \cite{transition_4-manifold} was to describe the geometric transition from hyperbolic ($t>0$) to Anti-de Sitter ($t<0$) structures. Since here we are interested in the $\Isom(\Hyp^4)$- and $\Isom(\AdS^4)$-chatacter varieties on their own, we found more useful to treat the two paths $\rho_t$ separately, and extend each of them by conjugation
with the orientation-reversing transformation $r$ also for negative (resp. positive) times. 
\end{enumerate}

\subsection{The collapsed representation and the cuboctahedron}\label{subsec:cuboct}

\begin{table}
\begin{eqnarray*}
 v_{\l {0}}=& \left(\sqrt{2},+1,+1,+1 \right),&\\ 
v_{\l {1}}  =& \left(\sqrt{2},+1,-1,+1 \right),&\qquad v_{\l {A}}  = \left( 1,+\sqrt{2},0,0 \right), \\
v_{\l {2}}  = &\left(\sqrt{2},+1,-1,-1 \right),&\qquad v_{\l {B}}  = \left( 1,0,+\sqrt{2},0 \right),\\
v_{\l {3}}  = &\left(\sqrt{2},+1,+1,-1 \right),&\qquad v_{\l {C}}  = \left( 1,0,0,+\sqrt{2} \right),\\
v_{\l {4}}  = &\left(\sqrt{2},-1,+1,-1 \right),&\qquad v_{\l {D}}  = \left( 1,0,0,-\sqrt{2} \right),\\
v_{\l {5}}  = &\left(\sqrt{2},-1,+1,+1 \right),&\qquad v_{\l {E}}  = \left( 1,0,-\sqrt{2},0 \right),\\
v_{\l {6}}  = &\left(\sqrt{2},-1,-1,+1 \right),&\qquad v_{\l {F}}  = \left( 1,-\sqrt{2},0,0 \right).\\
v_{\l {7}}  = &\left(\sqrt{2},-1,-1,-1 \right),&\\
\end{eqnarray*}
\caption{\footnotesize The vectors $ v_{\l i},v_{\l X}\in \R^{1,3}$ defining the bounding planes of an ideal right-angled cuboctahedron in $\Hyp^3$. These vectors are involved also in the Definition \ref{defi HP holonomy}, introducting the cocycles $\tau_\lambda$ in the vector space $Z^1_{\varrho_0}(\Gamma_{22},\R^{1,3})$.}\label{table:vectors cocycle}
\end{table}

For $t=0$ the hyperbolic and Anti-de Sitter representations $\rho_0$ take value in the stabiliser of the hyperplane {given by} $\{x_4=0\}$. Unlike the case $t\neq 0$, these representations are not holonomies of hyperbolic/AdS orbifold structures, but correspond to what we call the \emph{collapse} of the respective geometric structures.

Let us consider $\Isom(\Hyp^4)$ and $\Isom(\AdS^4)$ as subgroups of {$\GL(5,\R)$}. 
Then the representations $\rho_0$ agree for the hyperbolic and AdS case. {
Indeed, defining
$$\Hyp^3 := \Hyp^4 \cap \{ x_4 = 0 \} = \AdS^4 \cap \{ x_4 = 0, x_0 > 0 \} \subset \R^5~,$$
its stabiliser
$$G_0=\mathrm{Stab_{\Isom(\Hyp^4)}}({\Hyp^3})=\mathrm{Stab_{\Isom(\AdS^4)}}({\Hyp^3}) {< \GL(5,\R)}$$
consists of matrices in the block form:}
$$\left(\begin{array}{ccc|c}
  &&& 0 \\
  
   & A & & \vdots \\
  &&& 0 \\
    \hline  
    0&\ldots&0   & \pm1
\end{array}
\right)~.
$$

The stabiliser $G_0$ is isomorphic to $\Isom(\Hyp^3)\times\Z/2\Z$, where the $\Z/2\Z$-factor is generated by the reflection $r$ of Equation \eqref{eq:relfection singular}, which acts by switching the two sides of $\{x_4=0\}$. Under this isomorphism, the representation $\rho_0$ reads as:

\begin{equation} \label{eq:rep rho0 hyp ads}
\begin{aligned} 
\rho_0(\p i)&=r & &\text{for each } \l i\in\{\l0,\ldots,\l7\} \\
\rho_0(\m i)&=r_{v_{\l i}} & &\text{for each } \l i\in\{\l0,\ldots,\l7\} \\
\rho_0(\l X)&=r_{v_{\l X}} & &\text{for each } \l X\in\{\l A,\ldots,\l F\}
\end{aligned}
\end{equation}
where the vectors $v_{\l i},v_{\l X}\in \R^{1,3}$, collected in Table \ref{table:vectors cocycle}, define the bounding planes $H_{v_{\l i}}$ and $H_{v_{\l X}}$ of an ideal right-angled cuboctahedron in $\Hyp^3$. The triangular faces of this cuboctahedron are of type $\l i$, while the quadrilateral faces are of type $\l X$ (see Figure \ref{fig:cuboct}). 

\begin{figure}
\includegraphics[scale=.8]{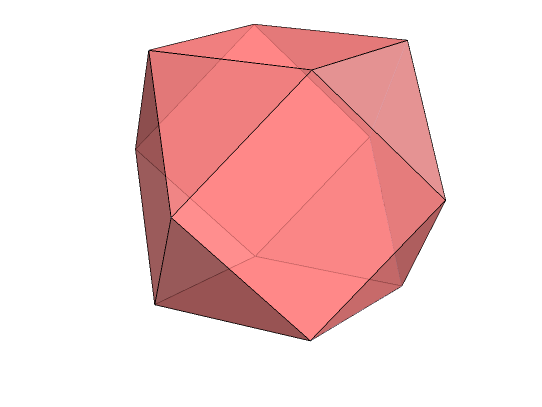}

\vspace{-.8cm}

\caption[The cuboctahedron.]{\footnotesize  
A cuboctahedron is the convex envelop of the midpoints of the edges of a regular cube (or octahedron). It is realised in $\Hyp^3$ as an ideal right-angled polytope.}\label{fig:cuboct}
\end{figure}

\subsection{The conjugacy action} \label{sec:conjugacy-action}

We are ready to start the study of the hyperbolic and AdS character varieties of $\Gamma_{22}$ near the representation $\rho_t$ introduced in Definition \ref{defi rhot}. 
%
%

We begin by analysing the action of $G$ by conjugation on $\mathrm{Hom}(\Gamma_{22},G)$. For $t\neq 0$, nearby $\rho_t$ the action of $G$ is ``good'', namely is free and proper, as we will see in the following two lemmas.

\begin{lemma} \label{lemma centraliser}
For $t\neq0$, the stabiliser of $\rho_t$ in $G$ is trivial. The stabiliser of $\rho_0$ in $G$ is the order-two subgroup generated by the reflection $r$ in the hyperplane 
$\{x_4=0\}$.
\end{lemma}

\begin{proof}
We give the proof for the hyperbolic and AdS case at the same time, since they are completely analogous. By Remark \ref{rmk:action of isometry group}, any element in the stabiliser of $\rho_t$ is induced by a matrix $A\in\O(q_{\pm 1})$ which maps every vector $f_t(\l s)$ in Table \ref{table:walls} or Table \ref{table:walls2} either to itself or to its opposite. Since the 6 vectors ${f_t(\l A)},\ldots,{f_t(\l F)}$ do not depend on $t$ and generate the hyperplane $\{x_4=0\}$, the matrix $A$ must preserve the hyperplane $\{x_4=0\}$. 

Moreover, let $\mathcal P_t$ be the polytope  bounded by the $22$ hyperplanes orthogonal to the vectors of Tables \ref{table:walls} or \ref{table:walls2}. It was proved in \cite[Proposition 3.19]{MR} and \cite[Proposition 7.21]{transition_4-manifold} that the intersection of $\mathcal P_t$  with the hyperplane defined by the equation $x_4=0$ is constant and is an ideal right-angled cuboctahedron in $\Hyp^3$ (see Section \ref{subsec:cuboct}). Since the action of $A$ on $\Hyp^3$ necessarily preserves each face of the cuboctahedron, it follows that $A$ must act  on the linear hyperplane $\{x_4=0\}$ as $\pm\mathrm{id}$. 

This shows that the only non-trivial candidates for $A$ are $\pm r$, where $r$ is the reflection of Equation \eqref{eq:relfection singular}. 
For $t=0$, the reflection $r$ preserves all the hyperplanes orthogonal to the vectors of Tables \ref{table:walls} or \ref{table:walls2}, hence the associated element in $G$ generates the stabiliser of $\rho_0$.  When $t\neq 0$, the reflection $r$ does not preserve any of the hyperplanes of the form $H_{f_t(\p i)}$ and $H_{f_t(\m i)}$, hence the stabiliser of $\rho_t$ is trivial in this case.
\end{proof}

The next lemma will be useful to show that the action of $G^+$ by conjugation is proper, in a suitable region of $\mathrm{Hom}_{\mathrm{refl}}(\Gamma_{22},G)$.

\begin{lemma} \label{lemma action proper}
Suppose that $\eta_n$ is a sequence in $\mathrm{Hom}_{\mathrm{refl}}(\Gamma_{22},G)$ converging to some $\rho_t$, and $h_n$ is a sequence in $G$ such that $h_n\cdot \eta_n$  converges. Then $h_n$ has a subsequence that converges in $G$.
\end{lemma}

\begin{proof}
Suppose that $\eta_n\to\rho_t$ and $h_n$ is a sequence in $G$ such that $h_n\cdot\eta_n\to\eta_\infty$. Since $\mathrm{Hom}_{\mathrm{refl}}(\Gamma_{22},G)$ is clopen in the representation variety, the limit point $\eta_\infty$ is in $\mathrm{Hom}_{\mathrm{refl}}(\Gamma_{22},G)$. Passing to the finite cover $g^{-1}(0)$ of Lemma \ref{lemma homeo model rep var}, and up to taking subsequences, we can then assume to have a sequence $f_n$ in $g^{-1}(0)$ (projecting to $\eta_n$) such that $f_n\to f_\infty$ and $h_n\cdot f_n\to \widehat f_\infty$. Here we are thinking of $f_n,f_\infty,\widehat f_\infty$ as functions from the standard generators of $\Gamma_{22}$ to $\R^5$, and (by a small abuse of notation) $h_n$ is a sequence in $\O(q_{\pm 1})$ acting by the obvious action on $\R^5$ (see Remark \ref{rmk:action of isometry group}).

We have to show that $h_n$ converges in $\O(q_{\pm 1})$ up to subsequences. Recall that $f_\infty$ is a lift in $g^{-1}(0)$ of $\rho_t$, and therefore (up to changes of sign) the vectors $f_\infty(\l{s})$ are given by  Table \ref{table:walls} or Table \ref{table:walls2} for some value of $t$. Take five generators $\l{s}_1,\ldots,\l{s}_5$ of $\Gamma_{22}$ such that $f_\infty(\l{s}_1),\ldots,f_\infty(\l{s}_5)$ are linearly independent, for instance $\m 0,\l A,\l B,\l C,\l D$. Since linear independence is an open condition, $\{f_n(\l{s}_1),\ldots,f_n(\l{s}_5)\}$ forms a basis of $\R^5$ for large $n$.

The linear isometry $h_n\in \O(q_{\pm 1})$, considered as a 5-by-5 matrix, is therefore determined by the condition that $h_n$ sends the basis $\{f_n(\l{s}_1),\ldots,f_n(\l{s}_5)\}$ to $\{h_n\cdot f_n(\l{s}_1),\ldots,h_n\cdot f_n(\l{s}_5)\}$. More concretely, we can write $h_n$ (as a matrix) as $(h_{n,1})^{-1}\circ h_{n,2}$, where $h_{n,1}$ is the matrix sending the standard basis to the basis $\{f_n(\l{s}_1),\ldots,f_n(\l{s}_5)\}$, and $h_{n,2}$ is the matrix sending the standard basis to the basis $\{h_n\cdot f_n(\l{s}_1),\ldots,h_n\cdot f_n(\l{s}_5)\}$. Since $f_n$ and $h_n\cdot f_n$ are converging sequences, we have that $h_{n,1} \to h_{\infty,1}$ and $h_{n,2} \to h_{\infty,2}$, and moreover $h_{\infty,1}$ is invertible since $f_\infty(\l{s}_1),\ldots,f_\infty(\l{s}_5)$ is a basis.

Therefore $h_n$ converges to a 5-by-5 matrix $h_\infty=(h_{\infty,1})^{-1}\circ h_{\infty,2}$, which is still in $\O(q_{\pm 1})$ since $\O(q_{\pm 1})$ is closed in the space of 5-by-5 matrices. This concludes the proof.
\end{proof}

{The two lemmas have some consequences on the character variety, which we now define:}

\begin{defi}[Character variety]\label{def:char_var}
Let $G$ be $\Isom(\Hyp^n)$, $G_{\HP^n}$ or $\Isom(\AdS^n)$, and $G^+$ denote its subgroup of orientation-preserving transformations. Given a finitely generated group $\Gamma$, 
{we call \emph{character variety} of $\Gamma$ in $G$ 
the quotient
$$X(\Gamma,G)=\mathrm{Hom}(\Gamma,G)/{G^+}~$$
by the action of $G^+$ by conjugation.}
\end{defi}

{
\begin{remark}\label{rmk:local product}
It follows from Lemmas \ref{lemma centraliser} and \ref{lemma action proper} that the $G^+$-action is locally a product in a neighbourhood of $\rho_t$. More precisely, there is a neighbourhood of the $G^+$-orbit of $\rho_t$ in $\mathrm{Hom}(\Gamma_{22},G)$ homeomorphic to $\mathcal{U} \times G^+$, where $\mathcal U$ is a neighbourhood of $[\rho_t]$ in $X(\Gamma_{22},G)$. Moreover the action of $G^+$ corresponds, under this homeomorphism, to left multiplication on the second factor.
\end{remark}}

\begin{remark} \label{remark quotient hausdorff}
{Let us suppose $G$ is $\Isom(\Hyp^n)$ or $\Isom(\AdS^n)$; in other words $G$ is reductive. Thanks to some well-known results from GIT (see for instance the concise exposition in \cite[Section 2]{CLMsurvey} and the references therein), the GIT quotient $\mathrm{Hom}(\Gamma_{22},G)/\!\!/G^+$ can be identified with the ``Hausdorff quotient'' of the representation variety by conjugation: namely, the quotient by $G^+$ of the subset of $\mathrm{Hom}(\Gamma,G)$ consisting of points with closed $G^+$-orbits.} 

In the portion of the character variety of our interest, no non-Hausdorff pathological situation arises. More precisely, the GIT quotient $\mathrm{Hom}(\Gamma_{22},G)/\!\!/G^+$ coincides with the ordinary topological quotient in a neighbourhood of each $[\rho_t]$. (This holds similarly for $\mathrm{Hom}(\Gamma_{22},G)/\!\!/G$.) For, it follows from Lemma \ref{lemma action proper} that:
\begin{itemize}
\item The $G^+$--action is proper on $G^+\cdot\{\rho_t\ |\ t\in (-1,1)\}$.
\item For each $t$, the $G^+$--orbit of $\rho_t$ is closed. (This follows by applying Lemma \ref{lemma action proper} to the constant sequence $\eta_n\equiv\rho_t$).
\end{itemize}

Actually the latter is true in a neighborhood of $\{\rho_t\ | \ t\in (-1,1)\}$, since in the proof of Lemma \ref{lemma action proper} we only used that, for five generators $\l{s}_1,\ldots,\l{s}_5$ of $\Gamma_{22}$, the corresponding vectors in $\R^5$ are linearly independent, and this is still true in an open neighborhood.

In fact, our argument shows a little more, namely that if $\rho$ is in such a neighbourhood, then $[\rho]$ is separated from any other point in 
$X(\Gamma_{22},G)$. This is because, if $[\rho]$ were not separated from $[\rho']$, we would have a sequence $\rho_n\to\rho$ and a sequence $h_n$ such that $h_n\rho_n h_n^{-1}$ converges to $\rho'$. But Lemma \ref{lemma action proper} shows that $h_n\to h_\infty$ up to subsequences, hence  by continuity $h_\infty$ conjugates $\rho$ and $\rho'$, namely $[\rho]=[\rho']$.
\end{remark}

\subsection{A smoothness result} \label{sec:smoothness}

In the hyperbolic case, the smoothness of the $\Isom(\Hyp^4)$-character variety near the points $[\rho_t]$ with $t \neq 0$ has been essentially proved in \cite[Theorem 12.3]{KS}:

\begin{prop}\label{prop:dimension tangent space hyp}
For $t\in (0,1)$, the set $\mathrm{Hom}(\Gamma_{22},\Isom(\Hyp^4))$ is a smooth $11$-dimensional manifold near $\rho_t$.
\end{prop}

(Recalling that for negative times $\rho_t$ is a conjugate of $\rho_{-t}$, the result holds for $t\in (-1,0)$ as well.)
Our main purpose is to extend and generalise the analysis for $t=0$, and do similarly for the $\Isom(\AdS^4)$-representation variety.

Let us first briefly sketch the lines of the proof Proposition \ref{prop:dimension tangent space hyp} given in \cite{KS}. By Lemma \ref{lemma homeo model rep var} (recall $g \colon \R^{(n+1)|S|}\to \R^{|S|+|R|}$ from the proof), it suffices to show that $g^{-1}(0)$ is a smooth submanifold of $\R^{110}$ near any preimage of $\rho_{t_0}$, for all $t_0\in (0,1)$.

Let
\begin{equation}\label{eq:f_t}
f_t\colon\{\mathrm{standard\ generators\ of\ }\Gamma_{22}\}\to\R^5
\end{equation}
be as in Table \ref{table:walls}, so giving an embedding of $(0,1)$ into $g^{-1}(0)\subset \R^{110}$ going through a preimage of $\rho_{t_0}$. The proof in \cite{KS} essentially consists in showing that the kernel of $g\colon \R^{110}\to \R^{102}$ is 11-dimensional for $t\in (0,1)$. Since there is a 10-dimensional smooth orbit given by the action of $\Isom^+(\Hyp^4)$, the proof boils down to showing that the tangent space to the orbit has a 1-dimensional complement, which is indeed given by the tangent space to the 1-dimensional submanifold $\{f_t\ |\ t\in (0,1)\}$. 

Since the action of $\Isom^+(\Hyp^4)$ is smooth, it then follows that the $\Isom^+(\Hyp^4)$-orbit of the curve $\{\rho_t\ |\ t\in (0,1)\}$ is a smooth 11-dimensional manifold, on which the $\Isom^+(\Hyp^4)$-action by conjugation is free and proper by Lemma \ref{lemma centraliser} and Lemma \ref{lemma action proper}.  Hence it follows from Proposition \ref{prop:dimension tangent space hyp} that $X(\Gamma_{22},\Isom(\Hyp^4))$ is a 1-dimensional smooth manifold near $[\rho_t]$, for $t\in (0,1)$.

In the next sections, we will prove the analogous of Proposition \ref{prop:dimension tangent space hyp} for the AdS case. However, we are interested also in the study of the character variety near ``the collapse'', that is the point $[\rho_0]$. Hence we will prove a more detailed statement.

Let $G$ be as usual $\Isom(\Hyp^4)$ or $\Isom(\AdS^4)$. 

\begin{defi}[The set $\mathrm{Hom}_{0}$] \label{defi Hom0}
We define $\mathrm{Hom}_{0}(\Gamma_{22},G)$ as the subset of $\mathrm{Hom}_{\mathrm{refl}}(\Gamma_{22},G)$ of representations $\rho$ such that the following holds. Let $\l s_1,\l s_2$ be any pair of generators of $\Gamma_{22}$ such that the hyperplanes fixed by $\rho_t(\l s_1)$ and $\rho_t(\l s_2)$ are either 
{asymptotic} or equal for some $t\neq 0$. Then, so are the hyperplanes fixed by $\rho(\l s_1)$ and $\rho(\l s_2)$.
\end{defi}

Recall from Lemmas \ref{lemma angle hyp} and \ref{lem: AdS angle} that two hyperplanes are 
{asymptotic} or equal if and only if, using the bilinear form $b_1$ for $\Hyp^4$ and $b_{-1}$ for $\AdS^4$, the product of their orthogonal unit vectors is $1$ in absolute value. It is thus easy to check from Tables \ref{table:walls} and \ref{table:walls2} that this condition is preserved by the deformation $f_t$ for all $t$ both in the hyperbolic and AdS case, and thus the definition is well-posed (i.e. it does not depend on the choice of $t\neq 0$). 

In the setting of Lemma \ref{lemma homeo model rep var}, $\mathrm{Hom}_0(\Gamma_{22},G)$ corresponds to a subset of $g^{-1}(0)\subset\R^{110}$ defined by the vanishing of 36 more quadratic conditions. Indeed, for each of the 12 {ideal vertices of the polytope $\mathcal P_t$ bounded by the hyperplanes of Tables \ref{table:walls} and \ref{table:walls2}}, we have 3 
{asymptoticity} conditions (see \cite[Proposition 7.13]{transition_4-manifold}). Hence $\mathrm{Hom}_0(\Gamma_{22},G)$ is locally homeomorphic to the zero locus of a function $g_0\colon\R^{110}\to\R^{138}$ extending $g$. More precisely:

\begin{lemma} \label{lemma homeo model rep var fixing parabolics}
The set $\mathrm{Hom}_0(\Gamma_{22},G)$ is finitely covered by a disjoint union of subsets of $\R^{110}$ defined by the vanishing of $138$ quadratic conditions.
\end{lemma}

\begin{remark} \label{rem:simmplicityAdS1}
For simplicity of exposition, from now on we will work in the AdS setting, i.e. in the case $G=\Isom(\AdS^4)$. All what follows can be easily adapted to the hyperbolic case. We will therefore omit the proofs and only highlight the points where differences with respect to the AdS case occur.
\end{remark}

The essential property we will prove is that near each of the representations $\rho_t$ the variety $\mathrm{Hom}_0(\Gamma_{22},G)$ is smooth. Hence the goal of the next two sections is to prove the following:

\begin{prop}\label{prop:dimension tangent space ads}
For $t\in(-1,1)$, the set $\mathrm{Hom}_0(\Gamma_{22},\Isom(\AdS^4))$ is a smooth $11$-dimensional manifold near ${\rho}_t$.
\end{prop}

The proof of Proposition \ref{prop:dimension tangent space ads} will be given at the end of Section \ref{subsec inf def plus minus}. From the results on cusp rigidity established in Section \ref{sec:rig dim four}, we obtain the smoothness of $\mathrm{Hom}(\Gamma_{22},\Isom(\AdS^4))$ for $t\neq0$ as a direct corollary:

\begin{cor}\label{cor:dimension tangent space ads}
For $t\in (-1,1)\smallsetminus \{0\}$, the set $\mathrm{Hom}(\Gamma_{22},\Isom(\AdS^4))$ is a smooth $11$-dimensional manifold near $\rho_t$.
\end{cor}

\begin{proof}
It is not difficult to check that, when $t\neq 0$, for every pair of generators $\l s_1,\l s_2$ of $\Gamma_{22}$ such that the associated hyperplanes $H_{f_t(\l s_1)}$ and $H_{f_t(\l s_2)}$ are 
{asymptotic}, there are 4 other generators $\l s_3,\ldots,\l s_6$ such that the reflections $r_{\l s_1},\ldots,r_{\l s_6}$ generate a cusp group in $\Isom(\AdS^4)$. By Lemma \ref{prop cube group ads}, the 
{asymptoticity} conditions are preserved since cusp groups stay cusp groups under small deformations. Hence a neighbourhood of $\rho_t$ in $\mathrm{Hom}(\Gamma_{22},\Isom(\AdS^4))$ is actually contained in $\mathrm{Hom}_0(\Gamma_{22},\Isom(\AdS^4))$. The proof now follows from Proposition \ref{prop:dimension tangent space ads}.
\end{proof}

The next sections will be devoted to the proof of Proposition \ref{prop:dimension tangent space ads}. We will adapt some of the ideas of \cite[Sections 5, 11, 12]{KS} used in the proof of Proposition \ref{prop:dimension tangent space hyp} in the hyperbolic case. An analogous argument shows that the statement of Proposition \ref{prop:dimension tangent space ads} holds also for the $\Hyp^4$-character variety, which for $t=0$ is new with respect to the results of \cite{KS}. 

\subsection{Infinitesimal deformations of the letter generators}

Recall Lemma \ref{lemma homeo model rep var fixing parabolics}. Throughout this and the following sections, we denote by
$$g_0\colon \R^{110}\to \R^{138}$$
the quadratic function defining the clopen subset of $\mathrm{Hom}_0(\Gamma_{22},\Isom(\AdS^4))$ that contains the lifts of the representations $\rho_t$. A continuous lift of the path $t\mapsto{\rho}_t$ is defined by $f_t$ in Table \ref{table:walls2}. {To prove Proposition \ref{prop:dimension tangent space ads} in the AdS case, i}t then suffices to show that for all $t\in(-1,1)$ the set $g_0^{-1}(0)\subset\R^{110}$ is a smooth 11-dimensional manifold near each $f_t$. 


\begin{notation*}
Let us fix $t\in(-1,1)$. For simplicity, by an abuse of notation, in this and next section we denote $f_t(\l s)\in\R^5$ by $\l s$. In other words, in what follows $\l s\in\R^5$ denotes a vector (of $q_{-1}$-norm $1$ or $-1$ depending whether the corresponding hyperplane in $\AdS^4$ is timelike or spacelike, respectively) from Table \ref{table:walls2}, and is therefore implicitly considered as a function of $t$. Its derivative in $t$ will be denoted by $\ld s$. The symbol $(\{{\p i}\},\{{\m j}\},\{{\l X}\})$ will denote the corresponding element of $g_0^{-1}(0)\subset\R^{110}$, as a function from the standard generators of $\Gamma_{22}$ to $\R^5$, while $(\{{\pd i}\},\{{\md j}\},\{{\ld X}\})$ will denote an element in the kernel of the differential of $g_0$ at $(\{{\p i}\},\{{\m j}\},\{{\l X}\})$, and will be called an \emph{infinitesimal deformation} of $(\{{\p i}\},\{{\m j}\},\{{\l X}\})$.
\end{notation*}

Observe that the vectors $\l A,\ldots,\l F$ of Table \ref{table:walls2} are constant in $t$, hence the derivative of the path in {$g_0^{-1}(0)$} provided by Table \ref{table:walls2} satisfies $\ld X=0$ for all $\l X\in\{\l A,\ldots,\l F\}$.

By Remark \ref{rmk:action of isometry group}, the natural $\O(q_{-1})$--action on $g_0^{-1}(0)$ is given by $\l s\mapsto A\cdot \l s$ for $A\in\O(q_{-1})$. Therefore the tangent space to the orbit of an element $(\{{\p i}\},\{{\m j}\},\{{\l X}\})$ of $g_0^{-1}(0)$ consists precisely of the elements of the kernel of $dg_0$ of the form
\begin{equation}\label{eq:topolino}
\l s\mapsto \ld s=\mathfrak a\cdot \l s~,
\end{equation}
where $\l s$ varies in $(\{{\p i}\},\{{\m j}\},\{{\l X}\})$ and $\mathfrak a=\left.\frac{d}{dt}\right|_{t=0}A_t\in\so(q_{-1})$, for any smooth path $t \mapsto A_t$ in $\O(q_{-1})$ with $A_0=\mathrm{id}$.

The first step in the proof of Proposition \ref{prop:dimension tangent space ads} is to show that, up to this infinitesimal action, we can assume that \emph{any} infinitesimal deformation vanishes at least on 4 elements of $\{\l A,\ldots,\l F\}$.

\begin{lemma} \label{lemma:inf action four letters vanish}
Fix $t\in(-1,1)$, and let $(\{{\pd i}\},\{{\md j}\},\{{\ld X}\})$ be an infinitesimal deformation of $(\{{\p i}\},\{{\m j}\},\{{\l X}\})$. Up to the action of $\mathfrak a\in\so(q_{-1})$ as in \eqref{eq:topolino}, we can assume that
\begin{equation} \label{tex willer}
\ld A=\ld B=\ld C=\ld D=0~,
\end{equation}
and that
\begin{equation} \label{kit carson}
\ld E=(0,0,0,0,\epsilon)\qquad\text{and}\qquad\ld F=(0,0,0,0,\phi)
\end{equation}
for some $\epsilon,\phi\in\R$.
\end{lemma}

The analogous lemma in the hyperbolic case, for $t\neq 0$, has been proved in \cite[Proposition 11.1]{KS}, and in fact the arguments here follow roughly the same lines as their proof. However, the first part of their proof uses a nice geometric argument which would be complicated to adapt to AdS geometry. For this reason, we rather use a linear algebra argument here. 

\begin{notation*}
To simplify the notation, from here to the end of Section \ref{subsec inf def plus minus}, we denote by $\langle\cdot,\cdot\rangle$ the bilinear form $b_{-1}$. If one wants to repeat the proof for $G=\Isom(\Hyp^4)$, then $\langle\cdot,\cdot\rangle$ 
should denote $b_1$. The reader should pay attention that in Section \ref{sec:gp cohomology} the bracket $\langle\cdot,\cdot\rangle$ will instead be used to denote the Minkowski bilinear form on $\R^4$.
\end{notation*}

\begin{proof}
The proof will follow from three claims.

First we claim that we can assume $\ld A=\ld B=\ld C=0$. Equivalently,
given any infinitesimal deformation $(\{\pd i\},\{\md j\},\{\ld X\})$, we want to show that there exists $\mathfrak a\in\so(q_{-1})$ such that
\begin{equation} \label{eq:dot ABC vanish}
\mathfrak a\cdot \l A=\ld A~,\qquad \mathfrak a\cdot \l B=\ld B~,\qquad \mathfrak a\cdot \l C=\ld C~.
\end{equation}
Indeed, {if \eqref{eq:dot ABC vanish} is true,} we can then subtract to $(\{\pd i\},\{\md j\},\{\ld X\})$ the element in the tangent space to the orbit of the form \eqref{eq:topolino} (i.e. given by $\ld s=\mathfrak a\cdot \l s$) and obtain a new infinitesimal deformation for which $\ld A=\ld B=\ld C=0$.

To show the first claim, consider the basis $\{\l A,\l B,\l C,\l D,e_4\}$ of $\R^5$, where $e_4=(0,0,0,0,1)$. Recall that matrices $\mathfrak a$ in the Lie algebra $\so(q_{-1})$ are characterised by the condition that
\begin{equation} \label{eq:characterise lie algebra}
\langle \mathfrak a\cdot u,w\rangle+\langle u,\mathfrak a\cdot w\rangle=0
\end{equation}
for every $u,w$, and that it suffices in fact to check the condition for all pair of elements $u,w$ of our fixed basis. Moreover, to define the matrix $\mathfrak a$ in $\so(q_{-1})$, it suffices to define it on 4 vectors of the basis of $\R^5$, such that \eqref{eq:characterise lie algebra} holds when $u,w$  are chosen among these 4 vectors. The definition of $\mathfrak a$ on the last vector of the basis is then uniquely determined by \eqref{eq:characterise lie algebra}. 

Let us now apply these preliminary remarks. By differentiating the conditions 
$$\langle \l A,\l A\rangle=\langle \l B,\l B\rangle=\langle \l C,\l C\rangle=1$$
we obtain
\begin{equation}\label{eq: compatibility lie algebra 1}
\langle \l A,\ld A\rangle=\langle \l B,\ld B\rangle=\langle \l C,\ld C\rangle=0~.
\end{equation}
By differentiating the 
{asymptoticity} conditions
$$\langle \l A,\l B\rangle=\langle \l A,\l C\rangle=\langle \l B,\l C\rangle=-1$$
we get the conditions
\begin{equation}\label{eq: compatibility lie algebra 2}
\langle \l A,\ld B\rangle+\langle \ld A,\l B\rangle=0,\quad\langle \l A,\ld C\rangle+\langle \ld A,\l C\rangle=0,\quad\langle \l B,\ld C\rangle+\langle \ld B,\l C\rangle=0~.
\end{equation}
Equations \eqref{eq: compatibility lie algebra 1} and \eqref{eq: compatibility lie algebra 2} show that any linear transformation $\mathfrak a \in \mathfrak{so}(q_{-1})$ sending $\l A$ to $\ld A$, $\l B$ to $\ld B$ and $\l C$ to $\ld C$ satisfies the conditions of \eqref{eq:characterise lie algebra} for all pairs of $u,w$ chosen in $\{\l A,\l B,\l C\}$. {It remains to define $\mathfrak{a}$ on the remaining two elements $\l D$ and $e_4$ of the basis. Equation \eqref{eq:characterise lie algebra} imposes the value of $\langle \mathfrak{a} \cdot \l D, u  \rangle$ and $\langle \mathfrak{a} \cdot e_4, u  \rangle$ for all $u \in \{ \l A, \l B, \l C \}$. Moreover we must have $\langle \mathfrak{a} \cdot \l D, \l D \rangle = \langle \mathfrak{a} \cdot e_4, e_4 \rangle = 0$. Therefore $\mathfrak{a}\cdot \l D$ and $\mathfrak{a}\cdot e_4$ can be chosen with one degree of freedom given by the value of $\langle \mathfrak{a} \cdot \l D, e_4 \rangle = - \langle \l D, \mathfrak{a} \cdot e_4  \rangle$.} This shows that we can find $\mathfrak a \in \mathfrak{so}(q_{-1})$ satisfying Equation \eqref{eq:dot ABC vanish}, and our first claim is proved.

Second, we claim that we can further assume that
$$\langle\ld D,e_4\rangle=0~.$$
To see this second claim, by repeating the same reasoning as in the beginning of this proof, it suffices to find another $\mathfrak a'\in\so(q_{-1})$ so that 
\begin{equation} \label{eq:dot ABC vanish 2}
\mathfrak a'\cdot \l A= \mathfrak a'\cdot \l B= \mathfrak a'\cdot \l C=0\qquad\text{and}\qquad \mathfrak a'\cdot \l D=\langle \ld D,e_4\rangle e_4~.
\end{equation}
Indeed, if \eqref{eq:dot ABC vanish 2} holds then the conditions \eqref{eq:characterise lie algebra} are satisfied for $u,w$ chosen in $\{\l A,\l B,\l C,\l D\}$, and  we have already remarked that $\mathfrak a'\cdot e_4$ will then be uniquely determined by \eqref{eq:characterise lie algebra} in such a way that  
$\mathfrak a'\in \so(q_{-1})$. This shows our second claim.

Finally we claim that, under the above assumptions, necessarily $\ld D=0$, $\ld E=(0,0,0,0,\epsilon)$ and $\ld F=(0,0,0,0,\phi)$. This part of the proof follows closely \cite[Proposition 11.1]{KS}.

As observed in the proof of Corollary \ref{cor:dimension tangent space ads}, since $\langle \l A, \l D \rangle = -1$, the vectors $\l A$ and $\l D$ play the role of two non-commuting generators (reflections along two timelike hyperplanes that are 
{asymptotic}) of a cusp group generated  by the images of $\l A$, $\l D$, $\p3$, $\m3$, $\p2$, and $\m2$. By the assumption that 
{asymptoticity} conditions are preserved (recall that we are in $\mathrm{Hom}_0$), any deformation of $\l A$ and $\l D$ satisfies $\langle \l A,\l D\rangle=-1$. So, by differentiating and using $\ld A=0$, we obtain $\langle \l A,\ld D\rangle=0$. Analogously, $\langle \l B,\ld D\rangle=0$.

Together with $\langle \l D,\ld D\rangle=0$ (which follows from $\langle \l D,\l D\rangle=1$) and the assumption $\langle\ld D,v\rangle=0$, we have necessarily
$$\ld D=(\sqrt 2 \delta,\delta,\delta,-\delta,0)$$
for some $\delta$. Similarly for $\ld E$, using that $\langle \l A,\ld E\rangle=\langle \l C,\ld E\rangle=\langle \l E,\ld E\rangle=0$, we find
$$\ld E=(\sqrt 2 \epsilon',\epsilon',-\epsilon',\epsilon',\epsilon)~.$$
For $\l F$, from $\langle \l B,\ld F\rangle=\langle \l C,\ld F\rangle=\langle \l F,\ld F\rangle=0$ we find
$$\ld F=(\sqrt 2 \phi',-\phi',\phi',\phi',\phi)~.$$
Now using that $\l D$ and $\l E$ remain 
{asymptotic}, and similarly for  the pairs $\{ \l D, \l F \}$ and $\{ \l E, \l F \}$, we have the relations 
$$\langle \l D,\ld E\rangle+\langle \ld D,\l E\rangle=0~,\quad\langle \l D,\ld F\rangle+\langle \ld D,\l F\rangle=0~,\quad\langle \l E,\ld F\rangle+\langle \ld E,\l F\rangle=0~,$$
which read as:
$$2\sqrt 2\,\delta+2\sqrt 2\,\epsilon'=0~,\qquad 2\sqrt 2\,\delta+2\sqrt 2\,\phi'=0~,\qquad 2\sqrt 2\,\epsilon'+2\sqrt 2\,\phi'=0~.$$
Hence $\delta=\epsilon'=\phi'=0$, and this 
shows the claim. The proof of Lemma \ref{lemma:inf action four letters vanish} is complete.
\end{proof}

\subsection{Infinitesimal deformations of the positive and negative generators}\label{subsec inf def plus minus}

We conclude in this section the proof of Proposition \ref{prop:dimension tangent space ads}.

A direct computation from Table \ref{table:walls2} shows that the tangent vector to our explicit path {$f_t$ in $g_0^{-1}(0)$} is given by:
\begin{equation} \label{eq: inf variation walls explicit}
\pd i=\lambda\m i \qquad \md i=\lambda \p i\qquad \ld X=0
\end{equation}
where
$$\lambda=\frac{1}{(1-t^2)^{3/2}}\,,$$
for all $\l i\in\{\l0,\ldots,\l7\}$ and $\l X\in\{\l A,\ldots,\l F\}$. (In the hyperbolic case, from Table \ref{table:walls}, one would instead obtain $\pd i=\lambda\m i$, $\md i=-\lambda \p i$ and $\ld X=0$ for $\lambda=(1+t^2)^{-3/2}$.)

We shall now show that, under the assumption in the statement of Lemma \ref{lemma:inf action four letters vanish}, \emph{every} infinitesimal deformation $(\{\pd i\},\{\md j\},\{\ld X\})$ of $(\{\p i\},\{\m j\},\{\l X\})$ satisfies \eqref{eq: inf variation walls explicit} for some $\lambda$. Again, the proof follows roughly the lines of \cite[Section 12]{KS}, with the necessary adaptations to the AdS setting, and some simplifications. 

\begin{lemma} \label{lemma:inf variation 0 and 3}
Fix $t\in(-1,1)$, and let $(\{{\pd i}\},\{{\md j}\},\{{\ld X}\})$ be an infinitesimal deformation of the  normalised vectors $(\{{\p i}\},\{{\m j}\},\{{\l X}\})$ satisfying \eqref{tex willer} and \eqref{kit carson}. Then
\begin{equation}\label{eq:inf variation 0 3}
\begin{aligned}
\pd 0=\lambda \m 0 &\qquad  \md 0=\lambda \p 0 \\
\pd 3=\lambda \m 3 &\qquad  \md 3=\lambda \p 3
\end{aligned}
\end{equation}
for some $\lambda\in\R$ (depending on $t$). 
\end{lemma}
\begin{proof}
Using the assumptions $\ld A=\ld B=\ld C=0$, the derivatives of the relations $\langle \p 0,\l A\rangle=\langle \p 0,\l B\rangle=\langle \p 0,\l C\rangle=0$  yield 
\begin{equation}\label{eq:paperinik}
\langle \pd 0,\l A\rangle=\langle \pd 0,\l B\rangle=\langle \pd 0,\l C\rangle=0~.
\end{equation}
Together with 
\begin{equation}\label{eq:superpippo}
\langle \pd 0,\p 0\rangle=0~,
\end{equation}
we obtain $\pd 0=\lambda_0^+\m 0$ for some $\lambda_0^+$.

Indeed, the vectors $\l A,\l B,\l C$ and $\p 0$ are linearly independent, and $\m 0$ satisfies all the four linear conditions  \eqref{eq:paperinik} and \eqref{eq:superpippo}, hence $\m 0$ spans the space of solutions. Similarly for $\m 0$, we obtain $\md 0=\lambda_0^-\p 0$, and repeating the same argument for $\p 3$ and $\m 3$ (replacing the role of $\l C$ by $\l D$) we find $\pd 3=\lambda_3^+\m 3$ and $\md 3=\lambda_3^-\p 3$.

Now, differentiating the relation $\langle\p 0,\m 0\rangle=0$, we get
$$0=\langle\pd 0,\m 0\rangle+\langle\p 0,\md 0\rangle=\lambda_0^+\langle\m 0,\m 0\rangle+\lambda_0^-\langle\p 0,\p 0\rangle=\lambda_0^+-\lambda_0^-$$
which implies $\lambda_0^+=\lambda_0^-$. Similarly we have $\lambda_3^+=\lambda_3^-$. Finally by differentiating $\langle\p 3,\m 0\rangle=0$ we find
$$0=\langle\pd 3,\m 0\rangle+\langle\p 3,\md 0\rangle=\lambda_3^+\langle\m 3,\m 0\rangle+\lambda_0^-\langle\p 3,\p 0\rangle=\lambda_0^--\lambda_3^+$$
whence $\lambda_0^-=\lambda_3^+$. This concludes the proof.
\end{proof}

We remark that in the hyperbolic case the same computation shows that $\pd 0=\lambda \m 0$, $\md 0=-\lambda \p 0$, 
$\pd 3=\lambda \m 3$ and $\md 3=-\lambda \p 3$ for some $\lambda\in\R$, as the only differences with respect to the AdS argument is that $\langle\p 0,\p 0\rangle=\langle\p 3,\p 3\rangle=1$ and $\langle\p 3,\p 0\rangle=-1$ from Table \ref{table:walls}.

So, using the assumption $\ld A=\ld B=\ld C=\ld D=0$, we have proved that \eqref{eq: inf variation walls explicit} holds for $\p i\in\{\p0,\p3\}$ and $\m i\in\{\m0,\m3\}$. If we knew that $\ld E=\ld F=0$, we could repeat a similar argument to show that \eqref{eq: inf variation walls explicit} holds also for the remaining $\l i^\pm$'s. It thus essentially remains to show that $\ld E=\ld F=0$.

\begin{lemma} \label{lemma:inf variation E F vanish}
Fix $t\in(-1,1)$, and let $(\{{\pd i}\},\{{\md j}\},\{{\ld X}\})$ be an infinitesimal deformation of the  normalised vectors $(\{{\p i}\},\{{\m j}\},\{{\l X}\})$ satisfying \eqref{tex willer} and \eqref{kit carson}. Then
$\ld E=\ld F=0$
and there exists $\lambda\in\R$ such that, for every $\l i\in\{\l 0,\ldots,\l 7\}$,
$$\pd i=\lambda \m i \quad \mbox{and} \quad  \md i=\lambda \p i~.$$
\end{lemma}
\begin{proof}
Let $\lambda\in\R$ be as in the conclusion of Lemma \ref{lemma:inf variation 0 and 3}. Let us first focus on the variations of $\l 1$ and $\l 2$, similarly to the proof of Lemma \ref{lemma:inf variation 0 and 3}.  Taking the derivatives of the relations $\langle \p 1,\l A\rangle=\langle \p 1,\l C\rangle=0$ and using $\ld A=\ld C=0$, we have
$$\langle \pd 1,\l A\rangle=\langle \pd 1,\l C\rangle=0$$
whereas from $\langle \p 1,\p 1\rangle=-1$ we derive 
$$\langle \pd 1,\p 1\rangle=0~.$$

Here we do not know $\ld E=0$ yet, hence we cannot argue that $\langle \pd 1,\l E\rangle=0$, which would imply that $\pd 1$ is a multiple of $\m 1$. However,  observing that $\l A,\l C$ and $\p 1$ are linearly independent, and that the linear system for the $\pd 1$ given by the above three conditions is satisfied by the vectors $\m 0$ and $\m 1$ (which are linearly independent), by a dimension argument  $\pd 1$ is necessarily a linear combination of $\m 0$ and $\m 1$. Analogously one gets that $\md 1$ is necessarily a linear combination of $\p 0$ and $\p 1$, and replacing $\l C$ by $\l D$, and $\l 0$ by $\l 3$, we find similar relations for $\md 2$ and $\pd 2$. Let us summarise them here:
\begin{align*}
\md 1=\lambda_1^-\p 1+\mu_1^-\p 0~,\\
\pd 1=\lambda_1^+\m 1+\mu_1^+\m 0~,\\
\md 2=\lambda_2^-\p 2+\mu_2^-\p 3~,\\
\pd 2=\lambda_2^+\m 2+\mu_2^+\m 3~.
\end{align*}

We claim here that, as expected from \eqref{eq: inf variation walls explicit}, $\lambda_1^-=\lambda_1^+=\lambda_2^-=\lambda_2^+=\lambda$ and $\mu_1^-=\mu_1^+=\mu_2^-=\mu_2^+=0$, and moreover  $\ld E=0$. In fact, it will suffice to show $\mu_1^+=0$. 
Indeed, recalling the assumption $\ld E=(0,0,0,0,\epsilon)$, the derivative of the relation $\langle \l E,\p 1\rangle=0$ gives 
\begin{equation}\label{eq:derivative E1+}
0=\langle \ld E,\p 1\rangle+\langle \l E,\pd 1\rangle=\frac{1}{\sqrt{1-t^2}}(\epsilon-2\sqrt 2\mu_1^+)~,
\end{equation}
hence we will obtain $\epsilon=0$, namely $\ld E=0$. Once we have $\ld E=0$, we can proceed exactly as in Lemma \ref{lemma:inf variation 0 and 3} to deduce that $\mu_1^-=\mu_1^+=\mu_2^-=\mu_2^+=0$ and then $\lambda_1^-=\lambda_1^+=\lambda_2^-=\lambda_2^+=\lambda$ (which also follows from Equation \eqref{eq:ziopaperone} below). 


We shall need one more intermediate step. By differentiating the relation $\langle \m 0,\p 1\rangle=0$, we find
$$0=\langle \md 0,\p 1\rangle+\langle \m 0,\pd 1\rangle=\lambda\langle \p 0,\p 1\rangle+\lambda_1^+\langle \m 0,\m 1\rangle+\mu_1^+\langle \m 0,\m 0\rangle=\lambda-\lambda_1^++\mu_1^+~.$$
Using similarly  the relations $\langle \p 0,\m 1\rangle=\langle \m 2,\p 3\rangle=\langle \p 2,\m 3\rangle=0$ we find three analogous identities. We summarise these four important identities here:
\begin{equation}\label{eq:ziopaperone}
\lambda=\lambda_1^--\mu_1^-=\lambda_1^+-\mu_1^+=\lambda_2^--\mu_2^-=\lambda_2^+-\mu_2^+~.
\end{equation}

We can now focus on proving that $\mu_1^+=0$. Differentiating $\langle \p 1,\m 2\rangle=0$ we see that
$$0=\lambda_1^+\langle \m 1,\m 2\rangle +\mu_1^+\langle \m 0,\m 2\rangle+\lambda_2^-\langle \p 1,\p 2\rangle +\mu_2^-\langle \p1,\p 3\rangle~.$$
Using $\langle \m 1,\m 2\rangle=-1$, $\langle \p 1,\p 2\rangle=1$ and an explicit computation for the other two terms, we obtain:
$$\lambda_2^--\lambda_1^+=\frac{3+t^2}{1-t^2}\mu_1^++\frac{1+3t^2}{1-t^2}\mu_2^-~.$$
On the other hand, from Equation \eqref{eq:ziopaperone} we have $\lambda_2^--\lambda_1^+=\mu_2^--\mu_1^+$, whence
\begin{equation}\label{eq:t depending condition}
\mu_1^++{t^2}\mu_2^-=0~.
\end{equation}
If $t=0$, we are done. Otherwise, we will combine \eqref{eq:t depending condition} with the derivative of the relation 
$\langle \l E,\m 2\rangle=0$, namely
$$\frac{t}{\sqrt{1-t^2}}(-\epsilon-2\sqrt 2\mu_2^-)=0~,$$
which together with Equation \eqref{eq:derivative E1+} gives $\mu_1^++\mu_2^-=0$. 
Together with \eqref{eq:t depending condition}, this shows that $\mu_1^+=0$.

Having proved that $\ld E=0$, the proof that $\ld F=0$ follows exactly the same lines, with $\l 4$ and $\l 5$ playing the role of $\l 1$ and $\l 2$. Arguing as in Lemma \ref{lemma:inf variation 0 and 3} one then shows that $\pd i=\lambda \m i$ and $\md i=\lambda \p i$ for all $\l i\in\{\l1,\ldots,\l7\}$.
\end{proof}

This provides the conclusion of the proof of Proposition \ref{prop:dimension tangent space ads}.

\begin{proof}[Proof of Proposition \ref{prop:dimension tangent space ads}]
Let us fix $t\in(-1,1)$. We now show that the kernel of the differential of $g_0\colon \R^{110}\to \R^{138}$ is 11-dimensional at $(\{\p i\},\{\m j\},\{\l X\})\in g_0^{-1}(0)$. 

Lemmas \ref{lemma:inf action four letters vanish}, \ref{lemma:inf variation 0 and 3} and \ref{lemma:inf variation E F vanish} showed that every element in the kernel of $dg_0$ is of the form \eqref{eq: inf variation walls explicit} up to adding an element of the form \eqref{eq:topolino}, that is an element tangent to the orbit of the $\Isom(\AdS^4)$-action. It is also easy to see that such element in the tangent space of the orbit is unique, for if two elements $\mathfrak a_1$ and $\mathfrak a_2$ have this property, it follows that $\mathfrak a:=\mathfrak a_1-\mathfrak a_2$ satisfies $\mathfrak a\cdot \l X=0$ for $\l X=\l A,\l B,\l C,\l D$ and the characterising conditions \eqref{eq:characterise lie algebra} (already used in Lemma \ref{lemma:inf action four letters vanish}) show that $\mathfrak a=0$. The very same argument shows that the map defined in \eqref{eq:topolino} from the Lie algebra $\mathfrak{isom}(\AdS^4)$ into the kernel of the differential of $g_0$ (whose image is the tangent space to the orbit of the $\Isom(\AdS^4)$-action) is injective.

In other words, the 10-dimensional tangent space of the orbit has a 1-dimensional complement, consisting precisely of the elements of the form \eqref{eq: inf variation walls explicit}, hence the kernel of the differential of $g_0$ has dimension 11. By the constant rank theorem, $g_0^{-1}(0)$ is a manifold of dimension 11 near the elements in the orbit of $\rho_t$.
\end{proof}

\subsection{Topology of the neighbourhood $\mathcal{U}$} \label{sec:proof teoB ads}

We now state a weaker version of Theorem \ref{teo:main}:
\begin{theorem} \label{teo:main_weak}
Let $G$ be $\Isom(\Hyp^4)$, $\Isom(\AdS^4)$, or $G_{\HP^4}$. Then $[\rho_0]$ has a neighbourhood $\mathcal U=\mathcal V\cup\mathcal H$ in $X(\Gamma_{22},{G})$ homeomorphic to $\mathcal S=\{(x_1^2+\ldots+x_{12}^2)\cdot x_{13}=0\}\subset\R^{13}$, so that:
\begin{itemize}
\item $[\rho_0]$ corresponds to the origin;
\item $\mathcal V$ corresponds to the $x_{13}$-axis, and 
consists of the conjugacy classes of the holonomy representations $\rho_t^G$;
\item $\mathcal H$ corresponds to $\{x_{13}=0\}$, identified to a neighbourhood of the complete hyperbolic orbifold structure of the ideal right-angled cuboctahedron in its deformation space. 
\end{itemize}
The group $G/G^+\cong\Z/2\Z$ acts on $\mathcal S$ by changing sign to the last coordinate $x_{13}$.
\end{theorem} 

This statement is weaker than Theorem \ref{teo:main} because it gives a purely topological description {of $\mathcal{U}$}, while the smoothness and transversality of 
its components will be proved in Section \ref{sec:extra}.

We prove here Theorem \ref{teo:main_weak} in the Anti-de Sitter case. The proof in the hyperbolic case is completely analogous, so we omit it, while the proof in the HP case will be given later in Section \ref{sec: teo B HP}.
%
%
We decided to give a proof only in the AdS case, since the fact that the points $[\rho_t]$ for $t>0$ form a smooth curve (Proposition \ref{prop:dimension tangent space hyp}) has already been proved in \cite{KS}, while its AdS counterpart is completely new. The proof for the hyperbolic case is analogous (recall Remark \ref{rem:simmplicityAdS1}). Moreover, the description of the collapse (namely, at the representation $\rho_0$) is also new in both (hyperbolic and AdS) cases.

\begin{proof}[Proof of Theorem \ref{teo:main_weak} --- AdS case]

We split the proof into several steps. 

\begin{steps}
\item As a first step, let us define $\widetilde{\mathcal V}\subset \mathrm{Hom}(\Gamma_{22},\Isom(\AdS^4))$  as the $\Isom^+(\AdS^4)$--orbit of the curve $\{{\rho}_t\}_{{t\in(-1,1)}}$. Let us also observe that $\widetilde{\mathcal V}$ is contained in the subset $\mathrm{Hom}_0(\Gamma_{22},\Isom(\AdS^4))$ introduced in Definition \ref{defi Hom0}.

Since by Lemma \ref{lemma action proper} the $\Isom^+(\AdS^4)$--action by conjugation is free on $\{{\rho}_t\}_{{t\in(-1,1)}}$, the map $(g,t)\mapsto g\cdot{\rho}_t$ defines a continuous injection
$$\Isom^+(\AdS^4)\times(-1,1)\to\mathrm{Hom}_0(\Gamma_{22},\Isom(\AdS^4))~,$$
where by Proposition \ref{prop:dimension tangent space ads} the latter is a smooth 11-dimensional manifold. By the invariance of domain, this injection is a homeomorphism onto its image, which is $\widetilde{\mathcal V}$. By Lemma \ref{lemma centraliser} and Lemma \ref{lemma action proper}, the $\Isom^+(\AdS^4)$--action by conjugation is free and proper on $\widetilde{\mathcal V}$ thus the projection in the quotient $X(\Gamma_{22},\Isom(\AdS^4))$ is  
$${\mathcal V}:=\{[{\rho}_t]\ | \ t\in{(-1,1)}\}~,$$
which is homeomorphic to a line.

\item The second component ${\mathcal H}$ is defined as follows.

Recall from Section \ref{subsec:cuboct} that we have {a fixed 
copy $\Hyp^3\subset\AdS^4$ defined by 
$x_4 = 0$ and $x_0 > 0$, 
fixed by the reflection $r$. 
Its stabiliser 
$G_0$ is 
$\Isom(\Hyp^3)\times \langle r\rangle$, where 
we consider $\Isom(\Hyp^3)$ as a subgroup of $\Isom(\AdS^4)$.}

Consider the reflection group $\Gamma_{\mathrm{co}}$ of the ideal right-angled cuboctahedron. We define the map 
$$\Psi\colon\mathrm{Hom}(\Gamma_{\mathrm{co}},\Isom(\Hyp^3))\to\mathrm{Hom}(\Gamma_{22},\Isom(\AdS^4))$$ that associates to  $\eta\colon\Gamma_{\mathrm{co}}\to\Isom(\Hyp^3)$ the representation $\Psi_\eta\colon\Gamma_{22}\to \Isom(\AdS^4)$ sending each of the generators $\p0,\ldots,\p7$ of $\Gamma_{22}$ to the reflection $r$, and each of the generators $\m0,\ldots,\m7,\l A,\ldots,\l F$ to the corresponding element of $\Isom(\Hyp^3)<\Isom(\AdS^4)$ through $\eta$. It is then straightforward to check that:
\begin{enumerate}
\item The map $\Psi$ is well-defined and equivariant for the conjugacy action of $\Isom(\Hyp^3)<\Isom(\AdS^4)$.
\item The following induced map 
$$\widehat \Psi \colon \bar{X}(\Gamma_{\mathrm{co}},\Isom(\Hyp^3)) \to X(\Gamma_{22},\Isom(\AdS^4))~,$$
where $\bar{X}(\Gamma_{\mathrm{co}},\Isom(\Hyp^3)) := \mathrm{Hom}(\Gamma_{\mathrm{co}},\Isom(\Hyp^3)) / \Isom(\Hyp^3)$, is injective.
\end{enumerate}

Indeed, (1) holds because, using that $r$ commutes with the elements of $\Isom(\Hyp^3)<\Isom(\AdS^4)$, the images of the generators in $\Isom(\AdS^4)$ through $\Psi_\eta$ satisfy the relations of $\Gamma_{22}$, so that $\Psi_\eta$ is indeed a representation of $\Gamma_{22}$ in $\Isom(\AdS^4)$. The equivariance of $\Psi$ is clear using again that $r$ commutes with $\Isom(\Hyp^3)$. It also follows that $\widehat\Psi$ is well defined. Inded, by the equivariance of $\Psi$, if $\eta_1$ and $\eta_2$ are conjugate in $\Isom(\Hyp^3)$ then $\Psi_{\eta_1}$ and $\Psi_{\eta_1}$ are conjugate in $\Isom(\AdS^4)$. Up to composing with $r$, which commutes with both $\Psi_{\eta_1}$ and $\Psi_{\eta_2}$, then the latter are conjugate in $\Isom^+(\AdS^4)$.

Moreover, (2) holds because if two representations $\Psi_{\eta_1}$ and $\Psi_{\eta_2}$ in the image of $\Psi$ are conjugate by some $g\in\Isom^+(\AdS^4)$, then, since $\Psi_{\eta_1}(\p i)=\Psi_{\eta_2}(\p i)=r$, the isometry $g$ must fix 
$\Hyp^3\subset\AdS^4$, and therefore $g\in G_0 \cong \Isom(\Hyp^3)\times \langle r\rangle$. Moreover, up to composing with $r$, which commutes with both $\Psi_{\eta_i}$, we can also assume that $g$ belongs to the subgroup $\Isom(\Hyp^3)<\Isom(\AdS^4)$, hence $\eta_1$ and $\eta_2$ are conjugate in $\Isom(\Hyp^3)$. 

The set 
$\bar X(\Gamma_{\mathrm{co}},\Isom(\Hyp^3))$ is a 12-dimensional manifold in a neighborhood (say $\mathcal H_0$) of $[\eta_0]$, since it corresponds to a neighbourhood of the complete hyperbolic orbifold structure of the right-angled cuboctahedron in its deformation space. To show this, the same proofs of \cite[Proposition 5.2]{KS} apply (see also the related discussion in \cite[Section 5]{KS}), as a well-known ``reflective'' orbifold version of Thurston's hyperbolic Dehn filling \cite{thurstonnotes} (note that the ideal cuboctahedron has 12 cusps). 

Therefore a neighborhood $\mathcal H_0$ of $[\eta_0]$ 
$\bar X(\Gamma_{\mathrm{co}},\Isom(\Hyp^3))$ is homeomorphic to $\R^{12}$, and we can also assume that $\widehat\Psi|_{\mathcal H_0}$ is a homeomorphism onto its image. Then let us define ${\mathcal H}:=\widehat\Psi(\mathcal H_0)$.

\item We claim that the intersection of ${\mathcal H}$ and ${\mathcal V}$ consists only of the point $[\rho_0]$.

Indeed, suppose $[\rho]\in {\mathcal H}\cap{\mathcal V}$, for $\rho$ in $\mathrm{Hom}(\Gamma_{22},\Isom(\AdS^4))$. On the one hand $\rho=\Psi(\eta)$, where $\eta\in \mathrm{Hom}(\Gamma_{\mathrm{co}},\Isom(\Hyp^3))$ is a deformation  of the orbifold fundamental group of the cuboctahedron. On the other hand $\rho$ lies in  $\widetilde{\mathcal V}\subset\mathrm{Hom_0}(\Gamma_{22},\Isom(\AdS^4))$. In particular, $\eta$ maps each peripheral subgroup of $\Gamma_{\mathrm{co}}$ to a cusp group.

By the Mostow--Prasad rigidity, $\eta$ is conjugate to the holonomy representation $\eta_0$ of the complete right-angled ideal cuboctahedron. Since both $\rho$ and $\rho_0$ send each of the generators $\p0,\ldots,\p7$ to $r$, which commutes with $\Isom(\Hyp^3)<\Isom(\AdS^4)$, the representations $\rho$ and $\rho_0$ are also conjugate in $\Isom(\Hyp^3)$, and therefore $[\rho]=[\rho_0]$.

\item Let us now show that the point $[\rho_0]\in X(\Gamma_{22},\Isom(\AdS^4))$ has a neighbourhood ${\mathcal U}$ which is contained in the union of the two components ${\mathcal V}$ and ${\mathcal H}$.

To see this, let $\rho$ be a representation nearby $\rho_0$. We claim that if two generators  which are sent by $\rho_0$ to the same reflection $r$ (hence necessarily of the form $\p i$ and $\p j$) are sent to reflections in coinciding hyperplanes also by $\rho$, then all generators $\p 0,\ldots,\p 7$ are sent by $\rho$ to the same reflection. That is, if $\rho(\p i)=\rho(\p j)$ for some $\l i,\l j$, then $\rho(\p i)=\rho(\p j)$ for all $\l i,\l j$. This will show our thesis by the rigidity property of Proposition \ref{prop cube group ads collapsed}: if $[\rho]$ is not on the ``horizontal'' component $\mathcal H$, then no two letter generators are sent to the same reflection, and thus all the collapsed cusp groups of $\rho_0$ are cusp groups for $\rho$. That is, $\rho$ lies in $\mathrm{Hom}_0(\Gamma_{22},\Isom(\AdS^4))$ and thus in the ``vertical'' component $\widetilde{\mathcal V}$, since $\widetilde{\mathcal V}$ is open in $\mathrm{Hom}_0(\Gamma_{22},\Isom(\AdS^4))$.

To prove the claim, suppose that two generators $\p i$ and $\p j$ are such that $\rho(\p i)=\rho(\p j)$. By the symmetries of the polytope $\mathcal P_t$ (see 
\cite[Lemma 7.6]{transition_4-manifold}) and Proposition \ref{prop cube group ads collapsed}, we can assume the two generators are $\p 0$ and $\p 1$. Up to conjugation in $\Isom(\AdS^4)$, we can also assume $\rho(\p 0)=\rho(\p 1)=r$. To simplify the notation, let $f$ be a preimage of $\rho$ in $g^{-1}(0)$, which associates to each generator of $\Gamma_{22}$ a vector in $\R^5$ of square norm $1$ or $-1$ with respect to $q_{-1}$.

Up to changing the sign if necessary, $f(\p 0)=f(\p 1)=e_4=(0,0,0,0,1)$. From the relations in $\Gamma_{22}$, the vector $f(\p 2)$  is necessarily orthogonal to $f(\m 1)$, $f(\m 2)$, $f(\m 3)$ and $f(\l A)$. But by the assumption $f(\p 0)=f(\p 1)=e_4$ and the relations involving $\p 0$, the vector $e_4$ is orthogonal to $f(\m 1)$, $f(\m 3)$ and $f(\l A)$, while from the relations involving $\p 1$, the vector $e_4$ is orthogonal to $f(\m 2)$.

For a small deformation of $\rho_0$, the vectors  $f(\m 1)$, $f(\m 2)$, $f(\m 3)$ and $f(\l A)$ are linearly independent, because they are for $\rho_0$ (see Table \ref{table:walls2}). Hence the conditions of being orthogonal to these 4 vectors define a linear system of 4 independent equations, which are satisfied by $e_4$. Hence $f(\p 2)$, which is a solution of the system, coincides with $e_4$ up to rescaling. Since $q(f(\p 2))=-1$, we can assume that $f(\p 2)=e_4$.  Namely, $\rho(\p 2)=r$. By arguing similarly for $\p 3$ and then for all the other generators, one easily finds sufficiently many relations to show that $\rho(\p i)=r$ for each generator $\p i \in \{ \p0, \ldots, \p7 \}$, and therefore $\rho$ is in $\mathcal H$. This proves the claim.
 
\item Summarising the previous steps, we have shown that the class $[\rho_0]$ has a neighborhood ${\mathcal U}$ 
which only consists of points of ${\mathcal H}$ and ${\mathcal V}$. Since we already know that ${\mathcal H}$ and ${\mathcal V}$ are smooth manifolds outside of $\rho_0$, it is harmless to enlarge ${\mathcal U}$ so that it contains entirely ${\mathcal H}$ and ${\mathcal V}$.

We have therefore obtained a neighborhood ${\mathcal U}$ of $[\rho_0]$ in $X(\Gamma_{22},\Isom(\AdS^4))$ homeomorphic to
$$(\{0\}\times\R)\ \cup\ (\R^{12}\times\{0\})\ \subset\ \R^{13}~,$$
where the two components are precisely ${\mathcal H}$ and ${\mathcal V}$.

\item It remains to prove the last sentence about the action of the group $$\Isom(\AdS^4)/\Isom^+(\AdS^4)\cong\Z/2\Z$$ generated by the coset of the reflection $r$.

This is now simple: on the one hand, as observed after Definition \ref{defi rhot}, conjugation by $r$ acts on $\mathcal V$, which is homeomorphic to $(-1,1)$, by $[{\rho}_t]\mapsto [{\rho}_{-t}]$. On the other hand, by construction of $\mathcal H$, conjugation by $r$ fixes pointwise the elements in $\mathcal H$, which are of the form $\Psi_\eta$ for some $\eta\colon\Gamma_{\mathrm{co}}\to\Isom(\Hyp^3)$.
This concludes the proof.
\qedhere
\end{steps}
\end{proof}

We conclude the section with a couple of observations on the nature of the fixed points for the action of $G$ on $\mathrm{Hom}(\Gamma_{22},G)$.

Lemma \ref{lemma centraliser} shows that the stabiliser of each point ${\rho}_t$ in $\mathrm{Hom}(\Gamma_{22},G)$, for the conjugacy action of $G$,  is trivial, except $\rho_0$ which has stabiliser $\langle r\rangle$. In fact, a small adaptation of the proof shows that, in a neighborhood  of $\rho_0$, the stabiliser of all points in the the horizontal component $\mathcal H$ is as well the group $\Z/2\Z$ generated by $r$. This is because we can find a neighborhood of $\rho$ is in the image of $\Psi$ such that, for a lift $f$ of $\rho$, the vectors $f(\l A),f(\l B),f(\l C),f(\l D)\in\R^5$ are linearly independent. Indeed the vectors $f_0(\l A),f_0(\l B),f_0(\l C),f_0(\l D)$ are linearly independent, and being independent is an open condition. By the structure of the group $\Gamma_{22}$, the vectors $f(\l A),f(\l B),f(\l C),f(\l D)$ are necessarily orthogonal to $(0,0,0,0,1)$, since $\rho$ maps each generator $\p i$ to $r$. Hence one can repeat the proof of Lemma \ref{lemma centraliser} and see that an element in the stabiliser of $\rho$ must necessarily fix $\{x_4=0\}$ setwise, and moreover must act trivially on $\{x_4=0\}$. Hence the only possible candidates are the identity and $r$, both of which fix $\rho$ by definition of $\Psi$.

In conclusion, let us consider the full quotient $\mathrm{Hom}(\Gamma_{22},G)/G$, which is a $\Z/2\Z$--quotient of $X(\Gamma_{22},G)$, where $\Z/2\Z\cong G/G^+$. A local picture of this full quotient is given in Figure \ref{fig:rep_var} (right), as a consequence of the fact that the generator of $\Z/2\Z$ acts by changing sign to the $x_{13}$-coordinate, hence as a ``reflection'' with respect to the horizontal component $\mathcal H$. The ``horizontal'' component (which is the projection of $\mathcal H$ to the full quotient $\mathrm{Hom}(\Gamma_{22},G)/G$) entirely consists of points with associated group $\Z/2\Z$. They are ``double'' points in a suitable sense, which reminds ``mirror'' points in the language of orbifolds.

\section{Reflections and cusp groups in HP geometry} \label{sec:cusp_groups_HP}

In this section we introduce half-pipe geometry, discuss its relations with Minkowski geometry, and prove the half-pipe version of the flexibility and rigidity statements for right-angled cusp groups.

\subsection{Half-pipe geometry}

Let us denote by $q_0$ the following degenerate bilinear form on $\R^{n+1}$:
$$q_{0}(x) = -x_0^2 + x_1^2 + \ldots+x_{n-1}^2~.$$
Then \emph{half-pipe space} of dimension $n$ is defined as

$$\HP^n  = {\{x\in\R^{n+1}\,|\,q_0(x)= -1\, , \, x_0>0\}}~,$$
and the group of \emph{half-pipe transformations} is
$$ G_{\HP^n}  = {\mathsf \{ A \in \O(q_0) \, | \, A(\HP^n)=\HP^n\, , \, A e_n = \pm e_n \}}
$$
(here $e_0, \ldots, e_n$ is the canonical basis of $\R^{n+1}$). 

Explicitly, an element $A \in G_{\HP^n}$ has the form
\begin{equation*} 
A = \left(
\begin{array}{ccc|c}
  &&& 0 \\
  
   & \widehat A & & \vdots \\
  &&& 0 \\
    \hline  
    \star & \ldots & \star & \pm1
\end{array}
\right)
\end{equation*}
for some $n$-by-$n$ matrix $\widehat A$ which preserves the bilinear form of signature $(-,+,\ldots,+)$ on $\R^n$ {
the upper sheet $\Hyp^{n-1}$ of the hyperboloid $\{ q_0 = -1 \} \subset \R^n \times \{ 0 \}$}, where the stars denote the entries of any vector in $\R^n$. Hence there is an obvious epimorphism $G_{\HP^n}\to\Isom(\Hyp^{n-1})$, given by $A\mapsto \widehat A$. {Despite an inequality is involved in the definition of $G_{\HP^n}$, this group is naturally an algebraic Lie group via the isomorphism between $G_{\HP^n}$ and  $\O(1,n-1) \ltimes \R^n$ of Lemma \ref{lemma isomorphism}, and the fact that the latter group can be defined as an algebraic subgroup of $\mathrm{Aff}(\R^n) \subset \GL(n+1,\R)$.}


The \emph{boundary at infinity} of $\HP^n$ is 
$$\partial\HP^n  = {\{x\in\R^{n+1}\,|\,q_0(x)=0\}/\R^*}~,$$
and can be visualised as the union of a cylinder constituted by those $[x]\in \partial\HP^n$ such that $(x_0,\ldots,x_{n-1})$ does not vanish, and the point 
at infinity $[e_n]\in \partial\HP^n$. The latter is a distinguished point, since it is preserved by the action of every element of $ G_{\HP^n}$ on $\partial\HP^n$.

{As usual, we consider $\HP^n \cup \overline{\HP}{}^n$ as a subset of $\RP^n$, the \emph{ideal closure} of a 
subset $A \subset \HP^n$ {that is closed in $\HP^n$} is its closure $\overline{A}$ in 
$\RP^n$, 
and we have $\overline{\HP}{}^n = \HP^n \cup \partial \HP^n$.}

There is a natural map from $\HP^n$ to {$\{x\in\HP^n\,:\, x_n=0\}$}, which is a copy of $\Hyp^{n-1}$, given simply by $(x_0,\ldots,x_n)\mapsto(x_0,\ldots,x_{n-1},0)$. We shall call this map the \emph{projection}
$$\pi\colon\HP^n\to\Hyp^{n-1}~.$$
The map $\pi$ is equivariant with respect to the obvious epimorphism $G_{\HP^n}\to\Isom(\Hyp^{n-1})$, and extends to a map {
$\overline{\pi} \colon \overline{\HP}{}^n \smallsetminus \{ [e_n] \} \to \overline{\Hyp}{}^{n-1}$}. 
The fibres of $\overline{\pi}$ are called \emph{degenerate lines}, since they 
extend to projective lines in $\RP^n$ 
by adding the point $[e_n]$ at infinity, and the restriction of the bilinear form $b_0$ associated to $q_0$ is degenerate. Degenerate lines are preserved by the action of $ G_{\HP^n}$.


\subsection{Duality with Minkowski space}

We will find comfortable to exploit the well-known duality between half-pipe and Minkowski geometry. We will not provide details of the proofs here, see \cite{barbotfillastre,surveyseppifillastre,transition_4-manifold} for a more complete treatment.

The fundamental observation is that $\HP^n$ is identified to the space of spacelike affine hyperplanes in Minkowski space $\R^{1,n-1}:=(\R^{n},q_1)$
where $q_1$ is the non-degenerate bilinear form on $\R^n$ introduced in Section \ref{subsec:defi hyp ads}. The correspondence is given by associating to a point $x\in\HP^n$ the affine hyperplane of $\R^{1,n}$ defined by the equation 
\begin{equation}\label{eq:dual correspondence}
b_1((x_0,\ldots,x_{n-1}),(y_0,\ldots,y_{n-1}))+x_n=0~,
\end{equation}
for $b_1$ the bilinear form associated to $q_1$. 

The isometry group $\Isom(\R^{1,n-1})\cong \O(q_1)\ltimes\R^n$ acts naturally on the space of spacelike affine hyperplanes, and the correspondence is also well-behaved with respect to the group actions, as 
we summarise in the following lemma (see for instance \cite{barbotfillastre}, \cite{surveyseppifillastre} or \cite[Lemma 2.8]{transition_4-manifold}).

\begin{lemma}\label{lemma isomorphism}
There is a ``duality'' homeomorphism
$$\{\text{spacelike affine hyperplanes in }\R^{1,n-1}\}\cong\HP^n$$
which is equivariant with respect to a group isomorphism
$$\phi \colon \Isom(\R^{1,n-1}) \to G_{\HP^n}~.$$
\end{lemma}

In this work, we will adopt almost entirely this ``dual'' point of view for half-pipe geometry. In this setting, the boundary $\partial\HP^n$ has a natural identification:
\begin{equation}\label{eq:decomposition bdy}
\partial\HP^n\cong \{\text{lightlike affine hyperplanes in }\R^{1,n-1}\}\cup\{\infty\}~,
\end{equation}
where the point $[e_n]$ in $\partial\HP^n$ corresponds to $\infty$ on the right-hand side, while $\partial\HP^n\smallsetminus\{[e_n]\}$ identifies to the space of lightlike affine hyperplanes using again \eqref{eq:dual correspondence}. Geometrically, the decomposition in the right-hand side of \eqref{eq:decomposition bdy} reflects the fact that, up to taking a subsequence, a sequence of spacelike affine hyperplanes in $\R^{1,n-1}$ may either converge to a lightlike hyperplane or escape from all compact subsets. 

The projection $\pi$ is interpreted in this dual setting as the map which associates to a spacelike affine hyperplane in $\R^{1,n-1}$  its unique parallel linear hyperplane. Equivalently, thinking of $\pi$ with values in $\Hyp^{n-1}$, it associates to a spacelike affine hyperplane its normal direction with respect to the Minkowski product $b_1$. 
Of course $\pi$ extends to the complement of $\infty$ in $\partial\HP^n$, with values in $\partial\Hyp^{n-1}$. 


  

\subsection{Hyperplanes} \label{sec:hyperplanesHP}

Let us now consider hyperplanes in half-pipe geometry. 

\begin{defi}[HP hyperplane]
A \emph{half-pipe hyperplane} is the intersection of $\HP^n$ with a {linear hyperplane in $\R^{n+1}$}. It is called \emph{degenerate} if it contains a degenerate line of $\HP^n$; \emph{non-degenerate} otherwise.
\end{defi}


From now on, we will always think of $\HP^n$ dually as the space of spacelike affine hyperplanes in $\R^{1,n-1}$, using Lemma \ref{lemma isomorphism}. For more details on the proofs of the following statements, see \cite[Section 4.3]{transition_4-manifold}.

\begin{lemma}\label{lemma hyperplane hp}
Any non-degenerate hyperplane of $\HP^n$ is dual to the set of spacelike affine hyperplanes going through a given point $p\in \R^{1,n-1}$.
\end{lemma}

We will refer to the point $p$ as the \emph{dual} point to the non-degenerate hyperplane, and conversely we will make reference to the hyperplane \emph{dual} to a point of $\R^{1,n-1}$. With this duality approach, it is very easy to describe the relative position of non-degenerate hyperplanes:

\begin{lemma}\label{lemma intersection half pipe}
Given two points $p,q\in \R^{1,n-1}$, their dual hyperplanes
\begin{itemize}
\item intersect in $\HP^n$ if and only if $p-q$ is spacelike,
\item are disjoint in $\HP^n$ but their ideal closures intersect in $\partial\HP^n$ if and only if $p-q$ is lightlike,
\item have disjoint ideal closures in $\overline{\HP}^n$ if and only if $p-q$ is timelike.
\end{itemize}
\end{lemma}

In half-pipe geometry, the situation for degenerate and non-degenerate hyperplanes is qualitatively different, as we shall see also in Section \ref{subsec:reflHP} below. Let us first characterise degenerate hyperplanes in terms of Minkowski geometry:

\begin{lemma}\label{lemma hyperplane hp deg}
Any degenerate hyperplane of $\HP^n$ is the preimage of a hyperplane in $\Hyp^{n-1}$ by the projection map $\pi\colon\HP^n\to\Hyp^{n-1}$. That is, it is dual to the set of spacelike affine hyperplanes having normal direction in a given hyperplane of $\Hyp^{n-1}$. 
\end{lemma}

There are three possibilities for the relative position of two degenerate hyperplanes {$H_1 = \pi^{-1}(S_1)$ and $H_2 = \pi^{-1}(S_2)$} in $\HP^n$:
\begin{itemize}
\item If $S_1$ and $S_2$ intersect in $\Hyp^{n-1}$, then 
$H_1$ and $H_2$ intersect in $\HP^n$ in the subset $\pi^{-1}(S_1\cap S_2)$;
\item If {$\overline{S}_1$ and $\overline{S}_2$ intersect in $\partial\Hyp^{n-1}$}, then 
{$\overline{H}_1$ and $\overline{H}_2$} intersect in a degenerate line of $\partial \HP^n$;
\item If {$\overline{S}_1$ and $\overline{S}_2$ are} disjoint in $\overline{\Hyp}^{n-1}$, then 
{$\overline{H}_1$ and $\overline{H}_2$} only intersect in $\infty\in\partial\HP^n$.
\end{itemize}


{Asymptoticity of two hyperplanes and asymptoticity of a hyperplane to a point at infinity are defined similarly to the hyperbolic and AdS settings (see Section \ref{sec:relative}).}

\begin{figure}
\includegraphics[scale=0.5]{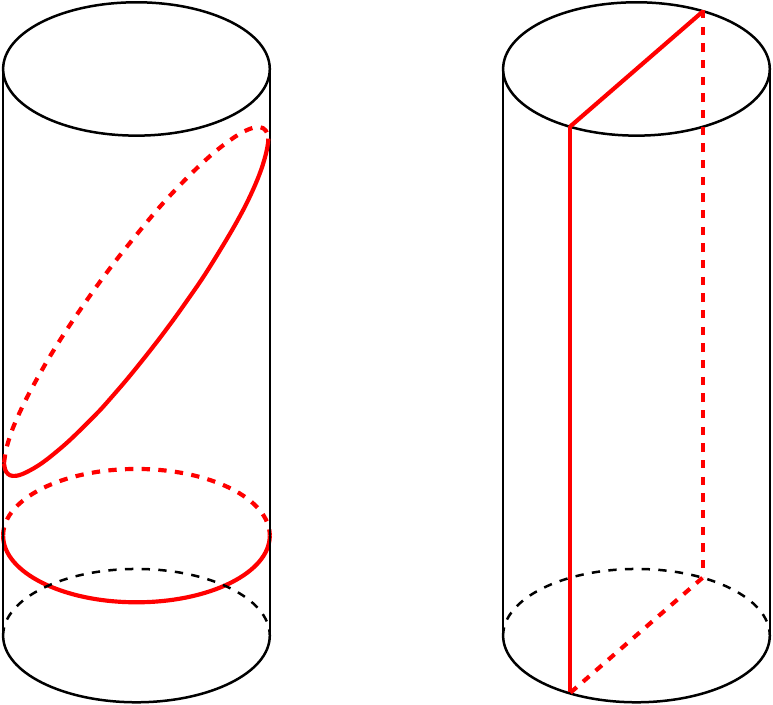}
\caption[Hyperplanes in half-pipe space.]{\footnotesize Hyperplanes in the affine (cylindric) model of $\HP^n$: on the left, two spacelike hyperplanes, on the right, a degenerate hyperplane.}
\label{fig:HPplanes}
\end{figure}

\subsection{Reflections}\label{subsec:reflHP}

Like in pseudo-Riemannian geometry, a \emph{reflection} in $\HP^n$ is a non-trivial involution in $G_{\HP^n}$ that fixes pointwise a hyperplane.

We shall again distinguish two cases:

\begin{prop}\label{prop:nondeg hyperplane HP}
There exists a unique reflection in $G_{\HP^n}$ fixing a given non-degenerate hyperplane in $\HP^n$.
\end{prop}

\begin{proof}
By Lemma \ref{lemma hyperplane hp}, a reflection in $G_{\HP^n}$ is induced by an element of $\Isom(\R^{1,n-1})$ that fixes setwise all the spacelike hyperplanes going through a point $p\in \R^{1,n-1}$. The involution $\phi(-\mathrm{id},2p)$ therefore has such a property. It is the only reflection fixing the hyperplane dual to $p$. Indeed, for a transformation $\phi(A,v)$ with this property,  the linear part $A$ must fix all the timelike directions in  $\R^{1,n-1}$, hence $A=\pm \mathrm{id}$, but the choice $A=\mathrm{id}$ implies necessarily $v=0$ because $\phi(A,v)$ has order two, and therefore gives  a trivial transformation. 
\end{proof}

Let us now consider degenerate hyperplanes:
\begin{prop}\label{prop:deg hyperplane HP}
There exists a one-parameter family of reflections in $G_{\HP^n}$ fixing a given degenerate hyperplane in $\HP^n$.
\end{prop}
\begin{proof}
From Lemma \ref{lemma hyperplane hp deg}, a degenerate hyperplane in $\HP^n$ has the form $\pi^{-1}(H_X)$ where, using the notation of Section \ref{sec:hyperplanes}, $X$ denotes a vector in $\R^{1,n-1}$ such that $q_1(X)=1$ and $H_X$ is the hyperplane in $\Hyp^{n-1}$ induced by the orthogonal complement $X^\perpp$. Any reflection in $G_{\HP^n}$ fixing $\pi^{-1}(H_X)$ pointwise must be of the form $\phi(A,v)$ where the linear part $A$ fixes $X^\perpp$ pointwise. Hence the only possible candidates for $A$ are the identity and the Minkowski reflection in $H_X$, which we denote by $r_X$. Since $(A,v)$ is assumed to be an involution, $A=\mathrm{id}$ only gives the trivial transformation (i.e. $v=0$). On the other hand, imposing the involutive condition for the choice $A=r_X$ we obtain the reflections $\phi(r_X,v)$ for any $v\in\mathrm{Span}(X)$.  These are indeed reflections in the half-pipe hyperplane $\pi^{-1}(H_X)$, since they fix setwise all spacelike hyperplanes of $\R^{1,n-1}$ with normal direction in $H_X$. 
\end{proof}

Finally, it is necessary to analyse conditions which assure that two reflections commute. From Proposition \ref{prop:nondeg hyperplane HP}, it is clear that two reflections $\phi(-\mathrm{id},2p)$ and $\phi(-\mathrm{id},2q)$ in non-degenerate hyperplanes do not commute unless $p=q$, i.e. unless the hyperplanes of reflection coincide.

By Proposition \ref{prop:deg hyperplane HP}, reflections in degenerate hyperplanes are induced by Minkowski reflections in timelike hyperplanes. Hence two reflections $\phi(r_{X_1},v_1)$ and $\phi(r_{X_2},v_2)$ commute if and only if their linear parts commute.

The remaining case is considered in the following lemma, which is straightforward:

\begin{lemma} \label{lemma:commutation conditions}
Let $v,w,X$ be vectors in $\R^{1,n-1}$, with $q_1(X)=1$ and $v\in\mathrm{Span}(X)$. The Minkowski isometries $(r_{X},v)$ and $(-\mathrm{id},w)$ commute if and only if $w=v+u$ with $u\in X^\perpp$. 
\end{lemma}
\begin{proof}
An easy computation shows that $(r_{X},v)$ and $(-\mathrm{id},w)$ commute if and only if 
\begin{equation}\label{eq:commuting conditions}
(\mathrm{id}-r_X)(w)=2v~.
\end{equation} 
Writing $w=\lambda X+u$ for $\lambda\in\R$ and $u\in X^\perpp$, we have $r_X(w)=-\lambda X+u$, hence the condition \eqref{eq:commuting conditions} is equivalent to $\lambda X=v$. 
\end{proof}

\subsection{Right-angled cusp groups} \label{subsec:HP cusps}

Let us now discuss the properties of flexibility and rigidity of cusp representations for half-pipe geometry, similarly to what we did for hyperbolic and AdS geometry in Section \ref{sec:cusp_rigidity}. The statements will be completely analogous, but the proofs simpler than their AdS (and hyperbolic) counterparts above.

The definitions of cusp groups and collapsed cusp groups are  parallel to the AdS case:

\begin{defi}[Cusp groups for $\HP^3$]
The image of a representation of $\Gamma_{\mathrm{rect}}$ into $G_{\HP^3}$ is called:
\begin{itemize}
\item a \emph{cusp group} if the four generators are sent to reflections in four distinct planes 
{asymptotic to a common} point in $\partial\HP^3$; 
\item a \emph{collapsed cusp group} if the four generators are sent to reflections along three distinct planes, two degenerate and one non-degenerate, 
{asymptotic to a common} point in $\partial\HP^3$.
\end{itemize}
\end{defi}

It follows from the discussion of the previous section that a cusp group representation must necessarily map two generators corresponding to opposite sides of the rectangle to reflections in degenerate hyperplanes, and the other two generators to reflections in non-degenerate hyperplanes.

The following example describes the structure of a (possibly collapsed) cusp group in HP geometry. By the non-uniqueness of half-pipe reflections in a degenerate plane (Proposition \ref{prop:deg hyperplane HP}), we need to describe not only the planes fixed by the reflections associated to each generators, but also the reflections themselves.

\begin{example} \label{ex:cusp group HP}
Let the image of $\rho\colon\Gamma_{\mathrm{rect}}\to G_{\HP^3}$ be a cusp group or collapsed cusp group, let ${\l s}_1, {\l s}_2$ be the generators such that $\rho({\l s}_1),\rho({\l s}_2)$ are reflections in a non-degenerate plane, and ${\l t}_1, {\l t}_2$ those such that $\rho({\l t}_1),\rho({\l t}_2)$ are reflections in a degenerate plane. 
Up to conjugacy, we can assume that $\rho({\l s}_1)=\phi(-\mathrm{id},0)$, that is, $\rho({\l s}_1)$ is the unique reflection in the dual plane to the origin of $\R^{1,2}$. 

Using Lemma \ref{lemma:commutation conditions}, $\rho({\l t}_1)$ and $\rho({\l t}_2)$ are necessarily of the form $\phi(r_{X_i},0)$, for $X_i$ a unit spacelike vector in $\R^{1,2}$. This means that the two degenerate planes fixed by $\rho({\l t}_i)$ are of the form $\pi^{-1}(H_{X_i})$, for $i=1,2$. Since the {ideal closures of the} four planes are assumed to meet in a single point in $\partial\HP^3$, necessarily the {ideal closure of the} geodesics $H_{X_1}$ and $H_{X_2}$ of $\Hyp^2$ meet in $\partial\Hyp^2$. This means that $X_1^\perpp\cap X_2^\perpp$ is a lightlike line in $\R^{1,2}$. 

Finally, by Lemma \ref{lemma:commutation conditions} $\rho({\l s}_2)$ must be of the form $\phi(-\mathrm{id},w)$ for some $w\in X_1^\perpp\cap X_2^\perpp $. This means that the non-degenerate plane fixed by $\rho({\l s}_2)$ is the dual of the point $w/2\in\R^{1,2}$. If $w=0$, then we have a collapsed cusp group, otherwise a cusp group. See Figure \ref{fig:HPcuspgroups} on the left.
\end{example}

\begin{figure}
\includegraphics[scale=0.65]{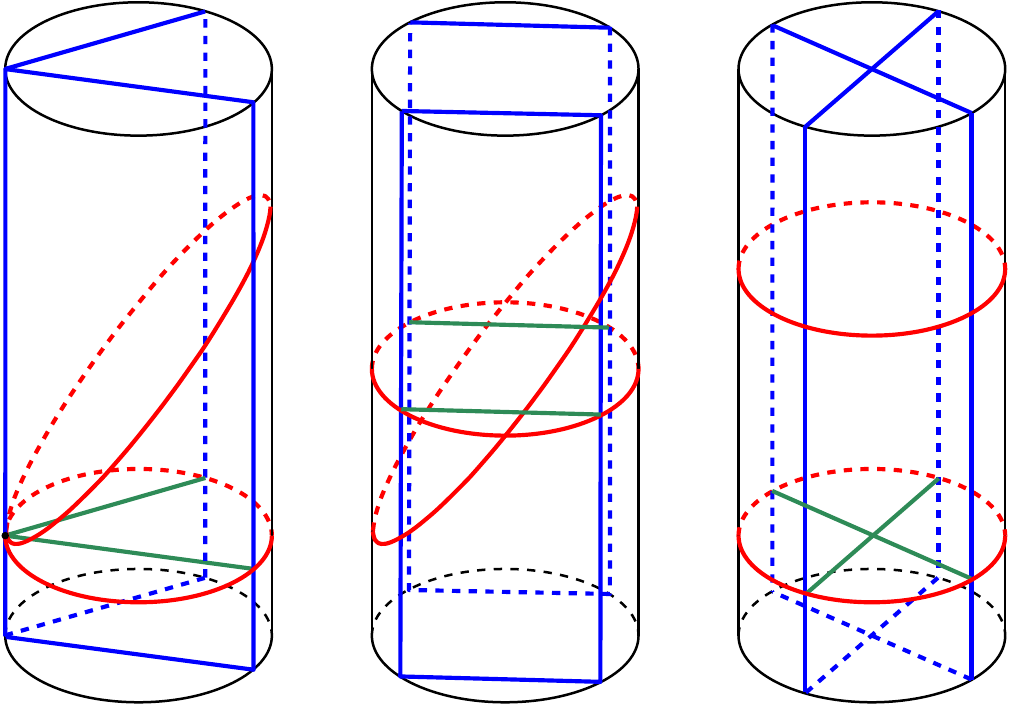}
\caption[Deforming cusp groups in ${G}_{\HP^3}$.]{\footnotesize Three possibilities for a representation of $\Gamma_{\mathrm{rect}}$ in ${G}_{\HP^3}$, as in the proof of Proposition \ref{prop rect group HP}. In red, two non-degenerate planes, in blue two degenerate planes, and in green their intersections, which are geodesics in a copy of $\Hyp^2$. On the left, the green geodesics are 
{asymptotic} and we have a cusp group. In the middle, they are ultraparallel, so the degenerate blue planes of $\HP^3$ are disjoint, while the non-degenerate red planes intersect. On the right, the green geodesics intersect, so do the blue (degenerate) planes, while the red (non-degenerate) planes are disjoint.}
\label{fig:HPcuspgroups}
\end{figure} 


Let us now prove the HP analogue of Propositions \ref{prop rect group hyp} and \ref{prop rect group ads} (see Figure \ref{fig:HPcuspgroups}).

\begin{prop} \label{prop rect group HP}
Let $\rho\colon \Gamma_{\mathrm{rect}}\to G_{\HP^3}$ be a representation whose image is a cusp group or a collapsed cusp group. For all nearby representations $\rho'$, exactly one of the following possibilities holds:
\begin{enumerate}
\item If ${\l s}_1$ and ${\l s}_2$ are generators such that $\rho({\l s}_1)=\rho({\l s}_2)$, then $\rho'({\l s}_1)=\rho'({\l s}_2)$. 
\item The image of $\rho'$ is a cusp group. 
\item A pair of opposite planes intersect in $\HP^3$, while the other pair of opposite planes have disjoint ideal closures in $\overline{\HP}^3$. 
\end{enumerate}
\end{prop}

\begin{proof}
Let $\rho'\colon \Gamma_{\mathrm{rect}}\to G_{\HP^3}$ be a representation nearby $\rho$. As in Example \ref{ex:cusp group HP}, we can assume that the reflection associated to one of the generators ${\l s}_1$ of $\Gamma_{\mathrm{rect}}$ is $\rho'({\l s}_1)=\phi(-\mathrm{id},0)$, so that its fixed plane is the dual plane to the origin of $\R^{1,2}$. Repeating the argument of Example \ref{ex:cusp group HP}, we have $\rho'({\l t}_i)=\phi(r_{X_i},0)$ for some unit spacelike vectors $X_i$, and $\rho'({\l s}_2)=\phi(-\mathrm{id},w)$ for some $w\in X_1^\perpp\cap X_2^\perpp$. 



If $w=0$, we are in case (1). Let us therefore assume $w\neq 0$.
If the {ideal closures of the} geodesics $H_{X_1}$ and $H_{X_2}$ intersect in $\partial\Hyp^2$, then $X_1^\perpp\cap X_2^\perpp$ is a lightlike geodesic, hence the image of $\rho'$ is a cusp group as in Example \ref{ex:cusp group HP} and we are in case (2).

If $H_{X_1}$ and $H_{X_2}$ intersect in $\Hyp^2$, then $X_1^\perpp\cap X_2^\perpp$ is a timelike geodesic, hence  $w$ is timelike. By Lemma \ref{lemma intersection half pipe}, the fixed planes of $\rho'({\l s}_1)$ and $\rho'({\l s}_2)$ are disjoint, while the degenerate hyperplanes fixed by $\rho'({\l t}_1)$ and $\rho'({\l t}_2)$, namely $\pi^{-1}(H_{X_1})$ and $\pi^{-1}(H_{X_2})$, intersect in $\HP^3$ (along a degenerate line). Hence point (3) is fulfilled.

Finally, if $H_{X_1}$ and $H_{X_2}$ are ultraparallel geodesics, then the closures of $\pi^{-1}(H_{X_1})$ and $\pi^{-1}(H_{X_2})$ only intersect in $\{\infty\}$. In this case $X_1^\perpp\cap X_2^\perpp$ is a spacelike geodesic, hence by Lemma \ref{lemma intersection half pipe} the fixed planes of $\rho'({\l s}_1)$ and $\rho'({\l s}_2)$ intersect. Therefore point (3) is fulfilled again.
\end{proof}

Moving to dimension four, we define cusp groups in half-pipe geometry:

\begin{defi}[Cusp groups for $\HP^4$]
The image of a representation of $\Gamma_{\mathrm{cube}}$ into $G_{\HP^4}$ is called:
\begin{itemize}
\item a \emph{cusp group} if the 6 generators are sent to reflections in 6 distinct hyperplanes 
{asymptotic to a common} point at infinity; 
\item a \emph{collapsed cusp group} if the 6 generators are sent to reflections
along five distinct hyperplanes, four degenerate and one spacelike, 
{asymptotic to a common} point at infinity.
\end{itemize}
\end{defi}

The half-pipe version of Proposition \ref{prop cube group ads collapsed} and \ref{prop cube group ads collapsed 2} is now proved along the same lines:

\begin{prop} \label{prop cube group HP collapsed} 
Let $\rho\colon \Gamma_{\mathrm{cube}}\to G_{\HP^4}$ be a representation whose image is a cusp group or a collapsed cusp group. For all nearby representations $\rho'$, exactly one of the following possibilities holds:
\begin{enumerate}
\item If ${\l s}_1$ and ${\l s}_2$ are  generators such that $\rho({\l s}_1)=\rho({\l s}_2)$ is a reflection in a non-degenerate hyperplane, then $\rho'({\l s}_1)=\rho'({\l s}_2)$. 
\item The image of $\rho'$ is a cusp group. 
\end{enumerate}
\end{prop}

\begin{proof}
Let us denote by ${\l s}_1$ and ${\l s}_2$ the generators of $\Gamma_{\mathrm{cube}}$ (corresponding to opposite faces of the cube) that are sent by $\rho$ to reflections in a non-degenerate hyperplane; by ${\l t}_1,{\l t}_2$ and ${\l u}_1,{\l u}_2$ the other two pairs of opposite generators, which are necessarily sent to reflections in degenerate hyperplanes. By continuity, the same holds for $\rho'$. 

Up to conjugation we can assume that $\rho'({\l s}_1)=\phi(-\mathrm{id},0)$, and therefore by Lemma \ref{lemma:commutation conditions} $\rho'({\l t}_i)=\phi(r_{X_i},0)$ and $\rho'({\l u}_i)=\phi(r_{Y_i},0)$, for $X_i,Y_i$ unit spacelike vectors. The restriction of $\rho'$ to the subgroup generated by these four elements gives a representation of $\Gamma_{\mathrm{rect}}$ in a copy of $\Isom(\Hyp^3)$, and is nearby a 3-dimensional cusp group. 

Suppose first $\rho'|_{\Gamma_{\mathrm{rect}}}$ is a cusp group in $\Isom(\Hyp^3)$. This means that $X_1^\perpp\cap X_2^\perpp\cap Y_1^\perpp\cap Y_2^\perpp$ is a lightlike line in $\R^{1,3}$. Then $\rho'({\l s}_2)$ is of the form $\phi(-\mathrm{id},w)$ and by Lemma  \ref{lemma:commutation conditions} $w\in X_1^\perpp\cap X_2^\perpp\cap Y_1^\perpp\cap Y_2^\perpp$. Hence $\rho$ gives a cusp group in $G_{\HP^4}$ and we are in point (2). 
 
If $\rho'|_{\Gamma_{\mathrm{rect}}}$ does not give a cusp group in $\Hyp^3$, by Proposition \ref{prop rect group hyp} two planes  intersect in $\Hyp^3$, while the {ideal closures of the} other two are disjoint in $\overline{\Hyp}^3$. We will derive a contradiction. Up to relabelling, we can assume that the planes $H_{X_1}$ and $H_{X_2}$ intersect in $\Hyp^3$, while the closures of $H_{Y_1}$ and $H_{Y_2}$ are disjoint. Hence in the degenerate subspace  $\pi^{-1}(H_{X_1})$ (which is a copy of $\HP^3$), the sets $\pi^{-1}(H_{Y_1})\cap \pi^{-1}(H_{X_1})$ and $\pi^{-1}(H_{Y_2})\cap \pi^{-1}(H_{X_1})$ are disjoint. Applying Proposition \ref{prop rect group HP} to the restriction of $\rho'$ to the subgroup generated by ${\l s}_1,{\l s}_2,{\l u}_1,{\l u}_2$,  the fixed planes of $\rho'({\l s}_1)$ and $\rho'({\l s}_2)$ intersect in $\pi^{-1}(H_{X_1})$ (and thus in $\HP^4$).

On the other hand, in $\pi^{-1}(H_{Y_1})$ (which is again a copy of $\HP^3$), $\pi^{-1}(H_{X_1})\cap \pi^{-1}(H_{Y_1})$ and $\pi^{-1}(H_{X_2})\cap \pi^{-1}(H_{Y_1})$ intersect. In fact  $H_{X_1}$ and $H_{X_2}$ intersect in $\Hyp^3$, hence also in $H_{Y_1}$ since $H_{X_1}$ and $H_{X_2}$ are orthogonal to $H_{Y_1}$. By Proposition \ref{prop rect group HP} again, the fixed planes of $\rho'({\l s}_1)$ and $\rho'({\l s}_2)$ are disjoint in $\HP^4$, which contradicts the conclusion of the previous paragraph.
\end{proof}

\section{Group cohomology and the HP character variety of $\Gamma_{22}$} \label{sec:gp cohomology}

The goal of this section is to prove the half-pipe part of Theorem \ref{teo:main_weak}. 
An essential step is an explicit computation of the first cohomology group $H^1_{\varrho_0}(\Gamma_{22},\R^{1,3})$ in Proposition \ref{prop:first cohomology group}, a result for which we will give other applications in Section \ref{sec:extra}.

\subsection{Preliminaries on group cohomology} \label{sec: H1 groups}

We recall here a few notions of group cohomology.

Let $\Gamma$ be a group, $V$ a finite-dimensional real vector space, and $\varrho\colon\Gamma\to\mathrm{GL}(V)$ a representation. The \emph{first cohomology group} of $\Gamma$ associated to $\varrho$ is the quotient

$$H^1_\varrho(\Gamma,V)=Z^1_\varrho(\Gamma,V)\big/B^1_\varrho(\Gamma,V)~,$$
where
\begin{itemize} 
\item the space of \emph{cocycles} is
$$Z^1_\varrho(\Gamma,V)=\left\{\tau\colon\Gamma\to V\,|\,\forall \gamma,\eta\in\Gamma\quad\tau(\gamma\eta)=\varrho(\gamma)\tau(\eta)+\tau(\gamma)\right\}~,$$ 
\item the space of \emph{coboundaries} is 
$$B^1_\varrho(\Gamma,V)=\left\{\tau\colon\Gamma\to V\,|\,\exists v\in V\,\,\forall \gamma\in\Gamma \quad\tau(\gamma)=\varrho(\gamma)v-v\right\}~.$$
\end{itemize}

The space $Z^1_\varrho(\Gamma,V)$ coincides with the space of \emph{affine deformations} of $\varrho$, namely the functions $\tau\colon\Gamma\to V$ such that $(\varrho,\tau)$ gives a representation of $\Gamma$ to $\GL(V)\ltimes V$. The difference $\tau-\tau'$ of two cocycles is a coboundary if and only if the corresponding representations $(\varrho,\tau)$ and $(\varrho,\tau')$ are conjugate in $V$. We have in summary:

\begin{lemma}\label{lem:parametr_lin}
The vector space $H^1_\varrho(\Gamma,V)$ parameterises the representations of $\Gamma$ in $\GL(V)\ltimes V$ having linear part $\varrho$, up to conjugation.
\end{lemma}

When $\Gamma$ is a right-angled Coxeter group, the space $Z^1_\varrho(\Gamma,V)$ has the following description in terms of generators and relations.

\begin{lemma}\label{lemma cocycle generators}
Let $\Gamma$ be a right-angled Coxeter group as in Defintion \ref{defi RACG}, and $\varrho\colon\Gamma\to\mathrm{GL}(V)$ a representation. Then $Z^1_\varrho(\Gamma,V)$ is isomorphic to the vector space of functions $\tau\colon S\to V$ such that:
\begin{itemize}
\item $\tau({\l s})\in \mathrm{Ker}(\mathrm{id}+\varrho({\l s}))$ for all ${\l s}\in S$, and
\item $(\mathrm{id}-\varrho({\l s}_i))\tau({\l s}_j)=(\mathrm{id}-\varrho({\l s}_j))\tau({\l s}_i)$ for all $({\l s}_i,{\l s}_j)\in R$.
\end{itemize}
\end{lemma} 
\begin{proof}
Clearly a cocycle in $Z^1_\varrho(\Gamma,V)$ is determined by its values on the generators. 
The conditions that have to be satisfied by $\tau$ for each relation are $0=\tau({\l s}^2)=\varrho({\l s})\tau({\l s})+\tau({\l s})$, from which we get the first point, and $\tau({\l s}_i{\l s}_j)=\tau({\l s}_j{\l s}_i)$ for every $({\l s}_i,{\l s}_j)\in R$. Expanding $\tau({\l s}_i{\l s}_j)=\varrho({\l s}_i)\tau({\l s}_j)+\tau({\l s}_i)$ we obtain the second point.
\end{proof}

\subsection{A curve of geometric representations} \label{sec:geom_cocycle}

We now introduce the half-pipe representations of our interest, which have been computed in \cite[Remark 7.16]{transition_4-manifold} by applying a rescaling argument to the hyperbolic or AdS holonomy representations $\rho_t$.

\begin{notation*}
Throughout the following, we will denote by $\langle\cdot,\cdot\rangle$ the Minkowski product of $\R^{1,3}$ (previously denoted by $b_1$) and by $v^\perp\subset\R^{1,3}$ the orthogonal complement of $v\in\R^{1,3}$ with respect to the Minkowski product.
\end{notation*}

Recall that by Lemma \ref{lemma isomorphism} the transformation group $G_{\HP^4}$ is isomorphic to $\Isom(\R^{1,3})\cong \O(1,3)\ltimes\R^4$. We will exhibit the half-pipe holonomies as representations in $\Isom(\R^{1,3})$.

\begin{defi}[The HP representation $\rho_\lambda$]\label{defi HP holonomy}
Given $\lambda\in\R$, we define a representation
$$\rho_\lambda=(\varrho_0,\tau_\lambda)\colon\Gamma_{22}\to\O(1,3)\ltimes\R^4$$
on the standard generators of $\Gamma_{22}$ as follows. The linear part $\varrho_0$ is independent of $\lambda$ and is defined by:
\begin{equation} \label{eq:rep rho0}
\begin{aligned} 
\varrho_0(\p i)&=-\mathrm{id} & &\text{for each } \l i\in\{\l0,\ldots,\l7\} \\
\varrho_0(\m i)&=r_{v_{\l i}} & &\text{for each } \l i\in\{\l0,\ldots,\l7\} \\
\varrho_0(\l X)&=r_{v_{\l X}} & &\text{for each } \l X\in\{\l A,\ldots,\l F\}
\end{aligned}
\end{equation}
while the translation part is:
\begin{equation} \label{eq:cocycle tau}
\begin{aligned} 
\tau_\lambda(\p i)=\tau_\lambda(\m i)&=(-1)^i\lambda v_{\l i} & &\text{for each } \l i\in\{\l 0,\ldots,\l 7\} \\
\tau_\lambda(\l X)&=0 & &\text{for each } \l X\in\{\l A,\ldots,\l F\}
\end{aligned}
\end{equation}
where the vectors $v_{\l s}$ are defined in Table \ref{table:vectors cocycle}.
\end{defi}

Recall that the vectors in Table \ref{table:vectors cocycle} definine the bounding planes of an ideal right-angled cuboctahedron in $\Hyp^3$. Moreover, $r_{v}$ denotes the reflection in $\O(1,3)$ in the hyperplane $v^\perp$, namely, the linear transformation acting on $v^\perp$ as the identity and on the subspace generated by $v$ as minus the identity. 

\begin{remark}
%
When $\lambda=0$, the representation $\rho_0$ is naturally identified to those introduced in Definition \ref{defi rhot} for $t=0$. Indeed, recall from Section \ref{subsec:cuboct} that in the hyperbolic and AdS case $\rho_0$ takes value in the stabiliser $G_0$ of the hyperplane $\{x_4=0\}$, and $G_0$ is a common subgroup of $\Isom(\Hyp^4)$ and $\Isom(\AdS^4)$, both seen as a subgroups of {$\GL(5,\R)$}. 

Now, the representation $\rho_0=(\varrho_0,0)$ introduced  in Definition \ref{defi HP holonomy} also takes value in the stabiliser of $\{x_4=0\}$ in $G_{\HP^4}$ which coincides again with the subgroup $G_0$ of {$\GL(5,\R)$}. Under the isomorphism with $\Isom(\R^{1,3})$, the group $G_0$ is dually identified with the stabiliser of the origin in $\R^{1,3}$, namely the linear subgroup $\O(1,3)<\Isom(\R^{1,3})$. {The explicit isomorphism $\O(1,3)\cong G_0 $ is
given by 
\begin{equation} \label{eq:isomorphism G0}
\O(1,3) \ni A \mapsto \pm \begin{pmatrix} A & 0 \\ 0 & 1\end{pmatrix}\in G_0~,
\end{equation}
where the sign $\pm$ is positive if $A$ preserves $\Hyp^{3} \subset \R^{1,3}$, and negative otherwise.}

Under the isomorphism \eqref{eq:isomorphism G0}, the representation $\rho_0=(\varrho_0,0)$ of Definition \ref{defi HP holonomy} (with zero translation part) coincides with the ``collapsed'' representation expressed in \eqref{eq:rep rho0 hyp ads}. This justifies that in the statement of Theorem \ref{teo:main} we refer to the \emph{same} representation $\rho_0$ in all three geometries.
\end{remark}

The goal of the following section is to compute the  first cohomology group associated to the representation $\varrho_0\colon\Gamma_{22}\to \O(1,3)$ of Definition \ref{defi HP holonomy}. Applications of the result will then be given in Sections \ref{sec: teo B HP} and \ref{sec:zariski2}.

\subsection{The geometric cocycle is a generator} \label{sec:geom_cocycle_gen}

Recall Definition \ref{defi HP holonomy}. The goal of this section is to prove the following:

\begin{prop} \label{prop:first cohomology group}
The vector space $H^1_{\varrho_0}(\Gamma_{22},\R^{1,3})$ has dimension one.
\end{prop}

To prove Proposition \ref{prop:first cohomology group}, we will show that every cohomology class in $H^1_{\varrho_0}(\Gamma_{22},\R^{1,3})$ is represented by a cocycle $\tau_\lambda$ of the form \eqref{eq:cocycle tau}, for some $\lambda\in\R$. 

We already know from \cite[Remark 7.16]{transition_4-manifold} that $\rho_\lambda=(\varrho_0,\tau_\lambda)$ of Definition \ref{defi HP holonomy} is a representation of $\Gamma_{22}$, hence $\tau_\lambda$ is a cocycle. This can however be checked directly from \eqref{eq:rep rho0} and \eqref{eq:cocycle tau} using Lemma \ref{lemma cocycle generators}.
Let us introduce some additional notation:

\begin{defi}[The subspace $U_0$] \label{defi:subspace cocycles}
We denote by $U_0$ the 1-dimensional vector subspace of $Z^1_{\varrho_0}(\Gamma_{22},\R^{1,3})$ composed of cocycles of the form \eqref{eq:cocycle tau}, for some $\lambda\in\R$.
\end{defi}

Let us observe that $\tau_\lambda$ vanishes on all the letter generators and that $\tau_\lambda\!\left(\m{i}\right)$ and $\tau_\lambda\!\left(\p{i}\right)$  are all vectors of norm $|\lambda|$ for the Minkowski product on $\R^{1,3}$, since all the $v_{\l i}$ have unit Minkowski norm. 

The ultimate goal  will be to show that any cocycle $\tau\in Z^1_{\varrho_0}(\Gamma_{22},\R^{1,3})$ has a unique decomposition $\tau=\tau_\lambda-\eta$, for some $\eta\in B^1_{\varrho_0}(\Gamma_{22},\R^{1,3})$ and $\tau_\lambda\in U_0$. The proof will follow from a sequence of computational lemmas.

\begin{lemma} \label{lemma:cocycle condition squares}
Let $\tau\in Z^1_{\varrho_0}(\Gamma_{22},\R^{1,3})$. Then,
\begin{equation*} 
\begin{aligned} 
\tau(\m i)&\in\mathrm{Span}(v_{\l i})& &\text{for each }\m i\in\{\m 0,\ldots,\m 7\},\ \text{and} \\
\tau(\l X)&\in\mathrm{Span}(v_{\l X})& &\text{for each }\l X\in\{\l A,\ldots,\l F\}.
\end{aligned}
\end{equation*}
\end{lemma}
\begin{proof}
By Lemma \ref{lemma cocycle generators} we get
$\tau(\m i)\in \mathrm{Ker}(\mathrm{id}+\varrho_0(\m i))$. This kernel equals the subspace generated by $v_{\l i}$ since $\varrho_0(\m i)$ is the Minkowski reflection fixing the hyperplane $v_{\l i}^\perp$. The proof for the letter generators is the same.
\end{proof}

The following step is a first reduction of the problem.
\begin{lemma}\label{lemma:four letters vanish}
Let $\tau\in Z^1_{\varrho_0}(\Gamma_{22},\R^{1,3})$. Then there exists a unique $\eta\in B^1_{\varrho_0}(\Gamma_{22},\R^{1,3})$ such that, if ${\hat\tau}=\tau-\eta$, then
\begin{equation} \label{eq:four letters vanish}
{\hat\tau}(\l A)={\hat\tau}(\l B)={\hat\tau}(\l C)={\hat\tau}(\l D)=0~.
\end{equation}
\end{lemma}

After the proof of Lemma \ref{lemma:four letters vanish}, we will show that if ${\hat\tau}$ satisfies \eqref{eq:four letters vanish}, then it is of the form \eqref{eq:cocycle tau} for some $\lambda\in\R$. Together with Lemma \ref{lemma:four letters vanish}, this will imply that 
\begin{equation} \label{eq:cocycle tau2}
Z^1_{\varrho_0}(\Gamma_{22},\R^{1,3})=U_0\oplus B^1_{\varrho_0}(\Gamma_{22},\R^{1,3})
\end{equation}
and therefore that $H^1_{\varrho_0}(\Gamma_{22},\R^{1,3})$ is one-dimensional.

\begin{proof}[Proof of Lemma \ref{lemma:four letters vanish}]
Let $\tau$ be any cocycle. By Lemma \ref{lemma:cocycle condition squares}, we have that $\tau(\l X)\in \mathrm{Span}(v_{\l X})$ for all $\l X\in\{\l A,\l B,\l C,\l D\}$.  Define the linear map
$$L\colon \R^{1,3}\to \mathrm{Span}(v_{\l A})\oplus\mathrm{Span}(v_{\l B})\oplus\mathrm{Span}(v_{\l C})\oplus\mathrm{Span}(v_{\l D})$$
by
$$L(w)=(\varrho_0(\l A)w-w,\varrho_0(\l B)w-w,\varrho_0(\l C)w-w,\varrho_0(\l D)w-w)~.$$
The proof follows if we show that $L$ is invertible.

Let us write the matrix associated to $L$ in the basis $\{v_{\l A},v_{\l B},v_{\l C},v_{\l D}\}$ on the source and on the target. Recalling that the $v_{\l X}$ are all unit vectors for the Minkowski product $\langle\cdot,\cdot\rangle$ and that $\rho_0(\l X)$ is the reflection in $v_{\l X}^\perp$, we have
$$\varrho_0(\l X)v_{\l Y}-v_{\l Y}=\varrho_0(\l X)\left(\langle v_{\l Y},v_{\l X}\rangle v_{\l X}\right)-\langle v_{\l Y},v_{\l X}\rangle v_{\l X}=-2\langle v_{\l Y},v_{\l X}\rangle v_{\l X}~.$$
This shows that the associated matrix of $L$ is
$$-2\begin{pmatrix} \langle v_{\l A},v_{\l A}\rangle & \langle v_{\l A},v_{\l B}\rangle & \langle v_{\l A},v_{\l C}\rangle & \langle v_{\l A},v_{\l D}\rangle \\ 
\langle v_{\l B},v_{\l A}\rangle & \langle v_{\l B},v_{\l B}\rangle & \langle v_{\l B},v_{\l C}\rangle & \langle v_{\l B},v_{\l D}\rangle \\
\langle v_{\l C},v_{\l A}\rangle & \langle v_{\l C},v_{\l B}\rangle & \langle v_{\l C},v_{\l C}\rangle & \langle v_{\l C},v_{\l D}\rangle \\
\langle v_{\l D},v_{\l A}\rangle & \langle v_{\l D},v_{\l B}\rangle & \langle v_{\l D},v_{\l C}\rangle & \langle v_{\l D},v_{\l D}\rangle 
\end{pmatrix},$$
which is invertible by the non-degeneracy of the Minkowski product.
\end{proof}

\begin{remark}\label{rmk:dimB1}
The proof of Lemma \ref{lemma:four letters vanish} also shows that $B^1_{\varrho_0}(\Gamma_{22},\R^{1,3})$ {has dimension 4}. Indeed, we have a surjective linear map 
$\R^{1,3}\to B^1_{\varrho_0}(\Gamma_{22},\R^{1,3})$ that sends $w\in\R^{1,3}$ to the coboundary $\tau(\gamma)=\varrho_0(\gamma)w-w$. 
This map is  injective, by the injectivity of the map $L$ introduced in the proof of Lemma \ref{lemma:four letters vanish}. Hence $\dim B^1_{\varrho_0}(\Gamma_{22},\R^{1,3})=4$.
\end{remark}

Let us now compute the cocycle condition which arises from any orthogonality condition in $\Gamma_{22}$.

\begin{lemma} \label{lemma:cocycle condition orthogonal}
Let $\tau\in Z^1_{\varrho_0}(\Gamma_{22},\R^{1,3})$.
\begin{itemize}
\item For any relation in $\Gamma_{22}$ of the form $\p i \m j=\m j \p i $, we have that $\tau(\p i)-\tau(\m j)\in v_{\l j}^\perp$.
\item For any relation in $\Gamma_{22}$ of the form $\p i \l X=\l X \p i $, we have that $\tau(\p i)-\tau(\l X)\in v_{\l X}^\perp$.
\end{itemize}
\end{lemma}
\begin{proof}
Let us show the first point, the second being completely analogous. By Lemma \ref{lemma cocycle generators}
$$\left(\mathrm{id}-\varrho_0(\m j)\right)\tau(\p i)=\left(\mathrm{id}-\varrho_0(\p i)\right)\tau(\m j)=2\tau(\m j)=\left(\mathrm{id}-\varrho_0(\m j)\right)\tau(\m j)~,
$$
where we have used that $\varrho_0(\p i)=-\mathrm{id}$, that $\varrho_0(\m j)$ is the reflection in the Minkowski hyperplane $v_{\l j}^\perp$, and that $\tau(\m j)\in\mathrm{Span}(v_{\l j})$ by Lemma \ref{lemma:cocycle condition squares}. Hence $\tau(\p i)-\tau(\m j)$ is in the kernel of $\mathrm{id}-\varrho_0(\m j)$, namely in $v_{\l j}^\perp$.
\end{proof}

Let us now go back to showing that a cocycle ${\hat\tau}$ as in Lemma \ref{lemma:four letters vanish} is of the form \eqref{eq:cocycle tau}. Our next step is:

\begin{lemma} \label{lemma: cocycle 0 and 3}
Suppose ${\hat\tau}\in Z^1_{\varrho_0}(\Gamma_{22},\R^{1,3})$ satisfies \eqref{eq:four letters vanish}. Then
${\hat\tau}(\p 0)={\hat\tau}(\m 0)=\lambda v_{\l 0}$ and ${\hat\tau}(\p 3)={\hat\tau}(\m 3)=-\lambda v_{\l 3}$
for some $\lambda\in\R$.
\end{lemma}
\begin{proof}
It follows from Lemma \ref{lemma:cocycle condition squares} that ${\hat\tau}(\m 0)=\mu_0v_{\l 0}$, and similarly ${\hat\tau}(\m 3)=\mu_3v_{\l 3}$. We remark that we have no similar condition on the $\p i$ coming from the relation that $\p i$ squares to the identity.

However, we claim that in our assumption also ${\hat\tau}(\p 0)\in \mathrm{Span}(v_{\l 0})$ and ${\hat\tau}(\p 3)\in \mathrm{Span}(v_{\l 3})$. Indeed, applying Lemma \ref{lemma:cocycle condition orthogonal} to the relation $\p 0\l A= \l A\p 0$ and using that ${\hat\tau}(\l A)=0$ by hypothesis, we get ${\hat\tau}(\p 0)\in v_{\l A}^\perp$. Similarly, from $\p 0\l B= \l B\p 0$ and $\p 0\l C= \l C\p 0$, we obtain that ${\hat\tau}(\p 0)$ is in $v_{\l B}^\perp$ and $v_{\l C}^\perp$. Now, $v_{\l A}$, $v_{\l B}$ and $v_{\l C}$ are linearly independent, hence $v_{\l A}^\perp\cap v_{\l B}^\perp\cap v_{\l C}^\perp$ is 1-dimensional and therefore coincides with $\mathrm{Span}(v_{\l 0})$, since $v_{\l 0}$ is orthogonal to all of them. By applying the same argument to ${\hat\tau}(\p 3)$ and the letters $\l A$, $\l B$, $\l D$ (since by hypothesis ${\hat\tau}$ vanishes on $\l A$, $\l B$, $\l C$ and $\l D$), we obtain that ${\hat\tau}(\p 3)\in \mathrm{Span}(v_{\l 3})$.

Hence we have shown that ${\hat\tau}(\p 0)=\lambda_0v_{\l 0}$ and ${\hat\tau}(\p 3)=\lambda_3v_{\l 3}$. We have to show that $\lambda_0=\mu_0=-\lambda_3=-\mu_3$. Let us apply Lemma \ref{lemma:cocycle condition orthogonal} to the relation $\p 0\m 0=\m 0\p 0$. We obtain
$${\hat\tau}(\p 0)-{\hat\tau}(\m 0)\in v_{\l 0}^\perp~,$$
that is,
$$0=\langle \lambda_0v_{\l 0}-\mu_0v_{\l 0},v_{\l 0}\rangle=\lambda_0-\mu_0~,$$
hence $\lambda_0=\mu_0$. Analogously $\lambda_3=\mu_3$. If we now apply Lemma \ref{lemma:cocycle condition orthogonal} to the relation $\p 0\m 3=\m 3\p 0$ we get
$${\hat\tau}(\p 0)-{\hat\tau}(\m 3)\in v_{\l 3}^\perp~,$$
which in turn gives 
$$0=\langle \lambda_0v_{\l 0}-\lambda_3v_{\l 3},v_{\l 3}\rangle=\lambda_0\langle v_{\l 0},v_{\l 3}\rangle-\lambda_3\langle v_{\l 3},v_{\l 3}\rangle=-\lambda_0-\lambda_3.$$
We conclude by setting $\lambda:=\lambda_0=-\lambda_3$.
\end{proof}

\begin{remark}
The proof of Lemma \ref{lemma: cocycle 0 and 3} only worked for $\l i=\l 0,\l 3$ because we used that ${\hat\tau}$ vanishes on $\l A,\l B,\l C$ and $\l D$, and we needed to pick three linearly independent vectors among these four. Once we show that ${\hat\tau}$ also vanishes on $\l E$ and $\l F$ (Lemma \ref{lemma:lettersEFvanish} below), the same argument will apply exactly in the same way to show that
$${\hat\tau}(\p i)={\hat\tau}(\m i)=\lambda v_{\l i}$$
for $\l i$ odd, and 
$${\hat\tau}(\p i)={\hat\tau}(\m i)=-\lambda v_{\l i}$$
for $\l i$ even. This will therefore conclude the proof that ${\hat\tau}$ is in the form \eqref{eq:cocycle tau}.
\end{remark}

\begin{lemma} \label{lemma:lettersEFvanish}
If ${\hat\tau}\in Z^1_{\rho_0}(\Gamma_{22},\R^{1,3})$ satisfies \eqref{eq:four letters vanish}, then
${\hat\tau}(\l E)={\hat\tau}(\l F)=0$.
\end{lemma}
\begin{proof}
From Lemma \ref{lemma:cocycle condition squares}, we know that
$${\hat\tau}(\l E)=e v_{\l E}\qquad \text{and}\qquad{\hat\tau}(\l F)=f v_{\l F}~.$$
We wish to show that $e=f=0$. Let us first prove that $e=0$.

Observe that $v_{\l A}$, $v_{\l C}$ and $v_{\l E}$ are linearly independent, and they are all orthogonal to $v_{\l 1}$. Hence $\{v_{\l 1},v_{\l A},v_{\l C},v_{\l E}\}$ is a (non-orthogonal!) basis of unit vectors and we can decompose:
$${\hat\tau}(\p 1)=\lambda_1v_{\l 1}+\alpha v_{\l A}+\gamma v_{\l C}+\epsilon v_{\l E}~.$$
(We ultimately will get, at the end of the proof, that $\lambda_1=-\lambda$ and $\alpha=\gamma=\epsilon=0$, but we do not know this yet.) As a preliminary remark, observe that ${\hat\tau}(\m 1)=\lambda_1v_{\l 1}$, since from the relation $\p 1\m 1=\m 1\p 1$ we obtain
$${\hat\tau}(\p 1)-{\hat\tau}(\m 1)\in v_{\l 1}^\perp~,$$
and comparing with the above decomposition, necessarily ${\hat\tau}(\m 1)=\lambda_1v_{\l 1}$.

Since ${\hat\tau}(\l{A})=0$, from the relation $\p 1\l A=\l A\p 1$ we obtain
$${\hat\tau}(\p{1})\in v_{\l A}^\perp~,$$
namely,
\begin{equation}\label{eq:paperoga1}
0=\langle {\hat\tau}(\p{1}),v_{\l A}\rangle=\alpha-\gamma-\epsilon~.
\end{equation}
From the same computation for the relation $\p 1\l C=\l C\p 1$ we derive
\begin{equation}\label{eq:paperoga2}
0=\langle {\hat\tau}(\p{1}),v_{\l C}\rangle=-\alpha+\gamma-\epsilon~.
\end{equation}
Finally, the relation $\p 1\l E=\l E\p 1$ implies ${\hat\tau}(\p{1})-{\hat\tau}(\l E)\in v_{\l E}^\perp$, whence
\begin{equation}\label{eq:paperoga3}
e=\langle {\hat\tau}(\l E),v_{\l E}\rangle=\langle {\hat\tau}(\p{1}),v_{\l E}\rangle=-\alpha-\gamma+\epsilon~.
\end{equation}
From \eqref{eq:paperoga1}, \eqref{eq:paperoga2} and \eqref{eq:paperoga3} together we find
\begin{equation}\label{eq:paperoga4}
\alpha=\gamma=-\frac{e}{2}\qquad \epsilon=0~.
\end{equation}

On the other hand, consider the relation $\p 1\m 2=\m 2\p 1$. It implies
$${\hat\tau}(\p{1})-{\hat\tau}(\m 2)\in v_{\l 2}^\perp~,$$
where we already know that ${\hat\tau}(\m 2)=\lambda_2v_{\l 2}$. A direct computation gives 
$$0=\langle \lambda_1v_{\l 1}+\alpha v_{\l A}+\gamma v_{\l C}+\epsilon v_{\l E}-\lambda_2v_{\l 2},v_{\l 2}\rangle =-\lambda_1-2\sqrt 2\gamma-\lambda_2~.$$
If we show that $\lambda_1=-\lambda_2$
, we are done for ${\hat\tau}(\l E)$, since $\gamma=0$ implies $e=0$ from \eqref{eq:paperoga4}.

To see this last point, recall that ${\hat\tau}(\p 0)=\lambda v_{\l 0}$ and ${\hat\tau}(\p 3)=-\lambda v_{\l 3}$ as proved in Lemma \ref{lemma: cocycle 0 and 3}. Now, from the orthogonality relation $\p 0\m 1=\m 1\p 0$ we find ${\hat\tau}(\p{0})-{\hat\tau}(\m 1)\in v_{\l 1}^\perp$. Using the preliminary remark at the beginning of the proof,
$$0=\lambda\langle v_{\l 0},v_{\l 1}\rangle-\lambda_1\langle v_{\l 1},v_{\l 1}\rangle=-\lambda-\lambda_1~.$$
Thus $\lambda_1=-\lambda$. Repeating the same argument to the relation  $\p 3\m 2=\m 2\p 3$ one finds $\lambda_2=\lambda$, and therefore $\lambda_1=-\lambda_2$.

The proof that $f=0$ follows the same lines, applied to $\p 4$ in place of $\p 1$, with the letters $\l B$, $\l D$ and $\l F$, and in the final part to $\m 5$ in place of $\m 2$.
\end{proof}

Having shown that ${\hat\tau}(\l X)=0$ for every $\l X$, it remains to show that ${\hat\tau}(\p i)={\hat\tau}(\m i)$ has the form of \eqref{eq:cocycle tau}. For $\l i=\l 0,\l 3$, this is the content of  Lemma \ref{lemma: cocycle 0 and 3}. Following the same proof, one shows first that 
$${\hat\tau}(\p i)={\hat\tau}(\m i)$$
for every $\l i$ (it suffices to modify the proof by picking three letters $\l X$, $\l Y$ and $\l Z$ so that $v_{\l X}$, $v_{\l Y}$ and $v_{\l Z}$ are orthogonal to $v_{\l i}$). Then  using the crossed relations $\p i\m j=\m j\p i$  --- it is easy to see that there are indeed enough of such relations --- one mimics the second part of Lemma  \ref{lemma: cocycle 0 and 3} and obtains that 
$${\hat\tau}(\p i)={\hat\tau}(\m i)=(-1)^i\lambda v_{\l i}~.$$

This concludes the proof of Proposition \ref{prop:first cohomology group}, namely that $\dim H^1_{\varrho_0}(\Gamma_{22},\R^{1,3})=1$.

\subsection{Topology of the neighbourhood $\mathcal{U}$} \label{sec: teo B HP}

We are ready to conclude the proof of our weak version of Theorem \ref{teo:main}, namely Theorem \ref{teo:main_weak}, in the half-pipe case. The  proof {of Theorem \ref{teo:main}} will be completed in Section \ref{sec:extra}.

As a preliminary setup, recall that $G_{\HP^4}\cong\O(1,3)\ltimes\R^{1,3}$, one has a natural map:
$${\mathcal L}\colon X(\Gamma,G_{\HP^4})\to X(\Gamma,\O(1,3))$$
which associates to the conjugacy class of a representation $\rho\colon\Gamma\to G_{\HP^4}$ the conjugacy class of the linear part of $\rho$.
Recalling Lemma \ref{lem:parametr_lin}, one has the identification 
\begin{equation} \label{eq:identification preimage with H1}
{\mathcal L}^{-1}([\varrho])\cong H^1_\varrho(\Gamma,\R^{1,3})~.
\end{equation}
Observe that if $\varrho'=h\circ \varrho\circ h^{-1}$ for $h\in\O(1,3)$, then $H^1_\varrho(\Gamma,\R^{1,3})$ and $H^1_{\varrho'}(\Gamma,\R^{1,3})$ are isomorphic by means of the map $\tau\mapsto h\circ\tau$.

\begin{proof}[Proof of Theorem \ref{teo:main_weak} --- HP case]
The proof follows a similar strategy to the AdS (and hyperbolic) case, so we will split again the proof in several steps which are parallel to those given in Section \ref{sec:proof teoB ads}. Most steps are much simpler here.

\begin{steps}
\item Let us define the vertical component $\mathcal V$ in $X(\Gamma_{22},G_{\HP^4})$ as ${\mathcal L}^{-1}([\varrho_0])$, namely, $\mathcal V$ consists of all the conjugacy classes of representations with linear part in $[\varrho_0]$. By \eqref{eq:identification preimage with H1}, $\mathcal V$ is identified to $H^1_{\varrho_0}(\Gamma_{22},\R^{1,3})$, hence is homeomorphic to a line by Proposition \ref{prop:first cohomology group}. By construction, $\mathcal V$ contains the holonomy of the half-pipe orbifold structures we built in 
\cite{transition_4-manifold}.

\item The second component ${\mathcal H}$ is defined similarly to the AdS case. We define the map 
$$\Psi\colon\mathrm{Hom}(\Gamma_{\mathrm{co}},\Isom(\Hyp^3))\to\mathrm{Hom}(\Gamma_{22},G_{\HP^4})$$ 
sending a representation  $\eta\colon\Gamma_{\mathrm{co}}\to\Isom(\Hyp^3)$ to the representation $\Psi_\eta\colon\Gamma_{22}\to \O(1,3)<\O(1,3)\ltimes\R^{1,3}$ (hence with trivial translation part, which we omit) 
which sends each of the generators $\m0,\ldots,\m7,\l A,\ldots,\l F$ to the corresponding element of $\O(1,3)$, and each $\p i\in\{\p0,\ldots,\p7\}$ to $-\mathrm{id}$.

Again, it is straightforward to check that the induced map  
$$\widehat\Psi\colon \bar X(\Gamma_{\mathrm{co}},\Isom(\Hyp^3))\to X(\Gamma_{22},G_{\HP^4})~.$$
is well-defined and injective.

The representation $\rho_0$ is clearly in the image of $\Psi$, since $\rho_0=\Psi_{\eta_0}$ where $\eta_0$ is the holonomy representation of the complete hyperbolic orbifold structure of the cuboctahedron. As in the AdS case, $[\eta_0]$ has a neighborhood $\mathcal H_0$ in 
$\bar X(\Gamma_{\mathrm{co}},\Isom(\Hyp^3))$ homeomorphic to $\R^{12}$ and on which $\widehat\Psi$ is a homeomorphism onto its image, and we define $\mathcal H$ to be the image of $\mathcal H_0$. 

\item Clearly, the intersection of ${\mathcal H}$ and ${\mathcal V}$ consists only of the point $[\rho_0]$, since any element in $\mathcal H$ has trivial translation part (up to conjugacy).

\item We now show that the point $[\rho_0]\in X(\Gamma_{22},G_{\HP^4})$ has a neighbourhood ${\mathcal U}$ which is contained in the union of the two components ${\mathcal V}$ and ${\mathcal H}$.

Let $\rho$ be a nearby representation, with linear part ${\mathcal L}\rho$ and translation part $\tau\colon\Gamma_{22}\to\R^{1,3}$. Observe that, since $-\mathrm{id}$ is an isolated point in the representations of $\Z/2\Z$ into $\O(1,3)$, for each  generator $\p i\in\{\p 0,\ldots,\p 7\}$ we have ${\mathcal L}\rho(\p i)=-\mathrm{id}$.

We claim that if two distinct generators which are sent by $\rho_0$ to $-\mathrm{id}$ (hence necessarily of the form $\p i$ and $\p j$) are sent by $\rho$ to the same reflection, than all the generators $\p 0,\ldots,\p 7$ are sent by $\rho$ to the same reflection. In other words, if $\tau(\p i)=\tau(\p j)$ for some $\p i \neq \p j$, then $\tau(\p i)=\tau(\p j)$ for all $\p i,\p j$. 

Assuming the claim, the proof then follows by the following argument. We can assume (up to conjugation) that $\tau(\p i)=0$ for all $\p i\in\{\p 0,\ldots,\p 7\}$. By Proposition  \ref{prop cube group HP collapsed}, if some of the collapsed cusp groups of $\rho_0$ is not deformed to a cusp group, then up to conjugation $\rho$ has the property that $\rho(\p i)=(-\mathrm{id},0)$ for all $\p i\in\{\p 0,\ldots,\p 7\}$, and therefore $[\rho]\in \mathcal H$. On the other hand, if all the collapsed cusp groups of $\rho_0$ are deformed in $\rho$ to cusp groups, then the linear part of $\rho$ is of the form ${\mathcal L}\rho=\Psi_\eta$ for a representation $\eta\colon\Gamma_{\mathrm{co}}\to \Isom(\Hyp^3)$ which sends all peripheral groups to (three-dimensional) cusp groups in $\Hyp^3$, and therefore $\eta$ is conjugate to $\eta_0$ in $\Isom(\Hyp^3)$ by the the Mostow--Prasad rigidity. Thus $[{\mathcal L}\rho]=[{\mathcal L}\rho_0]$, which means that $[\rho]\in\mathcal V$.

To prove the claim, suppose that $\tau(\p i)=\tau(\p j)$. We can assume that $\tau(\p i)=\tau(\p j)=0$ by conjugation. Analogously to the same step in the AdS case, by symmetry (see \cite[Lemma 7.6]{transition_4-manifold}) and Proposition \ref{prop cube group ads collapsed}, we can assume the two generators are $\p 0$ and $\p 1$. Hence we have $\rho(\p 0)=\rho(\p 1)=(-\mathrm{id},0)$. We see from the relations involving $\p 0$ that $\rho(\p 0)$ commutes with $\rho(\m 1)$, $\rho(\m 3)$ and $\rho(\l A)$, which have all linear part a reflection in $\Hyp^3$. By  Lemma \ref{lemma:commutation conditions}, $\rho(\m 1)$, $\rho(\m 3)$ and $\rho(\l A)$ have zero translation part. Additionally, we see from the relations involving $\p 1$ that $\rho(\m 2)$ has zero translation part. Now, from the relations involving $\p 2$, we get that $\rho(\p 2)$ commutes with $\rho(\m 1)$, $\rho(\m 2)$, $\rho(\m 3)$ and $\rho(\l A)$. Observe that the linear part of $\rho(\p 2)$ is necessarily $-\mathrm{id}$, in a neighbourhood of $\rho_0$. Hence by applying Lemma \ref{lemma:commutation conditions} again, the translation part of $\rho(\p 2)$ is in the intersection of the hyperplanes of $\R^{1,3}$ fixed by $\rho(\m 1)$, $\rho(\m 2)$, $\rho(\m 3)$ and $\rho(\l A)$. The hyperplanes fixed by $\rho_0(\m 1)$, $\rho_0(\m 2)$, $\rho_0(\m 3)$ and $\rho_0(\l A)$ are $v_{\l 1}^\perp$, $v_{\l 2}^\perp$, $v_{\l 3}^\perp$ and $v_{\l A}^\perp$, where the vectors $v_{\l 1}$, $v_{\l 2}$, $v_{\l 3}$ and $v_{\l A}$ are listed in Table \ref{table:vectors cocycle} and are linearly independent. Hence they remain linearly independent for $\rho$ a deformation of $\rho_0$ in a small neighbourhood. This means that the translation part of $\rho(\p 2)$ is zero, since the only solution of the linear system which imposes the orthogonality to these four linearly independent vectors is the trivial solution. This shows that $\rho(\p 2)=(-\mathrm{id},0)$, which therefore coincides with $\rho(\p 0)=\rho(\p 1)$.

Similarly to the AdS case, one argues similarly for $\p 3$ and then for all the other generators, to show that $\rho(\p i)=(-\mathrm{id},0)$ for each generator $\p i \in \{ \p0, \ldots, \p7 \}$, and this concludes the claim.
 
\item In summary, we showed that  $[\rho_0]$ has a neighborhood ${\mathcal U}$ in $X(\Gamma_{22},\Isom(\AdS^4))$ which only consists of points of ${\mathcal H}$ and ${\mathcal V}$. Additionally, one can  repeat the  same reasoning in the first part of the previous step, to show that for any other $[\rho_0']$ in $\mathcal V$ (hence having the same linear part as $\rho_0$ and non-vanishing translation part) a neighbourhood of $[\rho_0']$ is contained in $\mathcal V$, as a consequence of the half-pipe cusp rigidity of Proposition  \ref{prop cube group HP collapsed} (the non-collapsed case). Hence by taking the union of all these neighbourhoods, one finds a $\mathcal U$ containing $[\rho_0]$ such that $\mathcal U=\mathcal V\cup\mathcal H$.

\item For the last statement, it is evident that conjugation by $\Z/2\Z\cong G_{\HP^4}/G^+_{\HP^4}$ acts by switching sign to the $x_{13}$-coordinate, since conjugation by $(-\mathrm{id},0)$, whose class generates $\Z/2\Z$, acts on $H^1_{\varrho_0}(\Gamma_{22},\R^{1,3})$ by changing the sign. This concludes the proof.
\qedhere
\end{steps}
\end{proof}

\section{Smoothness and transversality} \label{sec:extra}

In this section we complete the proof of Theorem \ref{teo:main}, showing the smoothness and transversality of the two components $\mathcal{V}$ and $\mathcal{H}$ of the neighbourhood $\mathcal{U}$ of $[\rho_0]$ in $X(\Gamma_{22},{G})$. To that purpose, we first need to study the cohomology group $H^1_{\mathrm{Ad}\,\rho_0}(\Gamma_{22},\mathfrak{g})$, complementing and using the results of Section \ref{sec:geom_cocycle_gen}. Then, we conclude the proof by an application of the implicit function theorem.


\subsection{Preliminaries} \label{sec:zariski}


{Let $G$ be $\Isom(\Hyp^4)$, $\Isom(\AdS^4)$ or $G_{\HP^4}$, and $\mathfrak{g}$ be its Lie algebra.} We shall apply the definition of first cohomology group given in Section \ref{sec: H1 groups} 
to the representation
$$\mathrm{Ad}\,\rho_0 \colon \Gamma_{22}\to \mathrm{GL}(\mathfrak g)~,$$
which is the composition of our $\rho_0\colon\Gamma_{22}\to G$ and the adjoint representation $\mathrm{Ad}\colon G\to\mathrm{GL}(\mathfrak g)$
.

In general, for a finitely presented group $\Gamma$ with a given presentation with $s$ generators and $r$ relations, the set $\mathrm{Hom}(\Gamma,G)$ is identified to a subset of  $G^s$ defined by the vanishing of $r$ conditions given by the relations. If we encode these conditions by $F\colon G^s\to G^r$, so as to identify $\mathrm{Hom}(\Gamma,G)$ with $F^{-1}(0)$, then it is known from 
\cite{Weil} that $Z^1_{\mathrm{Ad}\,\rho}(\Gamma,\mathfrak g)$ is naturally identified with the kernel of $dF$ at $\rho$. The isomorphism between these two vector spaces essentially associates to a germ of paths at $\rho$ represented by $t\mapsto\rho_t$ the cocycle $\tau$ defined by
\begin{equation}\label{eq:cocyle tangent}
\tau(\gamma)=\left.\frac{d}{dt}\right|_{t=0}\rho_t(\gamma)\rho(\gamma)^{-1}~,
\end{equation}
which is therefore interpreted as an infinitesimal deformation of $\rho$.

\begin{remark}\label{rmk:zariskiZ1}
In general, the Zariski tangent space at $\rho$ of the real variety associated to $\mathrm{Hom}(\Gamma,G)$ is only a subspace of $Z^1_{\mathrm{Ad}\,\rho}(\Gamma,\mathfrak g)$. We will show that in our situation for $\Gamma_{22}$ they coincide at the point $\rho_0$.
\end{remark}

Let us now look at the coboundaries. It was observed in \cite{Weil} (see also \cite[Lemma 2.2]{JM}) that the subspace of $\mathrm{Ker}(dF)$ corresponding to the tangent space to the $G^+$-orbit of $\rho$ identifies to $B^1_{\mathrm{Ad}\,\rho}(\Gamma,\mathfrak g)$ under the correspondence \eqref{eq:cocyle tangent}. Indeed, by a straightforward computation, the differential of the orbit map $G \to \mathrm{Hom}(\Gamma,G)$ defined by $ g\mapsto g\rho g^{-1}$ maps an element $X\in\mathfrak g$ to the coboundary $\tau(\gamma)=X-\mathrm{Ad}\rho(\gamma)X$. Observe that in our setting,  by Lemma \ref{lemma centraliser}, the action of $G$ is not free at $\rho$; but the action of the identity component of $G$, namely $G^+$, is indeed free. As we will see below (Lemma \ref{lem:dimB1}), this implies by a standard argument that  $B^1_{\mathrm{Ad}\,\rho_0}(\Gamma_{22},\mathfrak g) \cong \mathfrak{g}$.

We conclude by stating a smoothness criterion used by Weil \cite[Lemma 1]{Weil}, essentially consisting of an application of the implicit function theorem. We refer to \cite[Section 2.2]{KS} for more details.

Let $\mathcal{C}$ be an algebraic subset of $\mathrm{Hom}(\Gamma,G)$ containing $\rho$, say obtained by adding $k$ extra polynomial equations. We identify $\mathcal{C}$ with $\tilde{F}^{-1}(0)$, for some $\tilde{F} \colon G^s \to G^{r+k}$ compatible with $F$. Suppose that a neighbourhood of $\rho$ in $\mathcal{C}$ is a smooth submanifold of $G^s$ of the same dimension of the kernel $K$ of $d \tilde{F}$ at $\rho$. Then the following hold: at the point $\rho$, the real variety associated to $\mathcal{C}$ is smooth, its Zariski tangent space is isomorphic to $K$ (and not to a proper subspace, compare with Remark \ref{rmk:zariskiZ1}) and is naturally identified with the tangent space of $\mathcal{C}$ as a submanifold.

\subsection{The first cohomology group} \label{sec:zariski2}

Let us now go back to the representation $\rho_0\colon\Gamma_{22}\to G_0  = \mathrm{Stab}_G(\Hyp^3)$. In this subsection we analyse the vector space $H^1_{\mathrm{Ad}\,\rho_0}(\Gamma_{22},\mathfrak{g})$.

There is a well-known splitting
\begin{equation} \label{eq splitting lie lagebra}
\mathfrak g\cong \mathfrak{isom}(\Hyp^{n-1})\oplus\R^{n}~.
\end{equation}
When $G=\Isom(\Hyp^n)$ or $\Isom(\AdS^n)$, the splitting is given by writing an element $\mathfrak a$ of $\mathfrak g$ as
\begin{equation} \label{eq explicit splitting}
\mathfrak a=\left(
\begin{array}{ccc|c}
  &&& \vdots \\
  
   & \mathfrak a_0 & & \mp w \\
  &&& \vdots \\
    \hline  
    \ldots & w^TJ & \ldots  & 0
\end{array}
\right)~,
\end{equation}
where $J=\mathrm{diag}(-1,1,\ldots,1)$, for $\mathfrak a_0\in\so(1,n-1)$ and $w\in\R^{n}$. When $G=G_{\HP^n}$, the splitting \eqref{eq splitting lie lagebra} is even simpler to obtain, by using the isomorphism $G_{\HP^n}\cong \O(1,n-1)\ltimes\R^{1,n-1}$.

The splitting \eqref{eq splitting lie lagebra} is equivariant with respect to the three natural actions of $G_0$: the adjoint action on $\mathfrak g$, the adjoint action on $\mathfrak{isom}(\Hyp^3)$ by means of the isomorphism $G_0\cong \Isom(\Hyp^3)\times(\Z/2\Z)$, and the action on $\R^{1,3}$ by means of the isomorphism $G_0\cong \O(1,3)$ of \eqref{eq:isomorphism G0}. We thus have natural decompositions

\begin{equation} \label{eq splitting Z1}
Z^1_{\mathrm{Ad}\,\rho_0}(\Gamma_{22},\mathfrak g)=Z^1_{\mathrm{Ad}\,\varrho_0}(\Gamma_{22},\mathfrak{isom}(\Hyp^3))\oplus Z^1_{\varrho_0}(\Gamma_{22},\R^{1,3})~.
\end{equation}
and
\begin{equation} \label{eq splitting B1}
B^1_{\mathrm{Ad}\,\rho_0}(\Gamma_{22},\mathfrak g)=B^1_{\mathrm{Ad}\,\varrho_0}(\Gamma_{22},\mathfrak{isom}(\Hyp^3))\oplus B^1_{\varrho_0}(\Gamma_{22},\R^{1,3})~,
\end{equation}
hence
\begin{equation} \label{eq splitting H1}
H^1_{\mathrm{Ad}\,\rho_0}(\Gamma_{22},\mathfrak g)=H^1_{\mathrm{Ad}\,\varrho_0}(\Gamma_{22},\mathfrak{isom}(\Hyp^3))\oplus H^1_{\varrho_0}(\Gamma_{22},\R^{1,3})~.
\end{equation}

Recall that $\rho_0$ coincides with the composition of {the representation} $\varrho_0 \colon \Gamma_{22} \to G_0$ {defined in \eqref{eq:rep rho0}}  with the inclusion $G_0 \to G$. 

{Let us look at the first factor of the decomposition \eqref{eq splitting H1} of $H^1_{\mathrm{Ad}\,\rho_0}(\Gamma_{22},\mathfrak g)$.}

\begin{prop} \label{prop:12d}
The vector space $H^1_{\mathrm{Ad}\,\varrho_0}(\Gamma_{22},\mathfrak{isom}(\Hyp^3))$ has dimension $12$. 
\end{prop}

\begin{proof}
We claim that the group $H^1_{\mathrm{Ad}\,\varrho_0}(\Gamma_{22},\mathfrak{isom}(\Hyp^3))$ is isomorphic to $H^1_{\mathrm{Ad}\,\iota}(\Gamma_{\mathrm{co}},\mathfrak{isom}(\Hyp^3))$, where $\Gamma_{\mathrm{co}}$ is the reflection group of the right-angled cuboctahedron and $\iota$ is its inclusion into $\Isom(\Hyp^3)$. The latter has dimension 12, since the character variety of $\Gamma_{\mathrm{co}}$ in $\Isom(\Hyp^3)$ is smooth and 12-dimensional near $[\iota]$. We have already mentioned (in Section \ref{sec:proof teoB ads}, Step 2) that this last fact is true by ``reflective hyperbolic Dehn filling''.

To show the claim, {we will show that the restriction morphism $H^1_{\mathrm{Ad}\,\varrho_0}(\Gamma_{22},\mathfrak{isom}(\Hyp^3))\to H^1_{\mathrm{Ad}\,\iota}(\Gamma_{\mathrm{co}},\mathfrak{isom}(\Hyp^3))$ is invertible by explicitly constructing an inverse.} Define a map
$$\psi\colon Z^1_{\mathrm{Ad}\,\iota}(\Gamma_{\mathrm{co}},\mathfrak{isom}(\Hyp^3))\to Z^1_{\mathrm{Ad}\,\varrho_0}(\Gamma_{22},\mathfrak{isom}(\Hyp^3))~,$$
identifying $\Gamma_{\mathrm{co}}$ with the subgroup of $\Gamma_{22}$ generated by $\m0,\ldots\m7,\l A,\ldots,\l F$, by
$$\psi(\tau)(\l X)=\tau(\l X)\qquad \psi(\tau)(\m i)=\tau(\m i)\qquad \psi(\tau)(\p i)=0~,$$
for every $\tau\in Z^1_{\mathrm{Ad}\,\iota}(\Gamma_{\mathrm{co}},\mathfrak{isom}(\Hyp^3))$. {We recall from Definition \ref{defi HP holonomy} that $\varrho_0$ satisfies:
$$\varrho_0(\l X)=\iota(\l X)\qquad \varrho_0(\m i)=\iota(\m i)\qquad \varrho_0(\p i)=-\mathrm{id}~.$$
It follows that 
\begin{equation}\label{eq:adjointrho0}
\mathrm{Ad}\,\varrho_0(\p i)=\mathrm{id}~.
\end{equation}}

{Let us now show that the map $\psi$ is well-defined, and induces a well-defined map in cohomology. First, using Lemma \ref{lemma cocycle generators} and \eqref{eq:adjointrho0}, one checks easily that} $\psi(\tau)(\p i\l X)=\psi(\tau)(\l X\p i)$ and $\psi(\tau)(\p i\m j)=\psi(\tau)(\m j\p i)$. Hence $\psi$ maps cocycles to cocycles. Moreover, it maps coboundaries to coboundaries, for if $\tau(\l s)=\mathrm{Ad}\,\iota(\l s)\mathfrak a-\mathfrak a$ for some $\mathfrak a\in\mathfrak{isom}(\Hyp^3)$, then of course $\psi(\tau)(\l s)=\mathrm{Ad}\,\varrho_0(\l s)\mathfrak a-\mathfrak a$ for $\l s=\m i$ or $\l X$. As $\mathrm{Ad}\,\varrho_0(\p i)=\mathrm{id}$ by \eqref{eq:adjointrho0}, the identity holds trivially also for $\l s=\p i$.

Thus $\psi$ induces a map 
$$\psi\colon H^1_{\mathrm{Ad}\,\iota}(\Gamma_{\mathrm{co}},\mathfrak{isom}(\Hyp^3))\to H^1_{\mathrm{Ad}\,\varrho_0}(\Gamma_{22},\mathfrak{isom}(\Hyp^3))~.$$
{We claim that it is a vector space isomorphism. To see it is injective, suppose $\psi(\tau)$ is a coboundary, namely $\psi(\tau)(\l s)=\mathrm{Ad}\,\varrho_0(\l s)\mathfrak a-\mathfrak a$ for all generators $\l s$. From \eqref{eq:adjointrho0}, we have $\psi(\tau)(\p i)=0$, and from the definition of $\varrho_0$, $\tau(\l s)=\mathrm{Ad}\,\iota(\l s)\mathfrak a-\mathfrak a$ for $\l s=\m i$ or $\l X$. Hence $\tau$ is a coboundary, which concludes injectivity}. It remains to show that it is surjective. To see this, given any cocycle $\sigma\in Z^1_{\mathrm{Ad}\,\varrho_0}(\Gamma_{22},\mathfrak{isom}(\Hyp^3))$, Lemma \ref{lemma cocycle generators} implies that $\sigma(\p i)$  is in the kernel of $\mathrm{id}+\mathrm{Ad}\,\varrho(\p i)$, which in fact equals $2\mathrm{id}$ by \eqref{eq:adjointrho0}. Hence $\sigma(\p i)=0$ and $\sigma$ is in the image of $\psi$. This concludes the proof. 
\end{proof}


The second factor of the decomposition \eqref{eq splitting H1} of $H^1_{\mathrm{Ad}\,\rho_0}(\Gamma_{22},\mathfrak{g})$ has already been computed in Proposition \ref{prop:first cohomology group}. We have in particular: 
\begin{cor} \label{cor:13d}
The vector space $H^1_{\mathrm{Ad}\,\rho_0}(\Gamma_{22},\mathfrak{g})$ has dimension $13$.
\end{cor}

%
%
%
%

We deduce the dimensions of the spaces of cocycles and coboundaries from the following simple lemma, which will also be used in the next section.
\begin{lemma}\label{lem:dimB1}
The orbit map $G^+ \to \mathrm{Hom}(\Gamma_{22},G)$ is an embedding.
\end{lemma}

\begin{proof}
The orbit map is injective as a consequence that the $G^+$-action is free (Lemma \ref{lemma centraliser}).
To see that its differential at any point is injective, suppose by contradiction that a non-zero vector is in the kernel of the differential. Acting by left multiplication on $G^+$, one then finds a nonvanishing vector field on $G^+$ which is, at any point, in the kernel of the differential. Hence the orbit map would be constant on any integral path of this vector field, thus contradicting injectivity.
\end{proof}

\begin{cor} \label{cor:23d}
The vector spaces $Z^1_{\mathrm{Ad}\,\rho_0}(\Gamma_{22},\mathfrak{g})$ and $B^1_{\mathrm{Ad}\,\rho_0}(\Gamma_{22},\mathfrak{g})$ have dimension $23$ and $10$, respectively.
\end{cor}

\begin{proof}
By Lemma \ref{lem:dimB1}, $B^1_{\mathrm{Ad}\,\rho_0}(\Gamma_{22},\mathfrak{g})$, which is the tangent space to the orbit, is isomorphic to $\mathfrak g$ and thus has dimension 10. Combining this with Corollary \ref{cor:13d}, we conclude that $Z^1_{\mathrm{Ad}\,\rho_0}(\Gamma_{22},\mathfrak{g})$ has dimension 10.
\end{proof}

\begin{remark}
One could also check directly that $\dim B^1_{\mathrm{Ad}\,\rho_0}(\Gamma_{22},\mathfrak{g})=10$: from Remark \ref{rmk:dimB1} the second factor  in the decomposition \eqref{eq splitting B1} has dimension 4, and by a similar argument one can prove that the first factor has dimension 6.
\end{remark}

\subsection{Conclusion of the proofs}
We can now prove the main result of the section.

Recall that $G$ is $\Isom(\Hyp^4)$, $\Isom(\AdS^4)$ or $G_{\HP^4}$, and $\mathfrak{g}$ is its Lie algebra.
We denote as usual by $\widetilde {\mathcal U}$, $\widetilde {\mathcal V}$, $\widetilde {\mathcal H} \subset \mathrm{Hom}(\Gamma_{22},G)$ the preimages of $\mathcal U$, $\mathcal{V}$, $\mathcal{H} \subset X(\Gamma_{22},G)$, respectively. All these sets are defined and studied in the proof of Theorem \ref{teo:main_weak} in Sections \ref{sec:proof teoB ads} and \ref{sec: teo B HP}. Recall that $\mathcal{U} = \mathcal{V} \cup \mathcal{H}$ and $\mathcal{V} \cap \mathcal{H} = \{ [\rho_0] \}$. 

For brevity, in the following, given a real affine algebraic set 
$\mathcal{S}$, by ``Zariski tangent space to'' (resp. ``component of'') 
$\mathcal{S}$ we refer to the Zariski tangent space to (resp. a component of) the real variety associated to 
$\mathcal{S}$.

\begin{theorem}\label{teo:main_weak2}
The sets $\widetilde{\mathcal{V}}$ and $\widetilde{\mathcal{H}}$ are smooth components of $\widetilde{\mathcal{U}}$, of dimension $11$ and $22$, respectively. Moreover, $\widetilde{\mathcal{V}} \cap \widetilde{\mathcal{H}}$ is the $G$-orbit of $\rho_0$, and the Zariski tangent spaces of $\widetilde{\mathcal{V}}$ and $\widetilde{\mathcal{H}}$ at $\rho_0$ intersect transversely in the Zariski tangent space of $\mathrm{Hom}(\Gamma_{22},G)$ at $\rho_0$.
\end{theorem}


\begin{proof}

%

We first show that $\widetilde{\mathcal{V}}$ is a smooth component 
of $\widetilde{\mathcal{U}}$ by applying the smoothness criterion stated 
at the end of Section \ref{sec:zariski}.
Note that $\widetilde{\mathcal{V}}$ 
is the intersection of $\widetilde{\mathcal{U}}$ with the algebraic subset $\mathrm{Hom}_0(\Gamma_{22},G)$ of $\mathrm{Hom}(\Gamma_{22},G)$ (see 
Definition \ref{defi Hom0} and the discussion below). In particular $\widetilde{\mathcal{V}}$ is a neighbourhood of $\rho_0$ in $\mathrm{Hom}_0(\Gamma_{22},G)$.

Moreover, 
$\widetilde {\mathcal V}$ is a smooth 11-dimensional manifold. 
This is true in the hyperbolic or AdS case by Proposition \ref{prop:dimension tangent space ads}, and in the HP case since there $\widetilde {\mathcal V}$ is the total space of a smooth vector bundle with fibre $Z^1_{\varrho_0}(\Gamma_{22},\R^{1,3})$ and base the $\Isom(\Hyp^3)$-orbit of $\varrho_0$ (hence a rank-5 bundle over a 6-manifold). Recall indeed that $Z^1_{\varrho_0}(\Gamma_{22},\R^{1,3})$ has dimension 5 by Proposition \ref{prop:first cohomology group} and Remark \ref{rmk:dimB1}.

The tangent space $T_{\rho_0}\widetilde{\mathcal{V}}$ to the smooth manifold $\widetilde{\mathcal{V}}$ is contained in $B^1_{\mathrm{Ad}\,\varrho_0}(\Gamma_{22},\mathfrak{isom}(\Hyp^3))\oplus Z^1_{\varrho_0}(\Gamma_{22},\R^{1,3})$ under the identification \eqref{eq:cocyle tangent} and the splitting \eqref{eq splitting Z1}. Indeed
, in the hyperbolic and AdS case, by a direct computation one can see that the derivative $\ddt \rho_t\circ\rho_0^{-1}$ is in the second factor in the decomposition \eqref{eq splitting Z1}, and gives a nonzero cocycle of the form \eqref{eq:cocycle tau} in $Z^1_{\varrho_0}(\Gamma_{22},\R^{1,3})$ whose class generates $H^1_{\varrho_0}(\Gamma_{22},\R^{1,3})$ (see \eqref{eq:cocycle tau2}). 
In the half-pipe case, the HP representations with linear part $\varrho_0$ are themselves identified to the vector space $Z^1_{\varrho_0}(\Gamma_{22},\R^{1,3})$, and their derivatives are in $Z^1_{\varrho_0}(\Gamma_{22},\R^{1,3})$ itself, seen as the second factor of \eqref{eq splitting Z1}.

As a consequence of Proposition \ref{prop:first cohomology group} and Corollary \ref{cor:23d}, $B^1_{\mathrm{Ad}\,\varrho_0}(\Gamma_{22},\mathfrak{isom}(\Hyp^3))\oplus Z^1_{\varrho_0}(\Gamma_{22},\R^{1,3})$ has dimension 11. Since $T_{\rho_0}\widetilde{\mathcal{V}}$ and $B^1_{\mathrm{Ad}\,\varrho_0}(\Gamma_{22},\mathfrak{isom}(\Hyp^3))\oplus Z^1_{\varrho_0}(\Gamma_{22},\R^{1,3})$ 
have the same dimension, they are equal, and we conclude that $\widetilde{\mathcal{V}}$ is a smooth component of $\widetilde{\mathcal{U}}$
by the smoothness criterion.

Let us now look at $\widetilde{\mathcal{H}}$. Note that it is the intersection of $\widetilde{\mathcal{U}}$ with an algebraic subset of $\mathrm{Hom}(\Gamma_{22},G)$. Combining our study of $\mathcal{H}$ in Sections \ref{sec:proof teoB ads} and \ref{sec: teo B HP} with Lemma \ref{lem:dimB1}, we get that $\widetilde{\mathcal{H}}$ is a smooth $22$-dimensional manifold. Moreover, 
 $T_{\rho_0}\widetilde{\mathcal{H}}$ 
is contained in $Z^1_{\mathrm{Ad}\,\varrho_0}(\Gamma_{22},\mathfrak{isom}(\Hyp^3))\oplus B^1_{\varrho_0}(\Gamma_{22},\R^{1,3})$ under \eqref{eq:cocyle tangent} and \eqref{eq splitting Z1}.

The proof of Proposition \ref{prop:12d} shows that the first factor is isomorphic to $Z^1_{\mathrm{Ad}\,\iota}(\Gamma_{\mathrm{co}},\mathfrak{isom}(\Hyp^3))$. The latter is the space of infinitesimal deformations of the right-angled cuboctahedron, and has dimension $18 = \sharp \, \mathrm{cusps} + \dim(\Isom(\Hyp^3))$ for general and well-known reasons around 3-dimensional hyperbolic Dehn filling (namely, roughly speaking, ``half'' of the infinitesimal deformations of the peripheral subgroups extend to the whole group; see the discussion in \cite[Section 5]{KS}). One can even compute this number directly as indicated in Section \ref{sec:zariski} by means of Table \ref{table:vectors cocycle}.

Since the second factor $B^1_{\varrho_0}(\Gamma_{22},\R^{1,3})$ has dimension 4 by Remark \ref{rmk:dimB1} (or Lemma \ref{lem:dimB1}), we conclude as similarly done for $\widetilde{\mathcal{V}}$ that $\widetilde{\mathcal{H}}$ is a smooth 22-dimensional component of $\widetilde{\mathcal{U}}$.

By the previous considerations, the integrable vectors arising from $\widetilde{\mathcal{V}}$ and $\widetilde{\mathcal{H}}$ generate $Z^1_{\mathrm{Ad}\,\rho_0}(\Gamma_{22},\mathfrak{g})$, 
and the Zariski tangent space at $\rho_0$ to $\mathrm{Hom}(\Gamma_{22},G)$ is the whole $Z^1_{\mathrm{Ad}\,\rho_0}(\Gamma_{22},\mathfrak{g})$ (recall Remark \ref{rmk:zariskiZ1}). 
Moreover $T_{\rho_0}\widetilde{\mathcal{V}}$ and $T_{\rho_0}\widetilde{\mathcal{H}}$, which are identified to the corresponding Zariski tangent spaces, are transverse in $Z^1_{\mathrm{Ad}\,\rho_0}(\Gamma_{22},\mathfrak{g})$. 
Finally, $\widetilde{\mathcal{V}} \cap \widetilde{\mathcal{H}}$ is the $G$-orbit of $\rho_0$ 
thanks to Theorem \ref{teo:main_weak}.
\end{proof}

The proof of Theorem \ref{teo:main} is complete combining Theorems \ref{teo:main_weak} and \ref{teo:main_weak2}.

\begin{remark} \label{rem:integrable}
We have shown in the proof of Theorem \ref{teo:main_weak2} that the Zariski tangent space at $\rho_0$ to $\mathrm{Hom}(\Gamma_{22},G)$ is isomorphic to the whole $Z^1_{\mathrm{Ad}\,\rho_0}(\Gamma_{22},\mathfrak{g})$
. Moreover, integrable vectors to $\widetilde{\mathcal{U}}$ at $\rho_0$ correspond precisely to those in
$$B^1_{\mathrm{Ad}\,\varrho_0}(\Gamma_{22},\mathfrak{isom}(\Hyp^3))\oplus Z^1_{\varrho_0}(\Gamma_{22},\R^{1,3})\ \cup\ Z^1_{\mathrm{Ad}\,\varrho_0}(\Gamma_{22},\mathfrak{isom}(\Hyp^3))\oplus B^1_{\varrho_0}(\Gamma_{22},\R^{1,3})~,$$
since $\widetilde{\mathcal{V}}$ and $\widetilde{\mathcal{H}}$ are the only components of $\widetilde{\mathcal{U}}$ by Theorem \ref{teo:main_weak}.
\end{remark}

\begin{remark} \label{rem:transverse}
Without entering into details on the semialgebraic structure of $X(\Gamma_ {22},G)$ when $G$ is reductive, our analysis implies that, roughly speaking, $\mathcal{V}$ and $\mathcal{H}$ can be thought as ``smooth and transverse components'' of $\mathcal{U}$ in a reasonable and satisfactory sense. See in particular Propositions \ref{prop:first cohomology group} and \ref{prop:12d} and the splitting \eqref{eq splitting H1}, and compare with the topological description of $\mathcal{U}$ in Theorem \ref{teo:main_weak}, the algebraic description of $\widetilde{\mathcal{U}}$ in Theorem \ref{teo:main_weak2}, and the differential geometric one in Remark \ref{rem:integrable}.
\end{remark}


%
%
%

\bibliographystyle{alpha}
\bibliography{sr-bibliography}

\newcommand{\etalchar}[1]{$^{#1}$}
\begin{thebibliography}{BBD{\etalchar{+}}12}

\bibitem[AP15]{AP}
Norbert A'Campo and Athanase Papadopoulos.
\newblock Transitional geometry.
\newblock {\em Sophus Lie and Felix Klein: the Erlangen program and its impact
  in mathematics and physics}, 23:217, 2015.

\bibitem[BBD{\etalchar{+}}12]{questionsads}
Thierry Barbot, Francesco Bonsante, Jeffrey Danciger, William~M. Goldman,
  Fran{\c{c}}ois Gu\'eritaud, Fanny Kassel, Kirill Krasnov, Jean-Marc
  Schlenker, and Abdelghani Zeghib.
\newblock Some open questions on anti-de {S}itter geometry.
\newblock 2012.
\newblock \href{https://arxiv.org/abs/1205.6103}{\texttt{arXiv:1205.6103}}.

\bibitem[BF20]{barbotfillastre}
Thierry Barbot and Fran{\c{c}}ois Fillastre.
\newblock Quasi-{F}uchsian co-{M}inkowski manifolds.
\newblock In {\em In the Tradition of Thurston}, pages 645--703. Springer,
  2020.

\bibitem[BLP05]{BLP}
Michel Boileau, Bernhard Leeb, and Joan Porti.
\newblock Geometrization of 3-dimensional orbifolds.
\newblock {\em Ann. of Math. (2)}, 162(1):195--290, 2005.

\bibitem[CDW18]{CDW}
Daryl Cooper, Jeffrey Danciger, and Anna Wienhard.
\newblock Limits of geometries.
\newblock {\em Trans. Amer. Math. Soc.}, 370:6585--6627, 2018.

\bibitem[CHK00]{CHK}
Daryl Cooper, Craig~D. Hodgson, and Steven~P. Kerckhoff.
\newblock {\em Three-dimensional orbifolds and cone-manifolds}, volume~5 of
  {\em MSJ Memoirs}.
\newblock Mathematical Society of Japan, Tokyo, 2000.
\newblock With a postface by Sadayoshi Kojima.

\bibitem[CLM18]{CLMsurvey}
Suhyoung Choi, Gye-Seon Lee, and Ludovic Marquis.
\newblock Deformations of convex real projective manifolds and orbifolds.
\newblock In {\em Handbook of group actions. {V}ol. {III}}, volume~40 of {\em
  Adv. Lect. Math. (ALM)}, pages 263--310. Int. Press, Somerville, MA, 2018.

\bibitem[Dan11]{danciger}
Jeffrey Danciger.
\newblock {\em Geometric transition: from hyperbolic to AdS geometry}.
\newblock PhD thesis, Stanford University, 2011.

\bibitem[Dan13]{dancigertransition}
Jeffrey Danciger.
\newblock A geometric transition from hyperbolic to anti-de {S}itter geometry.
\newblock {\em Geom. Topol.}, 17(5):3077--3134, 2013.

\bibitem[Dan14]{dancigerideal}
Jeffrey Danciger.
\newblock Ideal triangulations and geometric transitions.
\newblock {\em J. Topol.}, 7(4):1118--1154, 2014.

\bibitem[FS19]{surveyseppifillastre}
Fran\c{c}ois Fillastre and Andrea Seppi.
\newblock Spherical, hyperbolic, and other projective geometries: convexity,
  duality, transitions.
\newblock In {\em Eighteen essays in non-{E}uclidean geometry}, volume~29 of
  {\em IRMA Lect. Math. Theor. Phys.}, pages 321--409. Eur. Math. Soc.,
  Z\"{u}rich, 2019.

\bibitem[Hod86]{Hthesis}
Craig~David Hodgson.
\newblock {\em Degeneration and regeneration of hyperbolic structures on
  three-manifolds (foliations, {D}ehn surgery)}.
\newblock ProQuest LLC, Ann Arbor, MI, 1986.
\newblock Thesis (Ph.D.)--Princeton University.

\bibitem[HPS01]{P01}
Michael Heusener, Joan Porti, and Eva Su\'{a}rez.
\newblock Regenerating singular hyperbolic structures from {S}ol.
\newblock {\em J. Differential Geom.}, 59(3):439--478, 2001.

\bibitem[JM87]{JM}
Dennis {Johnson} and John~J. {Millson}.
\newblock {Deformation spaces associated to compact hyperbolic manifolds}.
\newblock {Discrete groups in geometry and analysis, Pap. Hon. G. D. Mostow
  60th Birthday, Prog. Math. 67, 48-106 (1987).}, 1987.

\bibitem[Koz13]{Kozai_thesis}
Kenji Kozai.
\newblock {\em Singular hyperbolic structures on pseudo-{A}nosov mapping tori}.
\newblock PhD thesis, Stanford University, 2013.

\bibitem[Koz16]{Kozai}
Kenji Kozai.
\newblock Hyperbolic structures from {S}ol on pseudo-{A}nosov mapping tori.
\newblock {\em Geom. Topol.}, 20(1):437--468, 2016.

\bibitem[KS10]{KS}
Steven~P. Kerckhoff and Peter~A. Storm.
\newblock From the hyperbolic 24-cell to the cuboctahedron.
\newblock {\em Geom. Topol.}, 14(3):1383--1477, 2010.

\bibitem[LMA15a]{LM2}
Mar{\'\i}a~Teresa Lozano and Jos{\'e}~Mar{\'\i}a Montesinos-Amilibia.
\newblock Geometric conemanifold structures on $\mathbb{T}_{p/q}$, the result
  of $p/q$ surgery in the left-handed trefoil knot $\mathbb{T}$.
\newblock {\em J. Knot Theory Ramifications}, 24(12):1550057, 2015.

\bibitem[LMA15b]{LM1}
Mar{\'\i}a~Teresa Lozano and Jos{\'e}~Mar{\'\i}a Montesinos-Amilibia.
\newblock On the degeneration of some 3-manifold geometries via unit groups of
  quaternion algebras.
\newblock {\em Rev. R. Acad. Cienc. Exactas F{\'\i}s. Nat. Ser. A Mat. RACSAM},
  109(2):669--715, 2015.

\bibitem[MR18]{MR}
Bruno Martelli and Stefano Riolo.
\newblock Hyperbolic {D}ehn filling in dimension four.
\newblock {\em Geom. Topol.}, 22(3):1647--1716, 2018.

\bibitem[Por98]{P98}
Joan Porti.
\newblock Regenerating hyperbolic and spherical cone structures from
  {E}uclidean ones.
\newblock {\em Topology}, 37(2):365--392, 1998.

\bibitem[Por02]{P02}
Joan Porti.
\newblock Regenerating hyperbolic cone structures from {N}il.
\newblock {\em Geom. Topol.}, 6:815--852, 2002.

\bibitem[Por13]{P13}
Joan Porti.
\newblock Regenerating hyperbolic cone 3-manifolds from dimension 2.
\newblock {\em Ann. Inst. Fourier (Grenoble)}, 63(5):1971--2015, 2013.

\bibitem[PW07]{P07}
Joan Porti and Hartmut Weiss.
\newblock Deforming {E}uclidean cone 3-manifolds.
\newblock {\em Geom. Topol.}, 11:1507--1538, 2007.

\bibitem[RS]{transition_4-manifold}
Stefano Riolo and Andrea Seppi.
\newblock Geometric transition from hyperbolic to {A}nti-de {S}itter structures
  in dimension four.
\newblock {\em Ann. Sc. Norm. Super. Pisa Cl. Sci.}
\newblock To appear (DOI: \texttt{10.2422/2036-2145.202005\_031}),
  \href{https://arxiv.org/abs/1908.05112}{\texttt{arXiv:1908.05112}}.

\bibitem[RS90]{GIT}
Roger~W. Richardson and Peter~J. Slodowy.
\newblock Minimum vectors for real reductive algebraic groups.
\newblock {\em J. London Math. Soc. (2)}, 42(3):409--429, 1990.

\bibitem[Sep19]{TSG}
Andrea Seppi.
\newblock Examples of geometric transition from dimension two to four.
\newblock {\em {A}ctes du s{\'e}minaire {T}h{\'e}orie {S}pectrale et
  {G}{\'e}om{\'e}trie}, 35:163--196, 2017-2019.

\bibitem[Ser05]{MR2140265}
Caroline Series.
\newblock Limits of quasi-{F}uchsian groups with small bending.
\newblock {\em Duke Math. J.}, 128(2):285--329, 2005.

\bibitem[Thu79]{thurstonnotes}
William~P. Thurston.
\newblock The geometry and topology of three-manifolds.
\newblock Electronic version 1.1,
  \href{http://library.msri.org/nonmsri/gt3m}{\texttt{http://library.msri.org/nonmsri/gt3m}},
  1979.

\bibitem[Tre19]{trettel_thesis}
Steve~J. Trettel.
\newblock {\em Families of geometries, real algebras, and transitions}.
\newblock PhD thesis, University of {C}alifornia, {S}anta {B}arbara, 2019.

\bibitem[Wei64]{Weil}
Andr{\'e} Weil.
\newblock Remarks on the cohomology of groups.
\newblock {\em Ann. of Math. (2)}, 80:149--157, 1964.

\end{thebibliography}

\end{document}